\documentclass[12pt]{article}

\usepackage{amsfonts}
\usepackage{amssymb}
\usepackage{amsthm}
\usepackage{amsmath}
\usepackage{amscd}
\usepackage{mathrsfs}
\usepackage{enumerate}
\usepackage{ulem}
\usepackage{fontenc}
\usepackage[all]{xy}
\usepackage{array}
\usepackage[english]{babel}
\usepackage{float}
\usepackage{longtable}
\usepackage[ruled,vlined]{algorithm2e}
\usepackage{bbm}
\usepackage{bm}

\def\CC {{\mathbb C}}     
\def\NN {{\mathbb N}}     
\def\TT {{\mathbb T}}     
\def\VV {{\mathbb V}}     
\def\ZZ {{\mathbb Z}}     

\def\B  {\mathcal B}

\def\lw  {\longrightarrow}
\def\mc {\mathcal}
\def\mk {\mathfrak}

\def\tst {\Longleftrightarrow}
\def\ul  {\underline}

\newtheorem{theorem}{Theorem}[section]
\newtheorem{lemma}[theorem]{Lemma}

\newtheorem{prop}[theorem]{Proposition}

\newtheorem{coro}[theorem]{Corollary}
\newtheorem{rem}{Remark}[section]
\newtheorem{definition}{Definition}[section]

\newtheorem{example}{Example}[section]

\begin{document}
\title{Representations of large Mackey Lie algebras and universal tensor categories}
\author{Ivan Penkov, Valdemar Tsanov}

\maketitle


\begin{abstract}
We extend previous work by constructing a universal abelian tensor category ${\bf T}_t$ generated by two objects $X,Y$ equipped with finite filtrations $0\subsetneq X_0\subsetneq ... X_{t+1}\subsetneq X$ and $0\subsetneq Y_0\subsetneq ... Y_{t+1}\subsetneq Y$, and with a pairing $X\otimes Y\to \mathbbm{1}$, where $\mathbbm{1}$ is the monoidal unit. This category is modeled as a category of representations of a Mackey Lie algebra $\mk{gl}^M(V,V_*)$ of cardinality $2^{\aleph_t}$, associated to a diagonalizable pairing between two complex vector spaces $V,V_*$ of dimension $\aleph_t$. As a preliminary step, we study a tensor category $\TT_t$ generated by the algebraic duals $V^*$, $(V_*)^*$. The injective hull of $\CC$ in $\TT_t$ is a commutative algebra $I$, and the category ${\bf T}_t$ is consists of the free $I$-modules in $\TT_t$. An essential novelty in our work is the explicit computation of Ext-groups between simples in both categories ${\bf T}_t$ and $\TT_t$, which had been an open problem already for $t=0$. This provides a direct link from the theory of universal tensor categories to Littlewood-Richardson-type combinatorics.
\end{abstract}


\section{Introduction}

For us, a {\it tensor category} is a $\CC$-linear, not necessarily rigid, symmetric monoidal abelian category. In this paper we construct a tensor category ${\bf T}_t$, generated by two objects $X$ and $Y$, equipped with finite filtrations $0\subsetneq X_0\subsetneq ... X_{t+1}\subsetneq X$ and $0\subsetneq Y_0\subsetneq ... Y_{t+1}\subsetneq Y$, and with a pairing $X\otimes Y\to \mathbbm{1}$ where $\mathbbm{1}$ is the monoidal unit, such that the category ${\bf T}_t$ is universal in the following sense: for every other tensor category equipped with objects $X',Y'$, a morphism $X'\otimes Y'\to\mathbbm{1}'$, and finite filtrations $0\subsetneq X'_0\subsetneq ... X'_{t'+1}\subsetneq X'$ and $0\subsetneq Y'_0\subsetneq ... Y'_{t'+1}\subsetneq Y'$ with $t'\leq t$, there is a left exact monoidal functor from the category ${\bf T}_t$ to this other category such that
$$
F(X)=X'\;,\; F(Y)=Y'\;,\; F(X_\alpha)=X'_{s(\alpha)} \;,\; F(Y_\alpha)=Y'_{s(\alpha)} \;,
$$
for some order preserving surjection $s:\{0,...,t+1\}\to\{0,...,t'+1\}$. 

Our work extends several previous works \cite{Penkov-Serganova-Mackey}, \cite{Sam-Snowden}, \cite{Chirvasitu-Penkov-OTC}, \cite{Chirvasitu-Penkov-RC}, \cite{Chirvasitu-Penkov-UTC}. The most recent of them is the paper \cite{Chirvasitu-Penkov-UTC} where the filtrations of $X$ and $Y$ are just of length two, i.e., amount to fixed subobjects $X_0\subset X$ and $Y_0\subset Y$. This case has many features of the general case, and we follow the main idea of \cite{Chirvasitu-Penkov-UTC}. Namely, we first construct a tensor category $\TT_{t}$ which consists of tensor modules over the Mackey Lie algebra $\mk{gl}^M=\mk{gl}^M_{t+1}$ of a split pairing ${\bf p}:V\otimes V_*\to\CC$. Here $V$ is a complex vector space of dimension $\aleph_{t}$ with $t\in\NN$, and $V_*$ is the span within $V^*:={\rm Hom}(V,\CC)$ of a system of vectors $\{x_b\}$ dual to a basis $\{v_b\}$ of $V$. The Lie algebra $\mk{gl}^M$ consists of all linear operators $\varphi:V\to V$ such that $\varphi^*(V_*)\subset V_*$, where $\varphi^*$ stands for the dual operator. We recall that the $\mk{gl}^M$-modules $V^*$ and $\bar V:=(V_*)^*={\rm Hom}(V_*,\CC)$ have finite filtrations $V_*=V^*_{0}\subsetneq ... \subsetneq V^*_{t+1} = V^*$ and $V=\bar V_{0}\subsetneq ... \subsetneq \bar V_{t+1}=\bar V$ with irreducible successive quotients. Using these filtrations we compute the socle and radical filtrations of the adjoint $\mk{gl}^M$-module, and also describe all ideals of the Lie algebra $\mk{gl}^M$. The latter result is not necessarily needed for our study of the category ${\bf T}_{t}$ and is of interest on its own.

The category $\TT_{t}$ is defined as the full tensor subcategory of the category of $\mk{gl}^M$-modules, generated by the two modules $V^*$ and $\bar V$, and closed under arbitrary direct sums. This category is not yet our desired universal tensor category, but is a natural and interesting tensor category. We classify the simple objects in $\TT_{t}$. It turns out that they are parametrized by pairs $\lambda_\bullet,\mu_\bullet$ where $\lambda_\bullet$ and $\mu_\bullet$ are finite sequences of length $t+2$ with elements arbitrary Young diagrams. We then describe the indecomposable injective objects in $\TT_{t}$ (equivalently, the injective hulls of the simple objects) and compute explicitly the layers of their socle filtrations. The simple objects of $\TT_{t}$ have infinite injective length and the injective hull $I$ of the trivial 1-dimensional $\mk{gl}^M$-module $\CC$ plays a special role. In particular, the $\mk{gl}^M$-module $I$ has also the structure of a commutative associative algebra.

An essential novelty going beyond the ideas of \cite{Chirvasitu-Penkov-UTC} is that we write down an explicit injective resolution of any simple object, and hence obtain explicit formulas for all Exts between simple modules in $\TT_{t}$.

Finally, following again \cite{Chirvasitu-Penkov-UTC}, we define the desired universal category ${\bf T}_{t}$. This is the category of $(\mk{gl}^M,I)$-modules, whose objects are the objects of $\TT_{t}$ which are free as $I$-modules (in particular, $I\in {\bf T}_{t}$) and whose morphisms are morphisms of $\mk{gl}^M$-modules as well as of $I$-modules. The tensor product in ${\bf T}_{t}$ is $\otimes_I$ and the simple objects in the new category are nothing but simple objects of $\TT_{t}$ tensored by $I$. These new simple objects have finite injective length in ${\bf T}_t$. Moreover, as an object of ${\bf T}_{t}$ the module $I$ is both simple and injective. We compute explicitly all Exts between simple objects in ${\bf T}_{t}$ by writing down canonical injective resolutions of simples. In the case of ${\bf T}_{0}$ studied in \cite{Chirvasitu-Penkov-UTC}, this yields a new formula for the dimension of ${\rm Ext}_{{\bf T}_0}^{q}(I\otimes L_{\kappa_1,\kappa_0;\nu_0,\nu_1},I\otimes L_{\lambda_1,\lambda_0;\mu_0,\mu_1})$ as the multiplicity of $I\otimes L_{\kappa_1^\perp,\kappa_0;\nu_0^\perp,\nu_1}$ in the $q$-th layer of the socle filtration of the injective hull of the module $I\otimes L_{\lambda_1^\perp,\lambda_0;\mu_0^\perp,\mu_1}$, where $L_{\kappa_1,\kappa_0;\nu_0,\nu_1},L_{\lambda_1,\lambda_0;\mu_0,\mu_1}$ are arbitrary simple objects in $\TT_{0}$ and ${^\perp}$ stands for conjugate Young diagram.

A brief outline of the contents is as follows. In \S2 we define Mackey Lie algebras and determine their ideals. In \S3 we introduce the module $I$. In \S4 we collect necessary notions from category theory. In \S5 and \S6, which contain the technical bulk of the paper, we study the categories $\TT_t$ and ${\bf T}_t$, respectively. We exhibit some unexpected combinatorial symmetries of these categories in \S7. In \S8 we prove the universality property of ${\bf T}_t$.\\

\noindent{\bf Acknowledgement:} Both authors are supported in part by DFG grants PE 980/8-1 and PE 980/9-1. V.Ts. is also supported by the Bulgarian Ministry of Education and Science, Scientific Programme ``Enhancing the Research Capacity in Mathematical Sciences (PIKOM)'', No. DO1-67/05.05.2022. Some initial inspiration for our work came from a set of examples of injective resolutions and a conjecture on the injective length of modules in a subcategory of $\TT_{1}$ due to T. Pham, \cite{Pham}.

\section{Basic notions}

The ground field for all vector spaces and tensor products is the field $\CC$ of complex numbers, unless stated otherwise. We set $\otimes:=\otimes_\CC$. If $V$ is a vector space, then $V^*:={\rm Hom}(V,\CC)$ stands for the dual vector space and $\mk{gl}(V)$ denotes the Lie algebra of all linear operators on $V$. By $\NN$ we denote the natural numbers (including $0$), and $|A|$ stands for the cardinality of a set $A$.

Let $V$ be a vector space. For any subset $A\subseteq V$, we write ${\rm span}A\subseteq V$ for the set of all (finite) linear combinations of elements of $A$. A subset $\B\subseteq V$ is a basis of $V$, if ${\rm span}\B=V$ and $\B$ is minimal with this property. We assume the Axiom of Choice which implies that every vector space admits a basis. The dimension of a vector space is the cardinality of a basis.

The space of linear operators on a vector space $V$, considered as a Lie algebra, will be denoted by $\mk{gl}(V)$.

If $M$ is a module over a Lie algebra, or an associative algebra, the {\it socle} of $M$, ${\rm soc}M$ is the semisimple submodule of $M$. The {\it socle filtration} of $M$ is defined inductively by setting ${\rm soc}^1M:={\rm soc}M$, ${\rm soc}^qM:=\pi_{q-1}^{-1}({\rm soc}(M/{\rm soc}^{q-1}M))$, where $\pi_{q-1}:M\to M/{\rm soc}^{q-1}M$ is the canonical projection. The {\it layers} of the socle filtration are defined as $\ul{\rm soc}^qM:={\rm soc}^qM/{\rm soc}^{q-1}M$. The socle filtration of a module $M$ is {\it exhaustive} if $M=\lim\limits_\to {\rm soc}^qM$. The socle filtration of a module of finite length is always exhaustive.

The {\it radical} of a $M$ is the joint kernel of all homomorphisms from $M$ to simple quotients. Setting ${\rm rad}^1M:={\rm rad}M$ and ${\rm rad}^qM:={\rm rad}({\rm rad}^{q-1}M)$ we obtain the {\it radical filtration} of $M$.

\subsection{Mackey Lie algebra and its structure}

Let $V,W$ be fixed vector spaces and
$$
{\bf p}:V\otimes W \to \CC
$$
be a fixed nondegenerate pairing (nondegenerate bilinear form). This determines embeddings $W\subset V^*$ and $V\subset W^*$. The Mackey Lie algebra associated to the pairing ${\bf p}$ is
$$
\mk{gl}^M(V,W):=\{\varphi\in\mk{gl}(V): \varphi^*(W)\subset W \} \;,
$$
where $\varphi^*$ stands for the endomorphism of $V^*$ dual to $\varphi$. We consider $\mk{gl}^M(V,W)$ as a Lie subalgebra of $\mk{gl}(V)$, but it can also be considered as an associative subalgebra of ${\rm End}V$.

We shall focus on the case where the vector spaces $V$ and $W$ are isomorphic and the pairing is diagonalizable. The latter means that there exist bases $\{v_b:b\in \B\}$ of $V$ and $\{w_b:b\in\B\}$, parametrized by the same set $\B$, so that, for $v=\sum\limits_{b\in\B} v(b)v_b\in V$ and $w=\sum\limits_{b\in\B} w(b)w_b\in W$, we have
$$
{\bf p}(v,w) = \sum\limits_{b\in\B} v(b)w(b) \;.
$$
In this situation, $W$ is referred to as the {\it restricted dual} $V_*$ of $V$. Since $W=V_*$ and ${\bf p}$ are fixed, we shall use the short notation $\mk{gl}^M$ for the Lie algebra $\mk{gl}^M(V,V_*)$. Also, we denote $\bar V:=(V_*)^*$ and assume that $V$ is embedded in $\bar V$ by use of the pairing ${\bf p}$.

From now on, we suppose that the dimension of $V$ is an infinite cardinal number of the form $\aleph_{t}$ with a fixed $t\in\NN$. We also use the notation $\mk{gl}^M_{t+1}$ for $\mk{gl}^M$, when the dependence on $t$ is to be emphasized.

Let $\B$ be the index set for a fixed pair of dual bases of $V$ and $V_*$ as above. We have $|\B|=\aleph_t$. Since a pair of dual bases is fixed, both vector spaces $V^*$ and $\bar V$ can be identified with the space ${\rm Maps}(\B,\CC)$. For $s\leq t+1$ we define $V^*_{s}$ and $\bar V_{s}$ to be the respective subspaces of $V^*$ and $\bar V$, identified with the subspace $\{x\in {\rm Maps}(\B,\CC):|{\rm supp}(x)|<\aleph_{s}\}\subset {\rm Maps}(\B,\CC)$, where ${\rm supp}(x):=\{b\in\B:x(b)\ne0\}$. Thus $V_*=V^*_{0}$, $V=\bar V_{0}$, $V^*=V^*_{t+1}$ and $\bar V=\bar V_{t+1}$.

The dimensions and the cardinalities of the vector spaces $\mk{gl}^M$, ${\rm End}V$, $V^*_s$, $\bar V_s$, for $s>0$, are all equal to $2^{\aleph_t}$.

The notion of support is extended from vectors in $V^*$ to vectors in tensor powers $(V^*)^{\otimes q}$ as follows. Any $v\in (V^*)^{\otimes q}$ can be written as a finite sum $v=\sum\limits_{j=1}^n v_1^j\otimes v_2^j\otimes ... \otimes v_q^j$ with $v^j_i\in V^*$. We put
$$
{\rm supp}(v):=\bigcup\limits_{i,j} {\rm supp}(v_i^j) \;.
$$
Clearly $|{\rm supp}(v)|=\max\limits_{i,j}\{|{\rm supp}(v_i^j)|\}$. The notion of cardinality of support can be further induced for elements of the quotient spaces $V^*/V^*_{s}$ by use of representatives. Analogous definitions are valid for elements of $\bar V$ and $\bar V/\bar V_{s}$.

The Mackey Lie algebra can be expressed as
\begin{align*}
\mk{gl}^M &= \{\varphi\in\mk{gl}(V):\forall b\in\B, |{\rm supp}(\varphi^*(x_b))|<\infty\} \\
          &\cong \{ \varphi\in\CC^{\B\times\B} : \forall b\in\B, |{\rm supp}(\varphi_{b,.})|<\infty \;, |{\rm supp}(\varphi_{.,b})|<\infty \} \;,
\end{align*}
where, as customary, $\varphi_{a,b}$ denotes the value of $\varphi$ at $(a,b)\in\B\times \B$. After choosing a linear order on $\B$, the Mackey Lie algebra can be identified with the space of $\B\times \B$-matrices with finitely many nonzero entries in each row and each column, with commutator as Lie bracket. The support of an element $\varphi\in\mk{gl}^M$, with respect to the fixed basis, is defined as
$$
{\rm supp}(\varphi) := \{(a,b)\in\B\times\B:\varphi_{a,b}\ne 0\} \;.
$$

The subalgebra $\mk{gl}(V,V_*):=V\otimes V_*\subset \mk{gl}^M$ is an ideal and consists of all elements in $\mk{gl}^M$ of finite rank. We put $\mk{sl}(V,V_*):={\rm ker}{\bf p}$. This is also an ideal of $\mk{gl}^M$. The set of elementary matrices $\{e_{a,b}:=v_a\otimes x_b:a,b\in\B\}$ is a basis of $\mk{gl}(V,V_*)$.

\begin{prop} (\cite{Chirvasitu-Penkov-RC})

The filtration of length $t+2$
$$
V_*=V^*_{0}\subset V^*_{1}\subset V^*_{2}\subset ... \subset V^*_{t+1}= V^*
$$
is the socle filtration of the $\mk{gl}^M$-modules $V^*$. The layers $V^*_{s+1}/V^*_{s}$ are irreducible. Analogous statements hold for the filtration 
$$
V=\bar V_{0}\subset \bar V_{1}\subset \bar V_{2}\subset ... \subset \bar V_{t+1}= \bar V\;.
$$
\end{prop}

Consequently, the above filtrations of $V^*$ and $\bar V$ depend only on the pairing ${\bf p}$ and not on the chosen basis of $V$ used in their definition.

Our goal in the rest of this section is to determine all ideals of the Lie algebra $\mk{gl}^M$. We also compute the socle and radical filtrations of the adjoint $\mk{gl}^M$-module. We start with

\begin{lemma}\label{Lemma Center and glMs ideals}
\begin{enumerate}
\item[{\rm (i)}] The center of the Mackey Lie algebra consists of the scalar transformations $\CC{\rm id}_V$ of $V$.
\item[{\rm (ii)}] For $0\leq s\leq t+1$, there is an ideal $\mk{gl}_{s}^M\subset\mk{gl}^M$ given by
\begin{gather*}
\mk{gl}_{s}^M := \{\varphi\in\mk{gl}^M:\varphi^*(V^*)\subset V^*_{s}\} \;.
\end{gather*}
\end{enumerate}
\end{lemma}

\begin{proof}
The proof is straightforward.
\end{proof}

\begin{rem}
Note that $\mk{gl}^M_0=\mk{gl}(V,V_*)$ and $\mk{gl}^M_{t+1}=\mk{gl}^M$. Furthermore, if $\mk{gl}^M$ is considered as a subalgebra of $\mk{gl}(V_*)$ instead of $\mk{gl}(V)$, then the ideal $\mk{gl}^M_s$ is given by $\{\psi\in\mk{gl}^M:\psi^*(\bar V)\subset \bar V_{s}\}$.
\end{rem}

For any subset $A\subset\B$ we denote by $\mk g^A$ the subalgebra whose elements are supported on $A\times A$. In particular, $\mk{gl}^M=\mk g^\B$. If $|A|=n$ is finite, then $\mk g^A$ is a copy of $\mk{gl}_n$. If $|A|$ is infinite, then $\mk g^A$ is a Mackey Lie algebra for the obvious restriction of the pairing ${\bf p}$. If $A,B\subset \B$ are disjoint then $\mk g^A$ and $\mk g^B$ commute, and we have a subalgebra of the form $\mk g^A\oplus\mk g^B\subset\mk{gl}^M$, which is block-diagonal if an order on $\B$ is chosen so that $A<B$.

\begin{lemma}\label{Lemma Levivarphi}
Let $\varphi\in\mk{gl}^M$. There exists a partition of $\B$ into a disjoint union of countable (possibly finite) sets $\B=\bigsqcup\limits_{b\in\B'}C_b$ such that 
\begin{gather}\label{For varphi in Levivarphi}
\varphi\in \mk{l}_\varphi:=\bigoplus\limits_{b\in\B'} \mk g^{C_b} \;,\;i.e.,\;\; {\rm supp}(\varphi)\subset \bigsqcup\limits_{b\in\B'} C_b^{\times 2} \;,
\end{gather}
where $\B'$ is an arbitrarily chosen set of representatives of the sets partitioning $\B$. Each set $C_b$ admits a (possibly finite) partition $C_b=\bigsqcup\limits_{n\in\NN}C_b^n$ so that
\begin{gather}\label{For varphi quasiblockdiag}
{\rm supp}(\varphi)\subset \bigsqcup\limits_{b\in\B'} \left( \bigcup\limits_{n\in\NN} (C_b^{n}\cup C_b^{n+1})^{\times 2} \right) \;.
\end{gather}
Moreover, there exists a well-order on $\B$, with respect to which the matrix of $\varphi$ is block-diagonal with blocks of (possibly finite) countable dimensions. Within each block there is a block structure with finite blocks, such that all nonzero entries of the matrix of $\varphi$ lie within the main block-diagonal and the two adjacent block-diagonals. 
\end{lemma}

\begin{proof}
For $b\in \B$ we set $A_\varphi(b):=\{a\in\B:\varphi_{a,b}\ne0\;or\;\varphi_{b,a}\ne0\}$, and note that $A_\varphi(b)$ is a finite subset of $\B$ since $\varphi\in\mk{gl}^M$. We define an equivalence relation on $\B$ by declaring two elements $a,b\in\B$ equivalent if either $a=b$ or there is a finite sequence $b=b_0,b_1,...,b_n=a$ such that $b_{j}\in A_{\varphi}(b_{j-1})$ for $j=1,...,n$. Each equivalence class is at most countable. Let $C_b$ denote the equivalence class of $b\in\B$ and let $C_b^n$ denote the (finite) set of elements $a$ for which a sequence $b_0,...,b_n$ as above exists, but a shorter sequence does not exist. Also let $C_b^0:=\{b\}$. We fix a set of representatives $\B'$ for the equivalence classes. Thus we obtain a decomposition of $\B$ into finite subsets:
\begin{gather}\label{For B as union Cbn}
\B = \bigsqcup\limits_{b\in\B'} \left(\bigsqcup\limits_{n\in\NN} C_b^n\right) \;.
\end{gather}

Now formula (\ref{For varphi quasiblockdiag}) follows by construction and implies formula (\ref{For varphi in Levivarphi}). The asserted order is defined as follows. For $b\in\B'$, we define a well-order on $C_b$ by declaring $b$ to be the minimal element, ordering each $C_b^n$ well, and setting $C_b^n<C_b^{n+1}$. These orders are combined into a well-order of $\B$ through an arbitrarily well-order of $\B'$. Moreover, the subalgebra $\mk l_\varphi \subset\mk{gl}^{M}$ containing $\varphi$ takes the form of a block-diagonal subalgebra with blocks of countable dimension, and the remaining statements concerning the countable blocks of $\varphi$ are easy to verify.
\end{proof}

\begin{rem}\label{Rem Diag or Blocks}
The block structure of the matrix of $\varphi$ constructed in Lemma \ref{Lemma Levivarphi} can be made transparent as follows. Let $D\in\mk{gl}^M$ be the diagonal element with $D_{a,a}=n+1$ if $a\in C_b^n$ where $b\in\B'$ is the unique element such that $a\in C_b$. Then $\varphi$ can be decomposed as $\varphi=\varphi_{-1}+\varphi_0+\varphi_1$ with $[D,\varphi_j]=j\varphi_j$. Hence $\varphi_{-1},\varphi_0,\varphi_1$ belong to any ideal of $\mk{gl}^M$ containing $\varphi$.

The matrix of $\varphi_0$ is block-diagonal with respect to the decomposition (\ref{For B as union Cbn}), while the matrices of $\varphi_{-1}$ and $\varphi_1$ are supported respectively on the first block-diagonal below and above the main block-diagonal. We observe that $\varphi=\varphi_0$ if and only if $\varphi$ is diagonal, i.e., if $\B'=\B$. Furthermore, if $\varphi$ is not diagonal and a block of $\varphi_{\pm1}$ vanishes, then the transposed block of $\varphi_{\mp1}$ is nonzero, i.e., for every $b\in\B'$ and every $n\in\NN$ such that $C_b^{n+1}$ is nonempty, we have
\begin{gather}\label{For supp varphi nil blocks}
{\rm supp}(\varphi)\cap ((C_b^n\times C_b^{n+1})\cup(C_b^{n+1}\times C_b^n)))\ne \emptyset \;.
\end{gather}
\end{rem}

\begin{lemma}\label{Lemma Ideal of varphi}
Let the matrix of $\varphi\in\mk{gl}^M$ have an infinite support, i.e., $|{\rm supp}(\varphi)|=\aleph_s$ with $s\in\{0,...,t\}$. In case $s=t$, suppose furthermore $\varphi\notin\CC{\rm id}_V\oplus\mk{gl}^M_{t}$. Then the ideal $\mk J_\varphi\subset \mk{gl}^M$ generated by $\varphi$ is equal to $\mk{gl}^M_{s+1}$.
\end{lemma}

\begin{proof}
For $\dim V=\aleph_0$ and $\varphi\notin \mk{gl}(V,V_*)\oplus\CC{\rm id}_V$, the statement is proven in \cite[Corollary 6.6]{Penkov-Serganova-Mackey} and the result is $\mk J_\varphi=\mk{gl}^M$. We shall use Lemma \ref{Lemma Levivarphi} to reduce the general case to the case $\dim V=\aleph_0$. In what follows, we identify the elements of $\mk{gl}^M$ with their matrices.

The first step is to show that the ideal $\mk J_\varphi$ contains a diagonal matrix whose support has the cardinality of the support of $\varphi$. Let $\mk l_\varphi\subset \mk{gl}^M$ be the subalgebra containing $\varphi$ provided by Lemma \ref{Lemma Levivarphi} and let $\varphi=\sum\limits_{b\in\B'} \varphi^{(b)}$ be resulting the decomposition, $\varphi^{(b)}$ being the projection of $\varphi$ to $\mk g^{C_b}$. For each $b\in\B'$ there are two possibilities: $C_b$ is either finite, or infinite countable. If $C_b$ is finite, then the ideal generated by $\varphi^{(b)}$ within $\mk g^{C_b}\cong\mk{gl}_{|C_b|}$ contains diagonal matrices by \cite[Lemma 6.5]{Penkov-Serganova-Mackey} (there exists $x,y,z\in\mk g^{C_b}$ such that $[x,[y,[z,\varphi^{(b)}]]]$ is diagonal). If $C_b$ is infinite countable, we can apply the aforementioned statement \cite[Corollary 6.6]{Penkov-Serganova-Mackey} to $\varphi^{(b)}\in\mk g^{C_b}$ because $\varphi^{(b)}$ is not equal to the sum of a scalar matrix and a finite matrix by (\ref{For supp varphi nil blocks}). We deduce that the ideal of $\mk g^{C_b}$ generated by $\varphi^{(b)}$ is the entire $\mk g^{C_b}$ and contains, in particular, the diagonal subalgebra of $\mk g^{C_b}$. 

Now suppose that $\varphi$ is diagonal and either $s<t$ or $\varphi\notin\CC{\rm id}_V\oplus\mk{gl}^M_{t}$. For the next step we will need a certain family of diagonal matrices $\varphi^{(g)}$ belonging to $\mk{J}_\varphi$. Consider a splitting of $\B$ in two parts, $\B=\B_1\sqcup\B_2$, such that $\B_1\subset{\rm supp}(\varphi)$, $|\B_1|=|{\rm supp}(\varphi)|$, and there is an injection $f:\B_1\to\B_2$ with $\varphi_{b,b}\ne \varphi_{f(b),f(b)}$ for all $b\in\B_1$. Put
$$
x:=\sum\limits_{b\in\B_1} \frac{1}{\varphi_{b,b}-\varphi_{f(b),f(b)}} e_{b,f(b)} \quad,\quad y:=\sum\limits_{b\in\B_1} g(b) e_{f(b),b} \;,
$$
where $g:\B_1\to\CC$ is arbitrary. Then
$$
\varphi^{(g)}:=[y,[x,\varphi]]=\sum\limits_{b\in\B_1} g(b) e_{b,b} - g(b) e_{f(b),f(b)}
$$
is a diagonal matrix with support contained in $\B_1\sqcup f(\B_1)$ and determined by the function $g$.

Next, using suitable matrices $\varphi^{(g)}$ we will show that any matrix in $\mk{gl}^M_{s+1}$ with zeros on its diagonal actually belongs to $\mk J_\varphi$. Let $(\mk{gl}^M)_{diag=0}$ be the set of matrices with zeros on the diagonal and $\psi\in(\mk{gl}^M)_{diag=0}$ be an arbitrary matrix with $|{\rm supp}(\psi)|=|{\rm supp}(\varphi)|$. Let $\B=\bigsqcup\limits_{b\in\B''}\tilde C_b$ be the partition of $\B$ defined by $\psi$ as in Lemma \ref{Lemma Levivarphi}. Note that
$$
{\rm supp}(\psi)\subset \bigsqcup\limits_{b\in\B'':|C_b|>1} \tilde C_b^{\times 2} \;.
$$
Hence the set $\B'''=\{b\in\B'':|C_b|>1\}$ has cardinality $|{\rm supp}(\varphi)|$. There is a surjective map 
$$
{\rm supp}(\varphi)\to \bigsqcup\limits_{b\in\B'''} \tilde C_b =:\B_3 \;.
$$
Let $\B_1\subset{\rm supp}(\varphi)$ be any subset such that $|\B_1|=|{\rm supp}(\varphi)|$ and $|\B\setminus \B_1|=|\B|$. Put $\B_2:=\B\setminus\B_1$. Let $f:\B_1\to\B_2$ be an injection such that $\varphi_{b,b}\ne\varphi_{f(b),f(b)}$ and $\B_3\subset \B_1\cup f(\B_1)$. Then $g:\B_1\to\CC$ can be selected so that $\varphi^{(g)}_{a,a}\ne\varphi^{(g)}_{a',a'}$ whenever $a,a'\in \tilde C_b$ for some $b$ and $a\ne a'$. The matrix $\varphi^{(g)}$ satisfies
$$
[\varphi^{(g)},(\mk{gl}^M)_{diag=0}\cap\mk l_\psi]= (\mk{gl}^M)_{diag=0}\cap \mk l_\psi \;.
$$
In particular $\psi\in \mk J_{\varphi}$. We conclude that $(\mk{gl}^M)_{diag=0}\cap \mk{gl}^M_{s+1}\subset\mk J_\varphi$, which in turn implies $\mk{gl}^M_{s+1}\subset \mk J_\varphi$. Since $\varphi\in\mk{gl}^M_{s+1}$, we get $\mk{gl}^M_{s+1}= \mk J_\varphi$.
\end{proof}

\begin{coro}\label{Coro no intermed ideals}
The nonzero ideals of $\mk{gl}^M$ contained in $\mk{gl}^M_{t}$ are exactly $\mk{sl}(V,V_*)$ and $\mk{gl}^M_{s}$ for $s\in\{0,...,t\}$. There is a single proper ideal of $\mk{gl}^M$ strictly containing $\mk{gl}^M_{t}$, and this is $\CC{\rm id}_V\oplus \mk{gl}^M_{t}$.
\end{coro}

\begin{proof}
Both statements follow immediately from Lemma \ref{Lemma Ideal of varphi}. 
\end{proof}

We are now in a position to describe the socle filtration of the Lie algebra $\mk{gl}^M$.

\begin{theorem}\label{Theo socfil glM}

The adjoint $\mk{gl}^M$-module is indecomposable, has length $t+4$, and its socle filtration is given by
\begin{gather*}
\begin{array}{ll}
{\rm soc}^1\mk{gl}^M  = \CC{\rm id}_V\oplus \mk{sl}(V,V_*) &\;,  \\
{\rm soc}^2\mk{gl}^M = \CC{\rm id}_V\oplus \mk{gl}(V,V_*) & \;,\quad \ul{\rm soc}^2\mk{gl}^M  = \mk{q} \cong \CC \\
{\rm soc}^{s+3}\mk{gl}^M = \CC{\rm id}_V\oplus \mk{gl}^M_{s+1} &\;,\quad \ul{\rm soc}^{s+3}\mk{gl}^M  = \mk{gl}^M_{s+1}/\mk{gl}^M_{s} \;, \quad s=0,...,t-1,\\
{\rm soc}^{t+3}\mk{gl}^M = \mk{gl}^M &\;,\quad \ul{\rm soc}^{t+3}\mk{gl}^M  = \mk{gl}^M/(\CC{\rm id}_V\oplus\mk{gl}^M_{t}) \;.
\end{array}
\end{gather*}
Moreover, for $s\geq 1$ the layer $\ul{\rm soc}^{s+1}\mk{gl}^M$ is a simple $\mk{gl}^M$-module.
\end{theorem}

\begin{proof}
The submodules of the adjoint $\mk{gl}^M$-module are the ideals of the Lie algebra $\mk{gl}^M$. We begin with the chain of ideals obtained in Lemma \ref{Lemma Center and glMs ideals}, with added initial term $\mk{sl}(V,V_*)$, i.e.,
\begin{gather}\label{For chain of ideals}
\mk{sl}(V,V_*)\subset\mk{gl}(V,V_*)=\mk{gl}^M_0\subset\mk{gl}^M_1\subset ...\subset\mk{gl}^M_{t}\subset \mk{gl}^M_{t+1}=\mk{gl}^M\;.
\end{gather}
The Lie algebra $\mk{sl}(V,V_*)$ is simple, being a direct limit of simple Lie algebras, and there are no ideals of $\mk{gl}^M$ between $\mk{sl}(V,V_*)$ and $\mk{gl}(V,V_*)$ because the quotient is 1-dimensional. Moreover, Corollary \ref{Coro no intermed ideals} implies that all  inclusions in (\ref{For chain of ideals}) are essential, and that all quotients $\mk{gl}^M_{s+1}/\mk{gl}^M_{s}$ for $s=0,...,t-1$, as well as the quotient $\mk{gl}^M/(\CC{\rm id}_V\oplus\mk{gl}^M_{t})$, are simple. Since $\CC{\rm id}_V\subset {\rm soc}^1\mk{gl}^M$, the statement about the socle filtration follows.

The fact that all inclusions in (\ref{For chain of ideals}) are essential implies that in order to establish the indecomposability of $\mk{gl}^M$ it suffices to show that the ideal $\CC{\rm id}_V$ does not split off. This is a direct corollary of the famous assertion of Heisenberg that the equation $[x,y]={\rm id}_V$ has a solution in infinite three-diagonal matrices. Classically this statement is known for $t=0$, but it holds for any $t$ since one easily constructs block-diagonal matrices $x,y$ in $\mk{gl}^M\setminus\mk{gl}^M_{t}$ such that $[x,y]={\rm id}_V$. It is essential that each diagonal block, being a three-diagonal Heisenberg matrix of countable size, has finite rows and columns, which ensures that $x,y$ lie in $\mk{gl}^M$ and not just in $\mk{gl}(V)$.
\end{proof}

\begin{coro}
The radical filtration of the adjoint $\mk{gl}^M$-module is the following modification of filtration (\ref{For chain of ideals}):
\begin{gather*}
\mk{sl}(V,V_*)\subset\mk{gl}(V,V_*)=\mk{gl}^M_0\subset\mk{gl}^M_1\subset ...\subset\mk{gl}^M_{t-1}\subset \CC{\rm id}_V\oplus\mk{gl}^M_{t} \subset \mk{gl}^M_{t+1}=\mk{gl}^M\;.
\end{gather*}
In other words,
\begin{gather*}
{\rm rad}^1\mk{gl}^M  = \CC{\rm id}_V\oplus\mk{gl}^M_{t} \;,\;  {\rm rad}^{s+1}\mk{gl}^M = \mk{gl}^M_{t-s} \; for\; s=1,...,t,\;
{\rm rad}^{t+2}\mk{gl}^M = \mk{sl}(V,V_*)\;.
\end{gather*}
\end{coro}

\begin{proof}
The statement follows immediately from the properties of the chain (\ref{For chain of ideals}), and from the fact that the direct sum $\CC{\rm id}_V\oplus\mk{gl}^M_{t-1}$ is a direct sum of ideals.
\end{proof}

\begin{theorem}\label{Theo Ideals List}
The following is a complete list of nonzero proper ideals in the Mackey Lie algebra $\mk{gl}^M$:
\begin{enumerate}
\item[{\rm (i)}] the center $\CC{\rm id}_V$; 
\item[{\rm (ii)}] $\mk{sl}(V,V_*)$;
\item[{\rm (iii)}] $\CC{\rm id}_V\oplus \mk{sl}(V,V_*)={\rm soc}^1\mk{gl}^M$; 
\item[{\rm (iv)}] $\CC(z{\rm id}_V+ e_{b,b})+\mk{sl}(V,V_*)\subset {\rm soc}^2\mk{gl}^M$ for arbitrary $z\in\CC\setminus \{0\}$ and $b\in\B$; this ideal depends only on $z$ and not on $b$;
\item[{\rm (v)}] $\mk{gl}^M_s\subset{\rm soc}^{s+2}\mk{gl}^M$ for $s=0,...,t$ (recall that $\mk{gl}^M_0=\mk{gl}(V,V_*)$);
\item[{\rm (vi)}] $\CC{\rm id}_V\oplus\mk{gl}^M_s ={\rm soc}^{s+2}\mk{gl}^M$ for $s=0,...,t$.
\end{enumerate}
\end{theorem}

\begin{proof}
Let $\mk J\subset\mk{gl}^M$ be a nonzero proper ideal of $\mk{gl}^M$ and let $r$ be the minimal integer such that $\mk J\subset{\rm soc}^{r}\mk{gl}^M$. Then $r\leq t+2$ since any $\varphi\in\mk{gl}^M$ which lies in the preimage of a nonzero element of the simple quotient $\mk{gl}^M/{\rm soc}^{t+2}\mk{gl}^M$ generates $\mk{gl}^M$. If $r=1$ then $\mk J$ is one of the ideals (i),(ii),(iii). Assume $2\leq r\leq t+2$. The minimality of $r$ ensures that $\mk J$ projects nontrivially to the layer $\ul{\rm soc}^{r}\mk{gl}^M$, which is a simple module. We consider two cases, $r=2$ and $r>2$. If $r>2$, then the layer $\ul{\rm soc}^{r}\mk{gl}^M=\mk{gl}^M_{r-2}/\mk{gl}^M_{r-3}$ is a nontrivial simple $\mk{gl}^M$-module. Hence, the projection of $\mk J$ to $\mk{gl}^M_{r-2}$ is the entire $\mk{gl}^M_{r-2}$. Since $\CC{\rm id}_V$ is central, we conclude that $\mk{gl}^M_{r-2}\subset\mk J$. So the possibilities are $\mk J=\mk{gl}^M_{r-2}$ or $\mk J=\CC{\rm id}_V\oplus\mk{gl}^M_{r-2}$, which account for items (v) and (vi) in our list with $r-2=s>0$. The case $r=2$ is covered by (v) and (vi) for $s=0$, along with the remaining item (iv). Indeed, if $\mk J\subset \CC{\rm id}\oplus\mk{gl}(V,V_*)$ projects nontrivially to $\ul{\rm soc}^2\mk{gl}^M\cong \CC$, we have either ${\rm soc}\mk J=\CC{\rm id}_V\oplus\mk{sl}(V,V_*)$ or ${\rm soc}\mk J=\mk{sl}(V,V_*)$. In the former case, we get $\mk J= \CC{\rm id}\oplus\mk{gl}(V,V_*)$ covering item (vi), $s=0$. In the latter case, since ${\rm soc}^2\mk{gl}^M/\mk{sl}(V,V_*)\cong\CC\oplus\CC$, we conclude that $\ul{\rm soc}^2\mk J\cong \CC$ and $\mk J$ is generated by an element of the form $\varphi=z{\rm id}_V+e_{b,b}$ with a suitable $z\in\CC$ and any $b\in\B$ (it is clear that all $b\in\B$ yield the same ideal for a fixed $z$). For $z=0$ we obtain $\mk J=\mk{gl}(V,V_*)$ covering item (v) with $s=0$. For $z\ne0$ we obtain item (iv). 
\end{proof}

\begin{coro}
The only ideal of $\mk{gl}^M$ which is not principal, i.e., is not generated by one element, is $\CC{\rm id}_V\oplus\mk{gl}(V,V_*)$.
\end{coro}

\begin{proof}
The result follows from the above proof and Lemma \ref{Lemma Ideal of varphi}.
\end{proof}

\begin{rem}
The ideals of item {\rm (iv)} in Theorem \ref{Theo Ideals List} are very similar to certain ideals of the Lie algebra $\mk{gl}(V)$ found in \cite{Stewart}, see also \cite{Polish}.
\end{rem}

\subsection{Tensor algebras and Schur functors}\label{Sec Schur}

Let us recall some general results for decompositions of tensor powers.

We denote by $T(X)$ the tensor algebra generated by a vector space $X$. For any Young diagram $\lambda$ and any vector space $X$, we denote by  $X_\lambda$ the image of $X$ under the Schur functor corresponding to $\lambda$: $X_\lambda\subset X^{\otimes|\lambda|}$. Here $|\lambda|$ denotes the number of boxes in $\lambda$. Also, we denote by $\Lambda$ the set of Young diagrams, by $\emptyset$ the empty diagram, and by $\lambda^\perp$ the transposed Young diagram (the corresponding partition is called conjugate). Standard Schur-Weyl duality yields the following decomposition in our context
$$
X^{\otimes m}=\bigoplus\limits_{|\lambda|=m} \CC^\lambda \otimes X_\lambda \quad,\quad T(X) = \bigoplus\limits_{\lambda} \CC^\lambda \otimes X_\lambda \;,
$$
where $\CC^\lambda$ is the irreducible module of the symmetric group on $|\lambda|$ letters determined by the partition $\lambda$. The $m$-the symmetric and skew-symmetric tensor powers $S^mX$ and $\Lambda^mX$ correspond to $\lambda=(m)$ and $\lambda=(1,...,1)$, respectively.

We denote by $c_\lambda:X^{\otimes|\lambda|}\to X_\lambda$ the projection associated to a standard Young tableau of shape $\lambda$, which we fix once and for all to be the tableau where the numbers $1,...,|\lambda|$ fill the boxes of $\lambda$ in their initial order.

\begin{prop} (\cite[Proposition~2.2]{Chirvasitu-Penkov-UTC}; \cite[\S~4]{Penkov-Serganova-Mackey})

Let $X,Y$ be two objects in a tensor category. Then the following hold.
\begin{enumerate}
\item For $m,n\geq 0$, $X^{\otimes m}\otimes Y^{\otimes n} = \bigoplus\limits_{|\lambda|=m,|\mu|=n} \CC^\lambda\otimes\CC^\mu \otimes X_\lambda\otimes Y_\mu$.
\item For $m\geq 0$, $S^m(X\otimes Y)= \bigoplus\limits_{|\lambda|=m} X_\lambda\otimes Y_\lambda$.
\item For $m\geq 0$, $\Lambda^m(X\otimes Y)= \bigoplus\limits_{|\lambda|=m} X_\lambda\otimes Y_{\lambda^\perp}$.
\end{enumerate}
\end{prop}

\subsection{Dense subalgebras}

\begin{definition}\label{Def Dense}
Let $\mk G$ be a Lie algebra, $R$ be a $\mk G$-module, and $\mk H\subset\mk G$ a subalgebra. The subalgebra $\mk H$ is said to act densely on $R$, if for any finite subset of vectors $r_1,...,r_n\in R$ and any $g\in\mk G$, there exists $h\in\mk H$ such that $g\cdot r_j=h\cdot r_j$ for  $j=1,...,n$.
\end{definition}

\begin{prop}\label{Prop DenseProperties}
Let $\mk G$ be a Lie algebra, $\mk J\subset \mk G$ an ideal and $R$ a $\mk G$-module.
\begin{enumerate}
\item[{\rm (a)}] If $\mk G$ acts densely on $R$, then for any partition $\lambda$ the $\mk G$-module $R_\lambda$ is simple with ${\rm End}_{\mk G}R_\lambda\cong \CC$.
\item[{\rm (b)}] If $\mk J$ acts densely and irreducibly on $R$ with ${\rm End}_{\mk J}R=\CC$, then the functor
$$
\bullet\otimes R : \mk G/\mk J{\rm -mod} \lw \mk G{\rm -mod}
$$
is fully faithful; it sends simple modules to simple modules and essential inclusions to essential inclusions.
\end{enumerate}
\end{prop}

Proofs are given in \cite[Proposition~4.5]{Chirvasitu-Penkov-RC} for part (a), \cite[Lemma~4.4]{Chirvasitu-Penkov-RC} and \cite[Lemma~3.3]{Chirvasitu-Penkov-UTC} for part (b).

\section{The module $I$}\label{Sec Module I}

Consider the canonical projection
\begin{gather}\label{For ptilde Q}
\tilde{\bf p}:\bar V \otimes V^*\lw \bar V\otimes V^*/(\mk{sl}(V,V^*)+\mk{sl}(\bar V,V_*)) =:Q \;,
\end{gather}
where $\mk{sl}(V,V^*):={\rm ker}(V\otimes V^*\to \CC)$ and $\mk{sl}(\bar V,V_*):={\rm ker}(\bar V\otimes V_*\to \CC)$.
Recall that $\mk{gl}(V,V_*)=V\otimes V_*$ and $\mk{sl}(V,V_*)={\rm ker}({\bf p}:V\otimes V_*\to\CC)$ are ideals of $\mk{gl}^M$. Hence
$$
\mk q:=\tilde{\bf p}(V\otimes V_*)=\tilde{\bf p}(V\otimes V^*)=\tilde{\bf p}(\bar V\otimes V_*) \subset Q 
$$
is a 1-dimensional trivial $\mk{gl}^M$-module, generated by $\tilde{\bf p}(v_b\otimes x_b)$ for an arbitrary $b\in\B$. Consequently, there is a short exact sequence of $\mk{gl}^M$-modules
\begin{gather}\label{For CQF}
0 \lw \CC \stackrel{\iota}{\lw} Q \stackrel{\pi}{\lw} F \lw 0 \;,
\end{gather}
where $\iota(\CC)=\mk q$ with $\iota(1)=\tilde{\bf p}(v_b\otimes x_b)$ for any $b\in\B$ and $F:=\bar V/V \otimes V^*/V_*$.

We define a $\mk{gl}^M$-module by setting
\begin{gather}\label{For I}
I:=\lim\limits_{\lw} S^kQ ,
\end{gather}
where $\iota_k:S^kQ\hookrightarrow S^{k+1}Q$ is the map generalizing $\iota_0=\iota$, given by
$$
S^kQ\cong S^kQ\otimes\CC \stackrel{{\rm id}\otimes \iota}{\lw} S^kQ \otimes Q \stackrel{multiply}{\lw} S^{k+1}Q \;.
$$

The exact sequence (\ref{For CQF}) generalizes, for $k\in\NN$, to
\begin{gather}\label{For SkCQF}
0 \lw S^kQ \stackrel{\iota_k}{\lw} S^{k+1}Q \stackrel{\pi_k}{\lw} S^{k+1}F \lw 0 \;.
\end{gather}
It follows that the successive quotients (or layers) of the defining filtration of $I$ are $S^{k}F$ for $k=0,1,2,...$.

\begin{prop}\label{Prop I is algebra} (\cite{Chirvasitu-Penkov-UTC})

The module $I$ carries a commutative algebra structure, made evident by the isomorphism of $\mk{gl}^M$-modules $I\cong S^\bullet Q/\langle 1-\iota(1)\rangle$, where $\langle 1-\iota(1)\rangle$ denotes the ideal of $S^\bullet Q$ generated by $1-\iota(1)$.
\end{prop}

We observe that for every pair of natural numbers $r,s\leq t+1$, we have a $\mk{gl}^M$-submodule of $Q$ defined as
$$
Q^{r,s}:=\tilde{\bf p}(\bar V_{r}\otimes V^*_{s}) \subseteq Q \;.
$$
Since $\mk q\subset Q^{r,s}$, the construction of $I$ can be repeated with $Q^{r,s}$ instead of $Q$, yielding a $\mk{gl}^M$-module
\begin{gather}\label{For Irs}
I^{r,s}:= \lim\limits_{k\to\infty} S^kQ^{r,s}
\end{gather}
which is an essential extension of the trivial module $\mk q\cong\CC$. Thus we obtain a family of $\mk{gl}^M$-submodules and commutative subalgebras of $I$:
$$
\mk q\subset I^{r,s} \subset I^{t+1,t+1}=I \;, \; r,s\leq t+1\;.
$$
Note that $I^{r,s}\subset I^{r',s'}$ if and only if $r\leq r'$ and $s\leq s'$.\\

Next, we define a morphism of $\mk{gl}^M$-modules $\psi:I\to F\otimes I$.

Let $S^\bullet Q=\bigoplus\limits_{k=0}^\infty S^kQ$ be the symmetric algebra over $Q$. Let
$$
\Delta:S^\bullet Q \lw S^\bullet Q\otimes S^\bullet Q \quad,\quad \Delta(v)=v\otimes 1 + 1\otimes v \;{\rm for}\; v\in Q
$$
be the comultiplication which defines a Hopf algebra structure on $S^\bullet Q$. The comultiplication is a morphism of $\mk{gl}^M$-modules. We denote by
$$
\Delta^k_j:S^k Q \to S^j Q \otimes S^{k-j} Q
$$
the composition of the restriction $\Delta:S^k Q \to \bigoplus\limits_{j=0}^k S^j Q \otimes S^{k-j} Q$ with the projection to the $j$-th summand.

For $k\in\NN$ we have a morphism $\psi^k=(\pi_k\otimes{\rm id})\circ\Delta^k_1$:
$$
\psi^k:S^kQ \stackrel{\Delta^k_1}\lw Q \otimes S^{k-1} Q \stackrel{\pi_k\otimes{\rm id}}{\lw} F\otimes S^{k-1} Q \;.
$$
We define
$$
\psi = \lim\limits_{\lw} \psi^k : I \lw F\otimes I \;.
$$

\begin{lemma}
We have $\psi^{k+1}\circ \iota_k = ({\rm id}\otimes \iota_{k-1}) \circ \psi^k$.
\end{lemma}

\begin{proof}
The argument in \cite[\S~3.1]{Chirvasitu-Penkov-UTC} can be repeated in our context without alteration.
\end{proof}

\begin{lemma}
The kernel of $\psi$ is one dimensional, given by ${\rm ker}(\psi)=\mk q\cong \CC$.
\end{lemma}

\begin{proof}
Since ${\rm ker}(\pi_k\otimes{\rm id})=(\iota_k\circ...\circ \iota_1)(\mk q)\otimes S^{k-1} Q$, we have
\begin{align*}
{\rm ker}(\psi^k) & = (\Delta^k_1)^{-1}((\iota_k\circ...\circ \iota_1)(\mk q)\otimes S^{k-1} Q) \\
                  & = (\iota_k\circ...\circ \iota_1)(\mk q) \;.
\end{align*}
\end{proof}

It follows that the map $\psi$ factors through the projection $I\to I/\mk q$. Abusing notation, we denote the resulting monomorphism $I/\mk q\to F\otimes I$ by $\psi$ as well.

The constructions of this subsection can be carried out for $I^{r,s}$ (see formula (\ref{For Irs})) instead of $I$. One only needs to replace $Q$ by $Q^{r,s}$ and $F$ by $F^{r,s}=\bar V_{r}/V \otimes V^*_{s}/V_*$. The restricted morphism $\psi$ yields $\psi:I^{r,s}\to F^{r,s} \otimes I^{r,s}$, factoring through $I^{r,s}\to I^{r,s}/\mk q$.

\section{Background from category theory}

\subsection{Ordered Grothendieck categories}

\begin{definition}\label{Def OrdGrothCat} (\cite[Def.~2.3]{Chirvasitu-Penkov-UTC})
Let $(\mc P,\preceq)$ be a poset. An ordered Grothendieck category with underlying order $(\mc P,\preceq)$ is a Grothendieck category $\mc C$ with a given set of objects $X_i$, $i\in\mc P$ with the following properties.

\begin{enumerate}
\item[(a)] The objects $X_i$ are semi-artinian, in the sense that all their nonzero quotients have nonzero socles.
\item[(b)] Every object in $\mc C$ is a subquotient of a direct sum of copies of various $X_i$.
\item[(c)] For every isomorphism type of simple objects in $\mc C$ there exists a unique $i\in\mc P$ such that this type occurs in
$$\mc S_i:=\{\textrm{isomorphism types of simples in soc}X_i\}\;.$$
\item[(d)] Simple subquotients of $X_i$ outside ${\rm soc}X_i$ are in the socle of some $X_j$ with $j\prec i$.
\item[(e)] Each $X_i$ is a direct sum of objects with simple socle.
\item[(f)] For $j\prec i$, the maximal subobject $X_{i\succ j}\subset X_i$ whose simple constituents belong to various $\mc S_k$ for $i\succeq k\npreceq j$ is the common kernel of a family of morphisms $X_i\to X_j$.
\end{enumerate}
We refer to $X_i$, $i\in\mc P$ as the order-defining objects of the ordered Grothendieck category $\mc C$.
\end{definition}

\begin{prop}\label{Prop InjHullInGrothCat} (\cite[Proposition~2.5, Corollary~2.6]{Chirvasitu-Penkov-RC})

Let $U$ be a simple subobject in ${\rm soc}X_i$ for some $i\in\mc P$ and let $\hat U$ be the direct summand of $X_i$ such that $U={\rm soc}\hat U$. Then $\hat U$ is an injective hull of $U$.

The indecomposable injective objects in $\mc C$ are, up to isomorphism, precisely the indecomposable summands of the objects $X_i$, $i\in\mc P$.
\end{prop}

\subsection{Tensor categories}

Let $\mc C$ be a tensor category.

\begin{rem}
Let $0\to x'\to x\to x''\to 0$ be a short exact sequence in a tensor category. Then the symmetric power $S^kx$ has a filtration $0=F_{-1}\subset F_0\subset ...\subset F_n=S^kx$ with $F/F_{j-1}\cong S^{k-j}x'\otimes S^jx''$ for $0\leq j\leq k$.
\end{rem}

\begin{lemma}\label{Lemma TensInjRes}
Suppose that the tensor product of any two injective objects in $\mc C$ is again an injective object. Let
\begin{align*}
& 0 \to U_1 \to M^0 \to M^1 \to M^2 \to ...\to M^m \to 0 \\
& 0 \to U_2 \to N^0 \to N^1 \to N^2 \to ...\to N^n \to 0 
\end{align*}
be injective resolutions of two objects $U_1,U_2$. Then an injective resolution of $U_1\otimes U_2$ is given by
$$
0 \to U_1\otimes U_2 \to R^0 \to R^1 \to R^2 \to ...\to R^{m+n} \to 0  \;,
$$
where $R^k=\bigoplus\limits_{j=0}^k M^{k-j}\otimes N^j$ for $k=0,1,...,m+n$, and the differential of this complex, restricted to $M^{k-j}\otimes N^j$, equals the tensor product of the respective differential of the initial two complexes.
\end{lemma}

\begin{proof}
The exactness of the resulting sequence follows from the K\"unneth formula. The modules $R^j$ are injective, by hypothesis, and hence we have and injective resolution.
\end{proof}

\begin{definition}\label{Def pure}
A simple object in a tensor category is called pure, if it is not isomorphic to the tensor product of two nontrivial simple objects.
\end{definition}

\section{Categories of tensor modules for Mackey Lie algebras}

We denote by $\TT_{t}$ the smallest full tensor Grothendieck subcategory of $\mk{gl}^M$-mod that contains $V^*$ and $\bar V$ and is closed under taking subquotients. For any set of objects $X,Y,...$ in $\TT_{t}$, we denote by $\TT(X,Y,...)$ the smallest full tensor Grothendieck subcategory of $\TT_t$ containing these objects and closed under taking subquotients. In particular, $\TT_t=\TT(V^*,\bar V)$. Since $t$ is fixed in the discussion, we abbreviate the notation to $\TT=\TT_t$ most of the time.

\subsection{The category $\TT(V_*,V)$}\label{Sec Vstar l po V m}

Here we recollect some known results on the category $\TT(V_*,V)$ that will serve as building blocks for some subsequent constructions.

For any pair of nonnegative integers $l,m$ we have a $\mk{gl}^M$-module decomposition
$$
V_*^{\otimes l}\otimes V^{\otimes m} = \bigoplus\limits_{|\lambda|=l,|\mu|=m} \CC^\lambda\otimes \CC^\mu\otimes (V_*)_\lambda\otimes V_\mu \;.
$$
For $(i,j)\in\{1,...,|\mu|\}\times\{1,...,|\nu|\}$ we denote by ${\bf p}_{i,j}:V_*^{\otimes l}\otimes V^{\otimes m}\to V_*^{\otimes(l-1)}\otimes V^{\otimes(m-1)}$ the contraction obtained by applying ${\bf p}:V_*\otimes V\to \CC$ to the $i$-th tensorand of $V_*^{\otimes l}$ and the $j$-th tensorand of $V^{\otimes m}$. The submodule annihilated by all these contractions is
$$
V_{l;m} := \bigcap\limits_{i,j} {\rm ker}({\bf p}_{i,j}) \subset V_*^{\otimes l}\otimes V^{\otimes m} \;.
$$
For any pair of Young diagrams $\lambda,\mu$ with $|\lambda|=l$ and $|\mu|=m$, and any fixed copy of $(V_*)_\lambda\otimes V_\mu$ inside $V_*^{\otimes l}\otimes V^{\otimes m}$, we denote
$$
V_{\lambda;\mu} := \bigcap\limits_{i,j} {\rm ker}({{\bf p}_{i,j}}_{\vert (V_*)_\lambda\otimes V_\mu}) \;.
$$
More generally, for a pair of nonnegative integers $m,n$ and a pair of multiindices of the same size $\ul{i}=\{1\leq i_1<...<i_k\leq m\}$, $\ul{j}=\{1\leq j_1<...<j_k\leq n\}$, we have a morphism of $\mk{gl}^M$-modules
\begin{align*}
{\bf p}_{\ul{i},\ul{j}}: & V_*^{\otimes m}\otimes V^{\otimes n}\to V_*^{\otimes (m-k)}\otimes V^{\otimes (n-k)} \\
 & (x_1\otimes ... \otimes x_k) \otimes (v_1\otimes ...\otimes v_k) \mapsto (\prod\limits_{l=1}^k{\bf p}(x_{i_l}\otimes v_{j_l})) (\otimes_{i\notin \ul{i}} x_i)\otimes (\otimes_{j\notin\ul{j}} v_j) \;.
\end{align*}

\begin{prop}\label{Prop DenseVlambdamu PS} \cite{Penkov-Serganova-Mackey}

For any pair of Young diagrams $\lambda,\mu$, the representation $V_{\lambda;\mu}$ of $\mk{gl}^M$ is irreducible and the action of $\mk{gl}^M$ on $V_{\lambda;\mu}$ is dense in the sense of Definition \ref{Def Dense}. The same holds for the restriction of this representation to $\mk{gl}(V,V_*)$.
\end{prop}

\begin{prop} (\cite[Theorem~4.1]{Penkov-Serganova-Mackey})

Let $l,m$ be nonnegative integers. The socle filtration of the $\mk{sl}(V, V_*)$-module $V_*^{\otimes l}\otimes V^{\otimes m}$ is given by
$$
{\rm soc}^k(V_*^{\otimes l}\otimes V^{\otimes m}) = \bigcap\limits_{\#\ul{i}=\#\ul{j}=k} {\rm ker}({\bf p}_{\ul{i},\ul{j}}) \quad,\quad k=1,...,\min\{l,m\}\;. 
$$
In particular,
$$
{\rm soc}(V_*^{\otimes l}\otimes V^{\otimes m}) = V_{l;m} = \bigoplus\limits_{|\lambda|=l,|\mu|=m} \CC^\lambda\otimes \CC^\mu\otimes V_{\lambda,\mu} \;.
$$
\end{prop}

\begin{theorem}\label{Theo PS soc filt Vlambda TENS Vmu} (\cite[Th.~2.3]{Penkov-Styrkas}, \cite[\S~4]{Penkov-Serganova-Mackey})

Let $\lambda,\mu\in\Lambda$ be Young diagrams. Then the layers of the socle filtration of the $\mk{sl}(V,V_*)$-module $(V_*)_\lambda\otimes V_\mu$ have the following isotypic decompositions
\begin{gather}\label{For h lambdamuxieta}
\ul{\rm soc}^{k+1}((V_*)_\lambda\otimes V_\mu) \cong \bigoplus\limits_{\xi,\eta\in\Lambda:|\lambda|-|\xi|=k} h^{\lambda;\mu}_{\xi;\eta}\cdot V_{\xi;\eta} \quad,\quad where\quad h^{\lambda;\mu}_{\xi;\eta}:=\sum\limits_{\nu\in\Lambda} N^\lambda_{\xi\nu} N^\mu_{\nu\eta} \;.
\end{gather}
The same applies for the Mackey Lie algebra $\mk{gl}^M$ instead of $\mk{sl}(V,V_*)$.
\end{theorem}

It is an elementary but essential observation that $h^{\lambda;\mu}_{\xi;\eta}\ne 0$ implies that there exists a unique $k=k^{\lambda;\mu}_{\xi;\eta}:=|\lambda|-|\xi|=|\mu|-|\eta|$ such that $h^{\lambda;\mu}_{\xi;\eta}={\rm Hom}(V_{\xi;\eta},\ul{\rm soc}^{k+1}((V_*)_\lambda\otimes V_\mu))\ne 0$.

\begin{definition}\label{Def Poset 0 0}
We consider the poset $\mc P_{0,0}=\NN^{2}$ with elements indexed as $(l;m)$ with the following relation:
$$
(l;m)\leq (l';m') \quad \tst \quad \begin{array}{|l} l - m= l' - m' \\ l\leq l' \;,\; m\leq m' \end{array}\;.
$$
\end{definition}

\begin{theorem}\label{Theo Tsmall is ord Groth} (\cite[\S~4.2]{Chirvasitu-Penkov-RC})

The category $\TT(V_*,V)$ is an ordered Grothendieck category with order-defining objects $(V_*)^{\otimes l} \otimes V^{\otimes m}$ parametrized by the poset $\mc P_{0,0}$. The socles of the order-defining objects are
$$
{\rm soc}((V_*)^{\otimes l} \otimes V^{\otimes m}) = V_{l;m} \;.
$$
The simple objects and the indecomposable injectives of $\TT(V_*,V)$ are, up to isomorphism, respectively, $V_{\lambda;\mu}$ and $(V_*)_{\lambda}\otimes V_{\mu}$ with $\lambda,\mu\in\Lambda$.
\end{theorem}

The next theorem describes injective resolutions of the simple objects in $\TT(V_*,V)$.

\begin{theorem}\label{Theo smallTT Inj n Ext to Hom} (\cite{Penkov-Styrkas}; \cite{Penkov-Serganova-Mackey})
For any pair of Young diagrams $\lambda,\mu$, the simple $\mk{gl}^M$-module $V_{\lambda;\mu}$ admits the following injective resolution in $\TT(V_*,V)$ of length $|\lambda\cap\mu^{\perp}|$:
$$
0\to V_{\lambda;\mu} \to \mc I^0(V_{\lambda;\mu}) \to \mc I^1(V_{\lambda;\mu}) \to ... \to \mc I^{|\lambda\cap\mu^{\perp}|}(V_{\lambda;\mu}) \to 0 \;,
$$
with $\mc I^k(V_{\lambda;\mu}) = \bigoplus\limits_{\xi,\eta\in\Lambda:|\lambda|-|\xi|=k} m^{\lambda;\mu}_{\xi;\eta}\cdot (V_*)_{\xi}\otimes V_{\eta}$ where $m^{\lambda;\mu}_{\xi;\eta}:=\sum\limits_{\nu\in\Lambda} N^{\lambda}_{\xi\nu}N^{\mu}_{\nu^\perp\eta}$.

Consequently, for $k\geq 0$
$$
{\rm Ext}^k_{\TT(V_*,V)}(V_{\xi;\eta},V_{\lambda;\mu})\cong {\rm Hom}(V_{\xi;\eta^\perp},\ul{\rm soc}^{k+1}((V_*)_{\lambda} \otimes V_{\mu^\perp}))\;,
$$
and ${\rm Ext}^k_{\TT(V_*,V)}(V_{\xi;\eta},V_{\lambda;\mu})\ne 0$ implies $k=k^{\lambda;\mu}_{\xi;\eta}=|\lambda|-|\xi|=|\mu|-|\eta|$.
\end{theorem}

In addition we observe that $m^{\lambda;\mu}_{\xi;\eta}=h^{\lambda;\mu^{\perp}}_{\xi;\eta^{\perp}}$.

\subsection{Some families of tensor modules}

In this section, generalizing constructions of \cite{Chirvasitu-Penkov-RC},\cite{Chirvasitu-Penkov-UTC}, we determine the simple $\mk{gl}^M$-subquotients of the tensor algebra $T(V^*\oplus \bar V)$. We also define and study several families of $\mk{gl}^M$-modules relevant for the structure of $\TT_t$ as an ordered Grothendieck category.

We let
$$
\bm\Lambda:=\Lambda^{t+1}\times\Lambda\times\Lambda\times\Lambda^{t+1}
$$
be the set of $2(t+2)$-tuples of diagrams. We view its elements $\bm\lambda\in\bm\Lambda$ as pairs of sequences of length $(t+2)$, notationwise separated by ``;'', with indices increasing outwards, and unindexed initial entries, i.e.,
$$
\bm\lambda=(\lambda_\bullet,\lambda;\mu,\mu_\bullet)=(\lambda_t,...,\lambda_0,\lambda;\mu,\mu_0,...,\mu_t)\;.
$$
If the tail of a sequence $\nu_\bullet=(\nu_0,...,\nu_t)$ consists of empty diagrams, we often omit it if the number $t$ is fixed in the context. The sequence of empty diagrams is denoted by $\emptyset_\bullet$.

We define the following three families of modules indexed by the set $\bm\Lambda$:
\begin{gather}\label{For LJITlambda}
\begin{array}{rl}
L_{\lambda_\bullet,\lambda;\mu,\mu_\bullet} :=& \left(\bigotimes\limits_{\alpha=0}^t (V^*_{\alpha+1}/V^*_{\alpha})_{\lambda_\alpha}\right)\otimes V_{\lambda;\mu}\otimes \left(\bigotimes\limits_{\beta=0}^t (\bar V_{\beta+1}/\bar V_{\beta})_{\mu_\beta}\right)\;, \\
J_{\lambda_\bullet,\lambda;\mu,\mu_\bullet} :=& \left(\bigotimes\limits_{\alpha=0}^t (V^*/V^*_{\alpha})_{\lambda_\alpha}\right)\otimes V^*_{\lambda}\otimes \bar V_{\mu}\otimes \left(\bigotimes\limits_{\beta=0}^t (\bar V/\bar V_{\beta})_{\mu_\beta}\right) \;,\\
I_{\lambda_\bullet,\lambda;\mu,\mu_\bullet} :=& I\otimes J_{\lambda_\bullet,\lambda;\mu,\mu_\bullet} \;,\\
K_{\lambda_\bullet,\lambda;\mu,\mu_\bullet} :=& I\otimes L_{\lambda_\bullet,\lambda;\mu,\mu_\bullet} \;.
\end{array}
\end{gather}

Further, let
\begin{gather}\label{For mcP}
\mc P := \NN^{t+1}\times \NN\times \NN\times\NN^{t+1}
\end{gather}
be the set of $(2t+4)$-tuple of nonnegative integers, which we convene notation-wise to split into two sequences of equal length and write as
$$
\bm l= (l_\bullet,l;m,m_\bullet)=(l_{t},...,l_0,l;m,m_0,...,m_{t}) \;,
$$
similarly to the elements of $\bm\Lambda$. We define the following families of modules parametrized by $\mc P$:
\begin{gather}\label{For LJITlm}
\begin{array}{rl}
L_{l_\bullet,l;m,m_\bullet} :=& \left(\bigotimes\limits_{\alpha=0}^t (V^*_{\alpha+1}/V^*_{\alpha})^{\otimes l_\alpha}\right)\otimes V_{l;m}\otimes \left(\bigotimes\limits_{\beta=0}^t (\bar V_{\beta+1}/\bar V_{\beta})^{\otimes m_\beta}\right)\;, \\
J_{l_\bullet,l;m,m_\bullet} := & \left(\bigotimes\limits_{\alpha=0}^t (V^*/V^*_{\alpha})^{\otimes l_\alpha}\right)\otimes (V^*)^{\otimes l} \otimes (\bar V)^{\otimes m} \otimes \left(\bigotimes\limits_{\beta=0}^t (\bar V/\bar V_{\beta})^{\otimes m_\beta}\right) \;,\\
I_{l_\bullet,l;m,m_\bullet} :=& I\otimes J_{l_\bullet,l;m,m_\bullet} \;,\\
K_{l_\bullet,l;m,m_\bullet} :=& I\otimes L_{l_\bullet,l;m,m_\bullet} \;.
\end{array}
\end{gather}

\begin{rem}\label{Rem index lm}
Notation-wise, it will be sometimes be convenient to include $l$ and $m$ into the indexed sequences $l_\bullet$ and $m_\bullet$, respectively, as $l=l_{-1}$ and $m=m_{-1}$. In other words, $\mc P$ will be interpreted alternatively as $\NN^{t+1}\times\NN\times\NN\times \NN^{t+1}$ or as $\NN^{t+2}\times\NN^{t+2}$. Correspondingly, we set $V^*_{-1}=0\subset V^*$ and $\bar V_{-1}$. The range of the index will be made clear in the context, with the initial convention as default. We do similarly for the elements of $\bm\Lambda$.
\end{rem}

There is a natural map from $\bm\Lambda$ to $\mc P$ given by component-wise size:
$$
|\cdot|:\bm\Lambda \to \mc P \;,\; (\lambda_\bullet,\lambda;\mu,\mu_\bullet)\mapsto (|\lambda_\bullet|,|\lambda|;|\mu|,|\mu_\bullet|) \;.
$$
We use the same notation for the map
$$
|\cdot|:\mc P \to \NN \;,\; (l_\bullet,l;m,m_\bullet)\mapsto l+m+\sum\limits_{\alpha=0}^t(l_\alpha+m_\alpha) \;,
$$
and we denote the composition of these two maps by
$$
||\cdot|| :\bm\Lambda\to \NN \;,\; (\lambda_\bullet,\lambda;\mu,\mu_\bullet) \mapsto |\lambda|+|\mu|+ \sum\limits_{\alpha=0}^t |\lambda_\alpha|+|\mu_\alpha| \;.
$$

We denote the symmetric group on $n$ letters by $\mk S_n$, and for $\bm l=(l_\bullet,l;m,m_\bullet)\in{\bf P}$ we let $\mk S_{\bm l}$ be the product of $2t$ symmetric groups of sizes corresponding to the entries of $\bm l$: 
$$
\mk S_{\bm l}=\mk S_{l_t}\times... \times\mk S_{l_0}\times \mk S_l\times \mk S_m\times\mk S_{l_0}\times...\times \mk S_{m_t} \;.
$$
Note that $\mk S_{\bm l}$ acts naturally on each of the modules (\ref{For LJITlm}), and
\begin{gather}\label{For LJITlm is sum LJITlambdamu}
L_{\bm l} = \bigoplus\limits_{|\bm\lambda|=\bm l} \CC^{\bm \lambda} \otimes L_{\bm\lambda}\;, \; J_{\bm l} = \bigoplus\limits_{|\bm\lambda|=\bm l} \CC^{\bm \lambda} \otimes J_{\bm\lambda} \;,\;
I_{\bm l} = \bigoplus\limits_{|\bm\lambda|=\bm l} \CC^{\bm \lambda} \otimes I_{\bm\lambda} \;,\; K_{\bm l} = \bigoplus\limits_{|\bm\lambda|=\bm l} \CC^{\bm \lambda} \otimes K_{\bm\lambda} \;,
\end{gather}
where $\CC^{\bm\lambda}$ denotes the irreducible representation of $\mk S_{\bm l}$ determined by $\bm\lambda\in\bm\Lambda$.

\subsubsection{Simple tensor modules}

\begin{theorem}\label{Theo Simples and Injhulls}
Let $\bm\lambda=(\lambda_\bullet,\lambda;\mu,\mu_\bullet)\in\bm\Lambda$ and let $L_{\bm\lambda},J_{\bm\lambda}$ be as in (\ref{For LJITlambda}). Then the $\mk{gl}^M$-module $L_{\bm\lambda}$ is simple and $L_{\bm\lambda}={\rm soc}J_{\bm\lambda}$. In particular, the inclusion $L_{\bm\lambda}\subset J_{\bm\lambda}$ is essential.
\end{theorem}

The proof will be given after some technical preparation. Let us note that the case $\mu_\bullet=\emptyset_\bullet$ is settled in \cite[Lemma~4.9]{Chirvasitu-Penkov-RC}. We shall combine this fact with a suitable generalization of \cite[Proposition~4.1 and Corollary~4.3]{Chirvasitu-Penkov-OTC} to obtain the complete result.

\begin{lemma}\label{Lemma First LinIndepToAny}
Let $v_1,...,v_p,w_1,...,w_p\in \bar V$ with $v_1,...,v_p$ linearly independent modulo $\bar V_{t}$, and $x_1,...,x_q,y_1,...,y_q\in V^*$ with $x_1,...,x_q$ linearly independent modulo $V^*_{t}$. Then there exists $\varphi\in\mk{gl}^M$ such that $\varphi(v_i)=w_i$ and $\varphi^*(x_j)=y_j$ for all $i=1,...,p$ and $j=1,...,q$.
\end{lemma}

\begin{proof}
The argument is analogous to the proof of \cite[Lemma~4.2]{Chirvasitu-Penkov-OTC}, but with transfinite recursion. Recall that $\B$ is the index set for a basis of $V$, and assume that $\B$ is ordered by the initial ordinal number with cardinality $|\B|=\aleph_t$. In particular $\B$ is well-ordered and, since the cardinals $\aleph_t$ for $t\in\NN$ are regular, for every $b\in \B$ the set $\B_{>b}:=\{c\in\B:b<c\}$ has the same cardinality $\aleph_t$ as $\B$ while $\B_{\leq b}:=\{c\in\B:c\leq b\}$ has strictly smaller cardinality. Let $M$ be the matrix with columns $v_1,...,v_p$ and $N$ be the matrix with rows $x_1,...,x_q$. For $b\in\B$, let $M(b)$ and $N(b)$ denote the corresponding row of $M$ and, respectively, column of $N$. The set of indices of the rows appearing in a given $p\times p$-minor of $M$, respectively $q\times q$-minor of $N$, will be called the support of that minor. The set of nonsingular $p\times p$-minors of $M$ has cardinality $\aleph_t$. Furthermore, it contains a subset, say $\mc M$, such that $|\mc M|=\aleph_t$ and distinct elements of $\mc M$ are supported on disjoint sets of rows of $M$. Let $\B\to\mc M, b\mapsto M_b$ be any injective map. Let $\tilde M_b$ be the $\aleph_t\times p$-matrix obtained from $M$ by replacing the minor $M_b$ by its inverse and setting all other rows equal to $0$. Similarly, there exists an injection $\mc B\to\mc N, b\mapsto N_b$ defined using the matrix $N$ and its $q\times q$-minors. We also use the analogous notation $\tilde N_b$ for the resulting $q\times \aleph_t$ matrices. The assumption on the order of $\B$ guaranties that the assignments $b\mapsto M_b$ and $b\mapsto N_b$ can be made so that for every $b\in\B$ and every $c\in {\rm supp}(M_b)\cup{\rm supp}(N_b)$ we have $b<c$.

Now we are ready to give a recursive definition of $\varphi$ as a matrix with respect to the chosen order of $\B$. Let $b_0$ be the minimal element of $\B$. Define the first row of $\varphi$ by setting $\varphi_{(b_0,a)}:=0$ for $a\notin{\rm supp}(M_{b_0})$ and 
$$
(\varphi_{(b_0,a_1)},...,\varphi_{(b_0,a_p)}) := (w_1(b_0),...,w_p(b_0))M_{b_0}^{-1}
$$
if ${\rm supp}(M_{b_0})=\{a_1,...,a_p\}$.

The first column has now its first entry $\varphi_{(b_0,b_0)}$ fixed. Put $\varphi_{(c,b_0)}:=0$ for $c\notin \{b_0\}\cup{\rm supp}(N_{b_0})$ and
$$
\begin{pmatrix} \varphi_{(c_1,b_0)}\\ \vdots\\ \varphi_{(c_q,b_0)}\end{pmatrix} :=-N_{b_0}^{-1}\left(\begin{pmatrix} y_1(b_0)\\ \vdots\\ y_q(b_0)\end{pmatrix}+ \varphi_{(b_0,b_0)} \begin{pmatrix} x_1(b_0)\\ \vdots\\ x_q(b_0)\end{pmatrix}\right) 
$$
if ${\rm supp}(N_{b_0})=\{c_1,...,c_q\}$.

Let $b\in \B$ and assume that the rows and columns of $\varphi$ are given for indices strictly smaller than $b$. To define the $b$-th row $\varphi_{(b,\cdot)}$ we extend the given data by setting $\varphi_{(b,d)}:=0$ if $d\geq b$ and $d\notin{\rm supp}(M_b)$, and
$$
(\varphi_{(b,d_1)},...,\varphi_{(b,d_p)}) :=((w_1(b),...,w_p(b))-\sum\limits_{a<b} \varphi_{(b,a)} (v_1(a),...,v_p(a))) M_b^{-1}
$$
if ${\rm supp}(M_b)=\{d_1,...,d_p\}$. Similarly, we extend the $b$-th column by $0$ outside ${\rm supp}(N_b)$ and put
$$
\begin{pmatrix} \varphi_{(e_1,b)}\\ \vdots\\ \varphi_{(e_q,b)}\end{pmatrix} :=-N_b^{-1}\left(\begin{pmatrix} y_1(b)\\ \vdots\\ y_q(b)\end{pmatrix}+\sum\limits_{a\leq b} \varphi_{(b,a)} \begin{pmatrix} x_1(a)\\ \vdots\\ x_q(a)\end{pmatrix}\right)
$$
if ${\rm supp}(N_b)=\{e_1,...,e_q\}$. The resulting matrix $\varphi$ determines an element of $\mk{gl}^M$ which satisfies the required properties by construction.
\end{proof}

\begin{lemma}\label{Lemma LinIndepToAny}

Let $0\leq\alpha_1<...<\alpha_n\leq t$ be $n$ natural numbers. For $m\in\{1,...,n\}$, let $v_1^m,...,v^m_{p_m},w_1^m,...,w_{p_m}^m\in\bar V_{m+1}/\bar V_{m}$ and $x_1^m,...,x_{q_m}^m,y_1^m,...,y_{q_m}^m\in V^*_{m+1}/V^*_{m}$ be arbitrary pairs of tuples of vectors in the respective spaces.

If the tuples $\{v^m_1,...,v^m_{p_m}\}$ and $\{x^m_1,...,x^m_{q_m}\}$ are linearly independent for every $m$, then there exists a transformation $\varphi\in\mk{gl}^M$ such that
$$
\varphi(v^m_j)=w^m_j \;\; and\;\; \varphi^*(x^m_j)=y^m_j \quad \forall j,\forall m \;.
$$
\end{lemma}

\begin{proof}
We shall reduce the statement to the case of Lemma \ref{Lemma First LinIndepToAny}. Let $\tilde v^m_j$, $\tilde w^m_j\in\bar V$ and $\tilde x^m_j,\tilde y^m_j\in V^*$ be representatives of the respective elements. We define subsets $A_1,...,A_n$ of $\B$ by setting
$$
A_m:=\cup_j({\rm supp}(\tilde v^m_j)\cap{\rm supp}(\tilde w^m_j)) \cup (\cup_j({\rm supp}(\tilde x^m_j)\cap{\rm supp}(\tilde y^m_j)))
$$
Note that $|A_m|=\aleph_{\alpha_m}$ and, by the hypothesis of linear independence, the above representatives can be chosen so that $A_m\cap A_l=\emptyset$ if $m\ne l$. As in the previous lemma, we fix a well-order of $\B$ and work with $\mk{gl}^M$ as a matrix algebra. If necessary, we change the order so that $A_1<A_2<...<A_n$. Let $\mk g^{A_m}\subset\mk{gl}^M$ be the subalgebra consisting of elements with supports in $A_m\times A_m$. Note that $\mk g^{A_m}$ is isomorphic to the Mackey Lie algebra of the restriction of the pairing ${\bf p}$ to $\CC_{A_m}\otimes \VV_{A_m}$, where $\CC_{A_m}:={\rm span}\{x_a:a\in A_m\}\subset \CC_\B:=V_*$ and $\VV_{A_m}:={\rm span}\{v_a:a\in A_m\}\subset \VV_\B:=V$. Let $\mk l =\mk g^{A_1}\oplus...\oplus\mk g^{A_n}$ be the resulting block-diagonal subalgebra of $\mk{gl}^M$, which is clearly contained in the ideal $\mk{gl}^M_{\alpha_n+1}$. Thus it suffices to show that, for every $m\in\{1,...,n\}$, there exists $\varphi_m\in\mk g^{A_m}$ such that $\varphi(v^m_j)=w^m_j$ and $\varphi^*(x^m_j)=y^m_j$ for all $j$. Furthermore, the elements $v^m_j,w^m_j$ can be seen as elements of the quotient $\VV^{A_m}/\VV^{A_m}_{\alpha_m}$ and similarly $x^m_j,y^m_j$ in $\CC^{A_m}/\CC^{A_m}_{\alpha_m}$. Thus we have $n$ occurrences of the situation of Lemma \ref{Lemma First LinIndepToAny} in distinct dimensions $\aleph_{\alpha_1},...,\aleph_{\alpha_n}$, which is the claimed reduction.
\end{proof}


\begin{lemma}\label{Lemma simple lambda00mu}
Let $\lambda_\bullet,\mu_\bullet\in\Lambda^{t+1}$. Then the $\mk{gl}^M$-module $L_{\lambda_\bullet,\emptyset;\emptyset,\mu_\bullet}$ is simple.
\end{lemma}

\begin{proof}
We follow the idea of the proof of \cite[Prop.~4.1]{Chirvasitu-Penkov-OTC} and use the simplicity of $L_{\lambda_\bullet,\emptyset;\emptyset,\emptyset_\bullet}$ (and analogously $L_{\emptyset_\bullet,\emptyset;\emptyset,\mu_\bullet}$) established in \cite[Prop.~4.2]{Chirvasitu-Penkov-RC}. We identify $L_{\lambda_\bullet,\emptyset;\emptyset,\mu_\bullet}$ with the submodule of $L_{|\lambda_\bullet|,\emptyset;\emptyset,|\mu_\bullet|}$ obtained as the image of the product of Young symmetrizers $c_{\lambda_\bullet}\otimes c_{\mu_\bullet}=(\otimes_{\alpha} c_{\lambda_\alpha})\otimes (\otimes_\alpha c_{\mu_\alpha})$, where we use the convention for $c_{\nu}$, $\nu\in\Lambda$, from Section \ref{Sec Schur}. We have the decomposition $L_{\lambda_\bullet,\emptyset;\emptyset,\mu_\bullet}=L_{\lambda_\bullet,\emptyset;\emptyset,\emptyset_\bullet}\otimes L_{\emptyset_\bullet,\emptyset;\emptyset,\mu_\bullet}$ allowing to write any $w\in L_{\lambda_\bullet,\emptyset;\emptyset,\mu_\bullet}$ as a finite sum of simple tensors $w=\sum_j x^j\otimes v^j$. By the argument of \cite[Prop.~1]{Chirvasitu-3Mackey}, the $\mk{gl}^M$-action on $L_{\lambda_\bullet,\emptyset;\emptyset,\mu_\bullet}$ allows us to reduce to the case where $A_1=\cup_j{\rm supp}(x^j)$ and $A_2=\cup_j{\rm supp}(v^j)$ are two disjoint infinite subsets of $\mc B$, and are contained in respective subsets $B_1,B_2\subset\B$ with $B_1\cap B_2=\emptyset$ and $|B_1|=|B_2|=\aleph_t$. Let $\mk l=\mk g^{B_1}\oplus\mk g^{B_2}\subset \mk g^\B=\mk{gl}^M$ be the corresponding block-diagonal subalgebra isomorphic to $\mk{gl}^M\oplus\mk{gl}^M$. The Lie algebra $\mk l$ acts on the space of elements of $L_{\lambda_\bullet,\emptyset;\emptyset,\emptyset_\bullet}$ supported on $B_1$, as well as on the space of elements of $L_{\emptyset_\bullet,\emptyset;\emptyset,\mu_\bullet}$ supported on $B_2$. But $L_{\lambda_\bullet,\emptyset;\emptyset,\emptyset_\bullet}$ and $L_{\emptyset_\bullet,\emptyset;\emptyset,\mu_\bullet}$ are simple $\mk{gl}^M$-modules, hence the $\mk{gl}^M$-submodule of $L_{\lambda_\bullet,\emptyset;\emptyset,\emptyset_\bullet}\otimes L_{\emptyset_\bullet,\emptyset;\emptyset,\mu_\bullet}$ containing $w$ also contains all elements of the form $c_{\lambda_\bullet} x \otimes c_{\mu_\bullet} v$, with $x\in L_{|\lambda_\bullet|,0;0,0_\bullet}$ and $v\in L_{0_\bullet,0;0,|\mu_\bullet|}$ supported respectively on $B_1$ and $B_2$.

By repeatedly applying Lemma \ref{Lemma LinIndepToAny} to linearly independent sets consisting of individual tensorands of the elements $c_{\lambda_\bullet} x\in L_{|\lambda_\bullet|,0;0,0_\bullet}$ and $c_{\mu_\bullet} v \in L_{0_\bullet,0;0,|\mu_\bullet|}$, we deduce that the entire module $L_{\lambda_\bullet,\emptyset;\emptyset,\mu_\bullet} = c_{\lambda_\bullet}(L_{|\lambda_\bullet|,0;0,0_\bullet})\otimes c_{\mu_\bullet}(L_{0_\bullet,0;0,|\mu_\bullet|})$ is contained in the submodule generated by $w$.
\end{proof}

\begin{proof}[Proof of Theorem \ref{Theo Simples and Injhulls}]
To verify the simplicity of $L_{\lambda_\bullet,\lambda;\mu,\mu_\bullet}$ we start with the decomposition $L_{\lambda_\bullet,\lambda;\mu,\mu_\bullet}\cong L_{\lambda_\bullet,\emptyset;\emptyset,\mu_\bullet}\otimes L_{\lambda;\mu}$, where both tensorands are simple due to Lemma \ref{Lemma simple lambda00mu} and Theorem \ref{Theo Tsmall is ord Groth}. Now, Proposition \ref{Prop DenseVlambdamu PS} allows us to invoke Proposition \ref{Prop DenseProperties} for $\mk G=\mk{gl}^M$, $\mk J=\mk g_\B$, $R=L_{\lambda;\mu}=V_{\lambda;\mu}$, and apply the functor $\bullet\otimes R$ applied to the simple module $L_{\lambda_\bullet,\emptyset;\emptyset,\mu_\bullet}$. This confirms that $L_{\lambda_\bullet,\lambda;\mu,\mu_\bullet}$ is simple. Finally, the inclusion $L_{\lambda_\bullet,\lambda;\mu,\mu_\bullet}\subset J_{\lambda_\bullet,\lambda;\mu,\mu_\bullet}$ is essential as a consequence of Lemma \ref{Lemma LinIndepToAny} and the fact that $V_{\lambda;\mu}$ is essential in $V^*_{\lambda}\otimes \bar V_{\mu}$.
\end{proof}

\begin{theorem}\label{Theo EndoSimpl}
The simple modules $L_{\lambda_\bullet,\lambda;\mu,\mu_\bullet}$ are pairwise nonisomorphic and have scalar endomorphism algebras, ${\rm End}L_{\lambda_\bullet,\lambda;\mu,\mu_\bullet}\cong\CC$.
\end{theorem}

\begin{proof}
There are known cases of the theorem, as follows. The case $\mu_\bullet=\emptyset_\bullet$ (and by analogy the case $\lambda_\bullet=\emptyset_\bullet$) is proven in \cite[Proposition~4.2]{Chirvasitu-Penkov-RC}. The case $t=0$ is proven in \cite[Theorem~3.6]{Chirvasitu-Penkov-UTC}. A combination of the two methods of proof yields the general result.
\end{proof}

\begin{prop}\label{Prop TT tensor prod simples}
Let $(\lambda_\bullet,\lambda;\mu,\mu_\bullet),(\lambda_\bullet',\lambda';\mu',\mu'_\bullet)\in\bm\Lambda$.
\begin{enumerate}
\item[{\rm (a)}] The simple module $L_{\lambda_\bullet,\lambda;\mu,\mu_\bullet}$ is pure if and only if either just the two inner diagrams $\lambda,\mu$ are nonempty, or all diagrams except at most one are empty.
\item[{\rm (b)}] For the layers of the socle filtration of the tensor product $L_{\lambda_\bullet,\lambda;\mu,\mu_\bullet}\otimes L_{\lambda_\bullet',\lambda';\mu',\mu'_\bullet}$ we have
\begin{align*}
& \ul{\rm soc}^{k+1}(L_{\lambda_\bullet,\lambda;\mu,\mu_\bullet}\otimes L_{\lambda_\bullet',\lambda';\mu',\mu'_\bullet}) \cong L_{\lambda_\bullet,\emptyset;\emptyset,\mu_\bullet}\otimes L_{\lambda_\bullet',\emptyset;\emptyset,\mu'_\bullet} \otimes \ul{\rm soc}^k(V_{\lambda;\mu}\otimes V_{\lambda';\mu'}) \\
& \qquad\quad\qquad \cong \bigoplus\limits_{(\kappa_\bullet,\kappa,\nu,\nu_\bullet)\in\bm\Lambda:|\lambda|+|\lambda'|-|\kappa|=k} \left({\bf N}^{\kappa_\bullet}_{\lambda_\bullet\lambda'_\bullet} {\bf N}^{\nu_\bullet}_{\mu_\bullet\mu'_\bullet}\sum\limits_{\xi,\eta\in\Lambda} N^{\xi}_{\lambda\lambda'}h^{\xi;\eta}_{\kappa;\nu} N^{\eta}_{\mu\mu'} \right) \cdot L_{\kappa_\bullet,\kappa;\nu,\nu_\bullet}\;,
\end{align*}
where ${\bf N}^{\zeta_\bullet}_{\xi_\bullet\eta_\bullet} := \prod\limits_{\alpha=0}^t N^{\zeta_\alpha}_{\xi_\alpha\eta_\alpha}$ for $\xi_\bullet,\eta_\bullet,\zeta_\bullet\in\Lambda^{t+1}$ and the numbers $h^{\lambda;\mu}_{\kappa;\nu}$ are given in (\ref{For h lambdamuxieta}).
\item[{\rm (c)}] The tensor product $L_{\lambda_\bullet,\lambda;\mu,\mu_\bullet}\otimes L_{\lambda_\bullet',\lambda';\mu',\mu'_\bullet}$ is a semisimple module if and only if at least one of the following four conditions holds:
$$
\lambda=\mu=\emptyset ; \lambda=\lambda'=\emptyset ; \lambda'=\mu'=\emptyset ; \mu=\mu'=\emptyset \;.
$$
\item[{\rm (d)}] If the tensor product $L_{\lambda_\bullet,\lambda;\mu,\mu_\bullet}\otimes L_{\lambda_\bullet',\lambda';\mu',\mu'_\bullet}$ is semisimple, then
$$
L_{\lambda_\bullet,\lambda;\mu,\mu_\bullet}\otimes L_{\lambda_\bullet',\lambda';\mu',\mu'_\bullet} \cong \bigoplus\limits_{(\kappa_\bullet,\kappa;\nu,\nu_\bullet)\in\bm\Lambda} {\bf N}^{\kappa_\bullet}_{\lambda_\bullet\lambda'_\bullet} N^{\kappa}_{\lambda\lambda'} N^{\nu}_{\mu\mu'} {\bf N}^{\nu_\bullet}_{\mu_\bullet\mu'_\bullet} \cdot L_{\kappa_\bullet,\kappa;\nu,\nu_\bullet} \;.
$$
\end{enumerate}
\end{prop}

\begin{proof}
Part (a) follows immediately from the classification of simple modules. For the rest of the statements, we note that there is a decomposition
$$
L_{\lambda_\bullet,\emptyset;\emptyset,\mu_{\bullet}}\otimes L_{\lambda_\bullet',\emptyset;\emptyset,\mu'_{\bullet}} \cong \bigoplus\limits_{(\kappa_\bullet;\nu_\bullet)\in\Lambda^{t+1}\times\Lambda^{t+1}} {\bf N}^{\kappa_\bullet}_{\lambda_\bullet\lambda_\bullet'}{\bf N}^{\nu_\bullet}_{\mu_{\bullet}\mu'_{\bullet}} \cdot L_{\kappa_\bullet,\emptyset;\emptyset,\nu_\bullet} \;.
$$
Indeed, this decomposition holds over the Lie algebra
$$
(\bigoplus\limits_{\alpha=0}^t \mk{gl}(V^*_{\alpha+1}/V^*_\alpha))\oplus(\bigoplus\limits_{\beta=0}^t \mk{gl}(\bar V_{\beta+1}/\bar V_\beta))
$$
and, by Theorem \ref{Theo Simples and Injhulls}, remains unchanged after restriction to $\mk{gl}^M$. Hence the module $L_{\lambda_\bullet,\emptyset;\emptyset,\mu_{\bullet}}\otimes L_{\lambda_\bullet',\emptyset;\emptyset,\mu'_{\bullet}}$ is semisimple. Furthermore, again by Theorem \ref{Theo Simples and Injhulls}, the tensor product $L_{\kappa_\bullet,\emptyset;\emptyset,\nu_\bullet}\otimes V_{\xi;\eta}$ is isomorphic to $L_{\kappa_\bullet,\xi;\eta,\nu_\bullet}$ and remains simple.

Next, observe that essential extensions between submodules of $V_*^{\otimes l}\otimes V^{\otimes m}$ remain essential after tensoring by $L_{\kappa_\bullet,\emptyset;\emptyset,\nu_\bullet}$, because this holds for the restriction of these representations to $\mk{gl}(V,V_*)$ which acts trivially on $L_{\kappa_\bullet,\emptyset;\emptyset,\nu_\bullet}$. With these observations, and the results of \S~\ref{Sec Vstar l po V m}, the deduction of statements of the proposition is immediate.
\end{proof}

The semisimplicity of the tensor products of ``one-sided'' simple modules, i.e., $L_{\lambda_\bullet,\lambda;\emptyset}\otimes L_{{_\bullet}\mu,\mu;\emptyset}$ or $L_{\emptyset;\lambda,\lambda_\bullet}\otimes L_{\emptyset;\mu,\mu_\bullet}$, as well as the obvious symmetry between the two cases, prompts us to introduce the following notation for any $\lambda_\bullet,\mu^{(1)}_\bullet,...,\mu^{(m)}_\bullet\in\Lambda^{t+1}$:
\begin{gather}\label{For bfN many}
\begin{array}{rl}
{\bf N}^{\lambda_\bullet}_{\mu^{(1)}_\bullet...\mu^{(m)}_\bullet} :=& \dim{\rm Hom}(V_{\lambda_\bullet,\emptyset;\emptyset},V_{\mu^{(1)}_\bullet,\emptyset;\emptyset}\otimes\dots\otimes V_{\mu^{(m)}_\bullet,\emptyset;\emptyset}) \\ =& \dim{\rm Hom}(V_{\emptyset;\emptyset,\lambda_\bullet},V_{\emptyset;\emptyset,\mu^{(1)}_\bullet}\otimes\dots\otimes V_{\emptyset;\emptyset,\mu^{(m)}_\bullet}) \\ =& \sum\limits_{\sigma^{(1)}_\bullet,...,\sigma^{(m-2)}_\bullet\in\Lambda^{t+1}} \prod\limits_{\alpha} N^{\lambda_\alpha}_{\mu^{(1)}_\alpha\sigma^{(1)}_\alpha} ( \prod\limits_{r=2}^{m-2} N^{\sigma^{(r-1)}_\alpha}_{\mu^{(r)}_\alpha\sigma^{(r)}_\alpha} ) N^{\sigma^{(m-2)}_\alpha}_{\mu^{(m-1)}_\alpha\mu^{(m)}_\alpha} \;.
\end{array}
\end{gather}

\subsubsection{Two orders and a family of morphisms}

Here we introduce two partial orders on the set $\mc P$ defined in (\ref{For mcP}). For $l_\bullet=(l_0,...,l_t)\in\NN^{t+1}$, we denote $|l_\bullet|=\sum\limits_{\alpha=0}^t l_\alpha$ and, for $\beta\leq t$, $|l_{\bullet_{\geq\beta}}|=\sum\limits_{\alpha=\beta}^t l_\alpha$.

\begin{definition}\label{Def Poset P}
Let $\preceq$ be the partial order on $\mc P$ defined by
\begin{gather*}
(l_\bullet,l;m,m_\bullet)\preceq (l'_\bullet,l;m,m'_\bullet) \quad \tst \quad \begin{array}{|l} l-m+|l_\bullet| - |m_\bullet|=l'-m'+|l'_\bullet| - |m'_\bullet| \\ l\leq l' \;,\; m\leq m' \\ |l_{\bullet_{\geq\beta}}| \geq |l'_{\bullet_{\geq\beta}}| \quad for \quad \beta\in\{0,...,t\} \\ |m_{\bullet_{\geq\beta}}| \geq |m'_{\bullet_{\geq\beta}}|  \end{array}\;.
\end{gather*}
From now on, $(\mc P,\preceq)$ denotes the resulting poset.
\end{definition}

\begin{definition}\label{Def Poset IP}
We define a partial order $\stackrel{\bf P}{\preceq}$ on the set $\mc P$ by strengthening the relation $\preceq$ with the additional requirements $l+|l_\bullet| \leq l'+|l'_\bullet|$, $m+|m_\bullet|\leq m'+|m'_\bullet|$, and denote the resulting poset by ${\bf P}$.
\end{definition}

Next, we define several attributes of a fixed element $\bm l=(l_\bullet,l;m,m_\bullet)$ of the set $\mc P$. There are two parallel constructions corresponding to the partial orders $\preceq$ and $\stackrel{\bf P}{\preceq}$. We begin with the notation
\begin{gather*}
\mc P(\bm l):= \{\bm k\in\mc P: \bm k\preceq \bm l\} \;,\quad 
{\bf P}(\bm l):= \{\bm k\in{\bf P}: \bm k\stackrel{\bf P}{\preceq} \bm l\}\;.
\end{gather*}

\begin{rem}\label{Rem mcPl has finite assend chains}
\begin{enumerate}
\item Both posets $\mc P(\bm l)$ and ${\bf P}(\bm l)$ have the following property: every strictly ascending sequence is finite.
\item The common underlying set $\NN^{2(t+2)}$ of the posets $\mc P$ and ${\bf P}$ is a monoid under component-wise addition. If $\bm l\preceq \bm k$ and $\bm l'\preceq\bm k'$ then $\bm l+\bm l'\preceq\bm k+\bm k'$. The same property holds for $\stackrel{\bf P}{\preceq}$.
\end{enumerate}
\end{rem}

\begin{lemma}\label{Lemma P1 and IP1}
For $\bm l= (l_\bullet,l;m,m_\bullet)\in\mc P$ and $\mc P^1(\bm l)$ be the set of elements obtained from $\bm l$ by one of the following elementary alterations:
\begin{enumerate}
\item[{\rm (i)}] if $l>0$ (resp., $l_\alpha>0$), subtract $1$ from $l$ (resp., from $l_\alpha$) and add $1$ to $l_0$ (resp., $l_{\alpha+1})$;
\item[{\rm (ii)}] if $m>0$ (resp., $m_\alpha>0$), subtract $1$ from $m$ (resp., from $m_\alpha$) and add $1$ to $m_0$ (resp., $m_{\alpha+1})$;
\item[{\rm (iii)}] if both $l$ and $m$ are positive, subtract $1$ from each of them; 
\item[{\rm (iv)}] add $1$ to both $l_0$ and $m_0$.
\end{enumerate}
Then $\mc P^1(\bm l)$ is the set of maximal elements of the poset $\mc P(\bm l)\setminus\{\bm l\}$.

Let ${\bf P}^1(\bm l)$ be the subset of $\mc P^1(\bm l)$ obtained using only ${\rm(i)},{\rm(ii)}$ and ${\rm(iii)}$. Then ${\bf P}^1(\bm l)$ is the set of maximal elements of the poset ${\bf P}(\bm l)\setminus\{\bm l\}$.
\end{lemma}

\begin{proof}
For any $\bm k\prec\bm l$ it is straightforward to construct an element $\bm k'$ obtained from $\bm l$ by using one of the alterations (i)-(iv) and satisfying $\bm k\preceq\bm k'\prec\bm l$. This proves the statement for $\mc P^1(\bm l)$. The statement for ${\bf P}^1(\bm l)$ is proven analogously.
\end{proof}

\begin{definition}\label{Def Pqp}
More generally, for $q\geq 1$ let $\mc P^q(\bm l)\subset\mc P(\bm l)$ be the set of maximal elements of the set 
$$
\{\bm k\in\mc P:\bm k\prec\bm l\}\setminus \left(\bigcup\limits_{j< q} \mc P^{j}(\bm l)\right) \;,
$$
with the convention $\mc P^0(\bm l)=\{\bm l\}$. We define ${\bf P}^q(\bm l)$ analogously.
\end{definition}

To any element $\bm l= (l_\bullet,l;m,m_\bullet)$ we associate the number
\begin{gather}\label{For qlm}
q^{(\bm l)}:=(l+m)(t+1)+\sum\limits_{j=0}^{t} (l_j+m_j)(t-j) . 
\end{gather}
Note that $q^{(\bm l)}=0$ if and only if $\bm l$ has the form $\bm l=(l_{t},0,...,0;0,...0,m_t)$. Lemma \ref{Lemma P1 and IP1} and Remark \ref{Rem mcPl has finite assend chains} imply that ${\bf P}(\bm l)$ and $\mc P(\bm l)$ split as disjoint unions: 
\begin{gather}\label{For bfPl and mcPl disjU}
{\bf P}(\bm l)=\bigsqcup\limits_{0\leq q\leq q^{(\bm l)}}{\bf P}^q(\bm l) \;,\quad {\mc P}(\bm l)=\bigsqcup\limits_{q\in\NN}{\mc P}^q(\bm l)\;.
\end{gather}
This structure behaves well with respect to addition in the underlying monoid, i.e.,
\begin{gather}\label{For Pq additive}
{\bf P}^{q}(\bm l)+{\bf P}^{q'}(\bm l') \subset {\bf P}^{q+q'}(\bm l+\bm l')\;,\;\; {\rm for}\;\;\bm l,\bm l'\in{\bf P}, q,q'\in\NN.
\end{gather}
Furthermore, if $\bm l=(l_\bullet,l;m,m_\bullet)$ then
\begin{gather}\label{For Pql compute}
\begin{array}{rl}
{\bf P}^{q}(\bm l) & = \bigcup\limits_{i+j+k=q} {\bf P}^{i}(l_\bullet,0;0_\bullet)+ {\bf P}^{j}(l;m) + {\bf P}^{k}(0_\bullet;0,m_\bullet) \\
& =\bigcup\limits_{|i_\bullet|+j+|k_\bullet|=q} \left( {\bf P}^{j}(l;m) + \sum\limits_{\alpha=0}^t ({\bf P}^{i_\alpha l_\alpha}(1_\alpha,0;0_\bullet)+ {\bf P}^{k_\alpha m_\alpha}(0_\bullet;0,1_\alpha)) \right) \;.
\end{array}
\end{gather}
Properties (\ref{For Pq additive}) and (\ref{For Pql compute}) hold as well for the poset $\mc P$ instead of ${\bf P}$.

We now move our attention to morphisms $f:I_{\bm l}\to I_{\bm k}$. The notation below refers to subtraction of elements in the monoid $\mc P$, and we automatically assume that the parameters satisfy the inequalities necessary for the results to be in $\mc P$.

\begin{definition}\label{Def Xibml}
For $\bm l=(l_\bullet,l;m,n_\bullet)\in\mc P$, we let $\Xi^1(I_{\bm l})$ be the set of morphisms $I_{\bm l}\to I_{\bm k}$ with $\bm k\in\mc P^{1}(\bm l)$, associated with the four types of elements of $\mc P^{1}(\bm l)$ according to Lemma \ref{Lemma P1 and IP1}, as follows:
\begin{enumerate}
\item[{\rm (i)}] $f^{(\alpha)}_{j}:I_{\bm l}\to I_{\bm l+(1_\alpha,-1_{\alpha-1};0)}$ is the projection $V^*/V^*_{\alpha-1}\to V^*/V^*_{\alpha}$ applied to the $j$-th tensorand in $(V^*/V^*_{\alpha-1})^{\otimes l_{\alpha-1}}$, extended by identity on all other tensorands in $I_{\bm l}$, for $0\leq\alpha\leq t$ and $0\leq j\leq l_{\alpha-1}$;
\item[{\rm (ii)}] $\bar f^{(\alpha)}_{j}:I_{\bm l}\to I_{\bm l +(0;-1_{\alpha-1},1_{\alpha})}$ is the projection $\bar V/\bar V_{\alpha-1}\to \bar V/\bar V_{\alpha}$ applied to the $j$-th tensorand in $(\bar V/\bar V_{\alpha-1})^{\otimes m_{\alpha-1}}$, extended by identity on all other tensorands in $I_{\bm l}$, for $0\leq\alpha\leq t$ and $0\leq j\leq m_{\alpha-1}$; 
\item[{\rm (iii)}] $\tilde{\bf p}_{i,j}:I_{\bm l}\to I_{\bm l-(1;1)}$ is the morphism $\tilde{\bf p}:V^*\otimes\bar V \to Q \subset I$ applied to the relevant pair of tensorands $V^*$ and $\bar V$ in $I_{\bm l}$, extended by identity on all other tensorands, for $0\leq i\leq l$ and $0\leq j\leq m$;
\item[{\rm (iv)}] $\psi_{\bm l}:I_{\bm l}\to I_{\bm l+(1,0;0,1)}$ is the morphism $\psi:I \to (V^*/V_*)\otimes(\bar V/V)\otimes I=I_{1,0;0,1}$ applied to the tensorand $I$, extended by the identity to all other tensorands in $I_{\bm l}$.
\end{enumerate}
We let $\Xi^q(I_{\bm l})$ be the set of morphisms $I_{\bm l}\to I_{\bm k}$ with $\bm k\in\mc P^{q}(\bm l)$ obtained as compositions $f_q\circ...\circ f_1$, where $f_j\in\Xi^1(I_{\bm k_j})$ for some decreasing sequence $\bm l=\bm k_0\succ\bm k_1\succ...\succ\bm k_q=\mk k$ satisfying $\bm k_j\in\mc P^1(\bm k_{j-1})$ for $j=1,...,q$.
\end{definition}

We also introduce a family of morphisms with domain $J_{\bm l}$. Since ${\rm soc}I=\CC$ there is a canonical embedding $J_{\bm l}\subset I\otimes J_{\bm l}= I_{\bm l}$. Let $\Xi^q(J_{\bm l})$ be the set of restrictions to $J_{\bm l}$ of morphisms from $\Xi^q(I_{\bm l})$ which are obtained as compositions of morphisms of type (i),(ii) and (iii). The codomains of these morphisms are of the form $I_{\bm k}$ with $\bm k\in{\bf P}^q(\bm l)$.

\subsubsection{The socle filtrations of the modules $J_{l_\bullet,l;m,m_\bullet}$ and $J_{\lambda_\bullet,\lambda;\mu,\mu_\bullet}$}\label{Sec Jlm}

Here we study the families of modules $J_{\bm l}$ and $J_{\bm\lambda}$ defined in (\ref{For LJITlm}). We begin with the former family, and the observation that, for $\bm l,\bm l'\in\mc P$, there is an isomorphism
$$
J_{\bm{l}}\otimes J_{\bm l'} \cong J_{\bm l+\bm l'} \;.
$$

\begin{example}\label{Exa socJ11}
Let us consider $J_{1;1}=V^*\otimes \bar V$ and describe its socle filtration. From Theorem \ref{Theo Simples and Injhulls} we get
$$
{\rm soc}(V^*\otimes \bar V)= {\rm soc}(V_*\otimes V) = L_{1;1} ={\rm ker}{\bf p}= \mk{sl}(V,V_*) \;,
$$
and observe that
$$
{\rm soc}(V^*\otimes \bar V) = \bigcap\limits_{f\in\Xi^1(J_{1;1})} {\rm ker}f \;.
$$
Further, we have ${\rm soc}^{j+1}(V^*)=V^*_j$ and ${\rm soc}^{j+1}(\bar V)=\bar V_j$ for $j=0,...,t+1$, and consequently
\begin{gather}\label{For socqJ11}
{\rm soc}^{q+1}(V^*\otimes \bar V)\subset \sum\limits_{i+j=q} {\rm soc}^{i+1}(V^*) \otimes {\rm soc}^{j+1}(\bar V)= \sum\limits_{i+j=q} V^*_i\otimes\bar V_j
\end{gather}
for $q\geq 0$. For $q\geq 1$, the containment in (\ref{For socqJ11}) is an equality, as can be shown by induction. Thus the length of the socle filtration of $V^*\otimes \bar V$ is $2(t+1)+1$, and
$$
{\rm soc}^{q+1}(V^*\otimes \bar V) = \bigcap\limits_{f\in\Xi^{q+1}(J_{1;1})} {\rm ker}f \;.
$$
For the higher layers of the socle filtration we obtain
$$
\ul{\rm soc}^2(V^*\otimes \bar V)= ((V^*_1/V^*_0)\otimes \bar V_0) \oplus \CC \oplus (V^*_0\otimes (\bar V_1/\bar V_0)) \cong L_{1,0;1}\oplus L_{0;0}\oplus L_{1;0,1} \;,
$$
and
$$
\ul{\rm soc}^{q+1}(V^*\otimes\bar V)= \bigoplus\limits_{i+j=q} \ul{\rm soc}^{i+1}(V^*)\otimes \ul{\rm soc}^{j+1}(\bar V) = \bigoplus\limits_{i+j=q} L_{1_{i-1};1_{j-1}}
$$
for $q\geq 2$. 
\end{example}

\begin{prop}\label{Prop socqJlm}
Let $\bm l=(l_\bullet,l;m,m_\bullet)\in{\bf P}$ and $p=\min\{l,m\}$. The socle filtration of $J_{\bm l}$ has length $1+q^{(\bm l)}$, see (\ref{For qlm}). For $0\leq q\leq q^{(\bm l)}$, we have
\begin{gather*}
\begin{array}{rl}
{\rm soc}^{q+1}J_{\bm l}& = \bigcap\limits_{f\in\Xi^{q+1}(J_{\bm l})} {\rm ker}f \\
\ul{\rm soc}^{q+1}J_{\bm l}& \cong \bigoplus\limits_{i+j+k=q,k\leq p} \binom{l}{k}\binom{m}{k} {\rm soc}(\ul{\rm soc}^{i+1}J_{l_\bullet,l-k;0}\otimes \ul{\rm soc}^{j+1}J_{0;m-k,m_\bullet}) \\
& \cong \bigoplus\limits_{i+j+|i_\bullet|+|j_\bullet|+k=q} \binom{l}{k}\binom{m}{k} ( {\rm soc}(\ul{\rm soc}^{i+1}((V^*)^{\otimes(l-k)})\otimes\ul{\rm soc}^{j+1}(\bar V^{\otimes(m-k)})) \otimes \\
& \qquad\quad {^{k\leq p}} \qquad\qquad\qquad \otimes (\bigotimes\limits_{\alpha,\beta=0}^t\ul{\rm soc}^{i_{\alpha}+1}J_{l_\alpha,0;0}\otimes \ul{\rm soc}^{j_\beta+1}J_{0;0,m_\beta}) ) \\
& \cong \bigoplus\limits_{\bm k\in{\bf P}^q(\bm l)} {\bf b}^{\bm l}_{\bm k}  L_{\bm k} \;,
\end{array}
\end{gather*}
where ${\bf P}^q(\bm l)$ is as in Definition \ref{Def Pqp} and, for $\bm k\in{\bf P}^q(\bm l)$,
$$
{\bf b}^{\bm l}_{\bm k}:=\sum\limits_{\begin{array}{c}k+\sum\limits_{\alpha=-1}^t (q_\alpha+\bar q_\alpha)=q \\
\sum\limits_{-1\leq\alpha\leq t} (r^{(\alpha)}_\bullet;s^{(\alpha)}_\bullet)=\bm k\end{array}} \binom{l}{k}\binom{m}{k} \frac{(l-k)!}{\prod\limits_{-1\leq \beta\leq t}r^{(-1)}_\beta!}\frac{(m-k)!}{\prod\limits_{-1\leq \beta\leq t}s^{(-1)}_\beta!} \prod\limits_{0\leq\alpha\leq t} \frac{l_\alpha!}{\prod\limits_{0\leq \beta\leq t}r^{(\alpha)}_\beta!} \frac{m_\alpha!}{\prod\limits_{0\leq \beta\leq t}s^{(\alpha)}_\beta!} \;,
$$
the sum running over all sets of integers $k,q_{-1},...,q_t,\bar q_{-1},...,\bar q_t\in\NN$, $k\leq p$, and all sets of elements $(r^{(-1)}_\bullet;s^{(-1)}_\bullet),...,(r^{(t)}_\bullet;s^{(t)}_\bullet)\in{\bf P}$ satisfying, in addition to the above equalities, $(r^{(-1)}_\bullet;0_\bullet)\in{\bf P}^{q_{-1}}(l-k;0_\bullet)$, $(0_\bullet;s^{(-1)}_\bullet)\in{\bf P}^{\bar q_{-1}}(0_\bullet;m-k)$ and $(r^{(\alpha)}_\bullet;0_\bullet)\in{\bf P}^{q_\alpha}(l_\alpha;0_\bullet),(0_\bullet;s^{(\alpha)}_\bullet)\in{\bf P}^{\bar q_\alpha}(0_\bullet;m_\alpha)$ for $0\leq \alpha\leq t$.
\end{prop}

\begin{proof}
From (\ref{For LJITlm is sum LJITlambdamu}) and Theorem \ref{Theo Simples and Injhulls} we obtain
$$
{\rm soc}J_{\bm l} = L_{\bm l} = \bigcap\limits_{f\in\Xi^{1}(J_{\bm l})} {\rm ker}f = \left(\bigotimes\limits_{\alpha=0}^t (V^*_{\alpha+1}/V^*_\alpha)^{\otimes l_\alpha} \right) \otimes V_{l;m} \otimes \left( \bigotimes\limits_{\alpha=0}^t (\bar V_{\alpha+1}/\bar V_\alpha)^{\otimes m_\alpha} \right).
$$
We have a corresponding decomposition of $J_{l_\bullet,l;m,m_\bullet}$, whose tensorands are essential extensions of the respective tensorands of ${\rm soc}J_{l_\bullet,l;m,m_\bullet}$:
\begin{gather}\label{For Jlm is a product of pures}
J_{l_\bullet,l;m,m_\bullet} \cong \left(\bigotimes\limits_{\alpha=0}^t J_{l_\alpha,0;0_\bullet} \right) \otimes J_{l;m} \otimes \left( \bigotimes\limits_{\alpha=0}^t J_{0_\bullet;0,m_\alpha} \right).
\end{gather}
We can split the proof of the proposition into two steps: first, verify the statement for each of the above tensorands, and second, show that the tensor product of the resulting filtrations yields the claimed socle filtration of $J_{\bm l}$. For both steps, we use K\"unneth-type products of the socle fitrations of the relevant tensorands. The key observation is that no simple constituent descends to a lower layer than expected. This follows from the density statement of Lemma \ref{Lemma LinIndepToAny} which allows us to apply Proposition \ref{Prop DenseProperties} to the ideals $\mk{gl}^M_\alpha$ and to the relevant extensions.

The three types of elements in ${\bf P}^{1}(\bm l)$, given as (i),(ii),(iii) in Lemma \ref{Lemma P1 and IP1}, correspond to the types (i),(ii),(iii) of morphisms in $\Xi^{1}(J_{\bm l})$ (see the discussion under Definition \ref{Def Xibml}). The modules of type $J_{l_\alpha,0;0_\bullet}$ can only be the domain of morphisms of type (i). By (\ref{For Pql compute}), the set ${\bf P}^{q}(l_\alpha,0;0_\bullet)$ consists of elements of the form $\sum\limits_{i=1}^{l_\alpha}(1_{\alpha+q_i},0;0_\bullet)$ with $q_1,...,q_{l_\alpha}\in\NN$ satisfying $\sum_j q_j=q$. We have
\begin{align*}
{\rm soc}^{q+1}J_{l_\alpha,0;0_\bullet} & = \sum\limits_{q_1+...+q_{l_\alpha}=q}\bigotimes\limits_{i=1}^{l_\alpha} {\rm soc}^{q_i+1}(V^*/V^*_\alpha) = \sum\limits_{q_1+...+q_{l_\alpha}=q}\bigotimes\limits_{i=1}^{l_\alpha} V^*_{\alpha+q_i+1}/V^*_\alpha \\ & = \bigcap\limits_{f\in\Xi^{q+1}(J_{l_\alpha,0;0_\bullet})} {\rm ker}f \;,\\
\ul{\rm soc}^{q+1}J_{l_\alpha,0;0_\bullet} & \cong \bigoplus\limits_{q_1+...+q_{l_\alpha}=q}\bigotimes\limits_{i=1}^{l_\alpha} \ul{\rm soc}^{q_i+1}(V^*/V^*_\alpha) \cong \bigoplus\limits_{q_1+...+q_{l_\alpha}=q}\bigotimes\limits_{i=1}^{l_\alpha} V^*_{\alpha+q_i+1}/V^*_{\alpha+q_i} \\ & \cong \bigoplus\limits_{(k_\bullet,0;0_\bullet)\in{\bf P}^q(l_\alpha,0;0_\bullet)} \left( \frac{l_\alpha!}{\prod\limits_{\alpha\leq \beta\leq t}k_\beta!} \right)  \cdot L_{k_\bullet,0;0_\bullet} \;.
\end{align*}
The situation with the modules $J_{0_\bullet;0,m_\alpha}$ is completely analogous, with $\Xi^q(J_{0_\bullet;0,m_\alpha})$ consisting of superpositions of morphisms of type (ii). The tensorand $J_{l;m}$ can be the domain of all three types (i),(ii),(iii) of morphisms in $\Xi^1(J_{l;m})$, as long as both $l,m$ are nonzero. We handle the morphisms of type (iii) involving $V^*\otimes \bar V$ using Example \ref{Exa socJ11}. 

As indicated above, Proposition \ref{Prop DenseProperties} implies that the socle filtration of $J_{l_\bullet,l;m,m_\bullet}$ is obtained as the K\"unneth product of the socle filtrations of the three modules $J_{l_\bullet,0;0_\bullet}$, $J_{l;m}$, $J_{0_\bullet;0,m_\bullet}$. The formula for the multiplicities follows by a standard counting argument.
\end{proof}

Next we turn our attention to the socle filtrations of the modules $J_{\bm\lambda}$ for $\bm\lambda=(\lambda_\bullet,\lambda;\mu,\mu_\bullet)\in\bm\Lambda$ defined in (\ref{For LJITlambda}). It was shown in Theorem \ref{Theo Simples and Injhulls} that $J_{\bm\lambda}$ is an essential extension of the simple module $L_{\bm\lambda}$. We observe that $J_{\bm\lambda}$ splits as a tensor product along the components of $\bm\lambda$:
$$
J_{\bm\lambda} = J_{\lambda;\emptyset}\otimes J_{\emptyset;\mu}\otimes (\bigotimes\limits_{\alpha=0}^t J_{\lambda_\alpha,\emptyset;\emptyset}\otimes J_{\emptyset;\emptyset,\mu_\alpha})\;.
$$
With this in mind, we shall successively compute the socle filtrations of $J_{\lambda_\alpha,\emptyset;\emptyset}$, $J_{\lambda_\bullet,\lambda;\emptyset}$ and $J_{\lambda_\bullet,\lambda;\mu,\mu_\bullet}$.

\begin{lemma}\label{Lemma soc q Jlambda}
Let $\lambda_\bullet=(\lambda_t,...,\lambda_{-1})\in\Lambda^{t+2}$. Then the length of the socle filtration of $J_{\lambda_\bullet;\emptyset_\bullet}$ is $1+q^{(|\lambda_\bullet|;0_\bullet)}$ and the layers are
\begin{gather*}
\ul{\rm soc}^{q+1}J_{\lambda_\bullet;\emptyset_\bullet} \cong \bigoplus\limits_{\kappa_\bullet\in\Lambda^{t+2}:(|\kappa_\bullet|;0_\bullet)\in{\bf P}^q(|\lambda_\bullet|;0_\bullet)} z^{\lambda_{\bullet}}_{\kappa_\bullet} \cdot L_{\kappa_\bullet;\emptyset_\bullet} \;,
\end{gather*}
where
$$
z^{\lambda_\bullet}_{\kappa_\bullet} := \sum\limits_{\rho^{(-1)}_\bullet,...,\rho^{(t)}_\bullet\in\Lambda^{t+2}:(|\rho^{(\beta)}_\bullet|;0_\bullet)\in{\bf P}^{j_\beta}(|\lambda_\beta|;0_\bullet),\sum j_\beta=q} \left( \prod\limits_{\beta=-1}^t N^{\lambda_{\beta}}_{\rho^{(\beta)}_{\beta}\rho^{(\beta)}_{\beta+1}...\rho^{(\beta)}_t} \right) {\bf N}^{\kappa_\bullet}_{\rho^{(-1)}_\bullet...\rho^{(t)}_\bullet} \;.
$$
\end{lemma}

\begin{proof}
The lemma is a reformulation of \cite[Proposition~4.30]{Chirvasitu-Penkov-RC} in our notation. The proof is done in steps, first observing that for every $\beta\in\{-1,...,t\}$
\begin{gather*}
\ul{\rm soc}^{q+1}J_{\lambda_\beta;\emptyset_\bullet} \cong \bigoplus\limits_{\rho_\bullet\in\Lambda^{t+2}:(|\rho_\bullet|;0_\bullet)\in{\bf P}^q(|\lambda_\beta|;0_\bullet)} N^{\lambda_{\beta}}_{\rho_\beta...\rho_t} \cdot L_{\rho_\bullet;\emptyset_\bullet} \;,\\
\end{gather*}
and then using the decomposition
$$
\ul{\rm soc}^{q+1}J_{\lambda_\bullet;\emptyset_\bullet} \cong \bigoplus\limits_{j_{-1}+...j_t=q} \bigotimes\limits_{\beta=-1}^t \ul{\rm soc}^{j_\beta+1}J_{\lambda_\beta;\emptyset_\bullet} \;.
$$
Note that the sets ${\bf P}^{j_\beta}(|\lambda_\beta|;0_\bullet)$ and ${\bf P}^{q}(|\lambda_\bullet|;0_\bullet)$ are described in the proof of Proposition \ref{Prop socqJlm}.
\end{proof}

Working towards the socle filtration of $J_{\lambda_\bullet,\lambda;\mu,\mu_\bullet}$, the decomposition $J_{\lambda_\bullet,\lambda;\mu,\mu_\bullet}=J_{\lambda_\bullet,\lambda;\emptyset}\otimes J_{\emptyset;\mu,\mu_\bullet}$ leads us to consider, for $k\in\ZZ_{\geq0}$, the semisimple $\mk{gl}^M$-module
\begin{gather}\label{For module Sk}
\begin{array}{rl}
Z^{k+1}_{\lambda_\bullet,\lambda;\mu,\mu_\bullet} :=& \bigoplus\limits_{i+j=k} {\rm soc}(\ul{\rm soc}^{i+1}J_{\lambda_\bullet,\lambda;\emptyset}\otimes\ul{\rm soc}^{j+1}J_{\emptyset;\mu,\mu_\bullet}) \\
\cong & \bigoplus\limits_{(\kappa_\bullet,\kappa;\nu,\nu_\bullet)\in\bm\Lambda:(|\kappa_\bullet|,|\kappa|;|\nu|,|\nu_\bullet|)\in{\bf P}^i(|\lambda_\bullet|,|\lambda|;0,0_\bullet)\times{\bf P}^j(0_\bullet,0;|\mu|,|\mu_\bullet|),i+j=k} z^{\lambda,\lambda_\bullet}_{\kappa,\kappa_\bullet} z^{\mu,\mu_\bullet}_{\nu,\nu_\bullet} \cdot L_{\kappa_\bullet,\kappa;\nu,\nu_\bullet}
\end{array} 
\end{gather}
whose decomposition is derived from Lemma \ref{Lemma soc q Jlambda} and Proposition \ref{Prop TT tensor prod simples}.

\begin{prop}\label{Prop soc q Jlambdamu}
Let $\bm\lambda=(\lambda_\bullet,\lambda;\mu,\mu_\bullet)\in \bm\Lambda$. Then
\begin{align*}
\ul{\rm soc}^{q+1}J_{\lambda_\bullet,\lambda;\mu,\mu_\bullet} & \cong \bigoplus\limits_{j+k=q} \bigoplus\limits_{\xi,\eta\in\Lambda} {\rm Hom}(V_{\xi;\eta},\ul{\rm soc}^{j+1}(V_{*\lambda}\otimes V_{\mu}))\otimes Z^{k+1}_{\lambda_{\bullet},\xi;\eta,\mu_{\bullet}} \\
& \cong \bigoplus\limits_{\bm\kappa=(\kappa_\bullet,\kappa;\nu,\nu_\bullet)\in\bm\Lambda:|\bm\kappa|\in{\bf P}^q(|\bm\lambda|)} \left( \sum\limits_{\xi,\eta\in\Lambda} z^{\xi,\lambda_{\bullet}}_{\kappa,\kappa_\bullet} h^{\lambda;\mu}_{\xi;\eta} z^{\eta,\mu_{\bullet}}_{\nu,\nu_\bullet} \right) \cdot L_{\kappa_\bullet,\kappa;\nu,\nu_\bullet} \;.
\end{align*}
\end{prop}

\begin{proof}
From Theorem \ref{Theo Simples and Injhulls} we know that $J_{\bm\lambda}$ is indecomposable with simple socle $L_{\bm\lambda}$. By (\ref{For LJITlm is sum LJITlambdamu}), $J_{\bm\lambda}$ appears as a direct summand in the module $J_{|\bm\lambda|}$. Furthermore, by Proposition \ref{Prop socqJlm}, the layer $\ul{\rm soc}^{q+1}J_{|\bm\lambda|}$ is a direct sum of modules of the form $L_{\bm k}$ with $\bm k\in {\bf P}^q(|\bm\lambda|)$. In turn, each $L_{\bm k}$ decomposes as a direct sum of modules of the form $L_{\bm\kappa}$ with $\bm\kappa\in\bm\Lambda$, $|\bm\kappa|=\bm k$. Hence, $\ul{\rm soc}^{q+1}J_{\bm\lambda}$ consists exactly of the simple subquotients of $J_{\bm\lambda}$ of isomorphic to some $L_{\bm\kappa}$ with $|\bm\kappa|\in{\bf P}^q(|\bm\lambda|)$. It remains to compute the multiplicity with which $L_{\bm\kappa}$ occurs as a subquotient of $J_{\bm\lambda}$. To this end, we start with the decomposition $J_{\lambda_\bullet,\lambda;\mu,\mu_\bullet}=J_{\lambda_\bullet,\lambda;\emptyset_\bullet}\otimes J_{\emptyset_\bullet;\mu,\mu_\bullet}$. The socle filtrations of the two tensorands are obtained from Lemma \ref{Lemma soc q Jlambda}. The tensor products of simple modules are described in Proposition \ref{Prop TT tensor prod simples}, and the formula for the multiplicities follows immediately.
\end{proof}

\subsubsection{The socle filtrations of the modules $I_{\lambda_\bullet,\lambda;\mu,\mu_\bullet}$ and $I_{l_\bullet,l;m,m_\bullet}$}\label{Sec Ilm Ilambdamu}

\begin{prop}\label{Prop soc q I}
The socle filtration of the module $I=\lim\limits_\to S^kQ$ defined in (\ref{For I}) is infinite and exhaustive: for $q\in\NN$ the $(q+1)$-st layer is given by
$$
\ul{\rm soc}^{q+1}I \cong \bigoplus\limits_{j+|\zeta|=q} \ul{\rm soc}^{j+1}J_{\zeta,\emptyset;\emptyset,\zeta} \cong \bigoplus\limits_{j+|\zeta|=j} Z^{j+1}_{\zeta,\emptyset;\emptyset,\zeta} \;.
$$
\end{prop}

\begin{proof}
The socle filtration of a direct limit of modules of finite length is always exhaustive. The layers of the defining filtration of $I$ are (see formula (\ref{For SkCQF}))
$$
S^kQ/S^{k-1}Q\cong S^k(V^*/V_*\otimes \bar V/V)=S^k J_{1,0;0,1} \cong \bigoplus\limits_{|\zeta|=k} J_{\zeta,\emptyset;\emptyset,\zeta} \;.
$$
By Proposition \ref{Prop soc q Jlambdamu} we have
$$
\ul{\rm soc}^{j+1}(S^kQ/S^{k-1}Q)\cong\bigoplus\limits_{|\zeta|=k} Z^{j+1}_{\zeta,\emptyset;\emptyset,\zeta} \;.
$$
This yields a filtration of $I$ with layers as indicated in our statement. To show that this is the socle filtration, it remains to show that no simple constituent appears in a lower socle than expected. It suffices to prove the statement for the submodule $S^kQ\subset I$, and we will do this by induction on $k$. The case $k=0$ is trivial. By Theorem \ref{Theo Simples and Injhulls} we have
\begin{gather}\label{For socSkJ1001}
{\rm soc}(S^kQ/S^{k-1}Q) \cong {\rm soc}(S^kJ_{1,0;0,1})\cong \bigoplus\limits_{|\zeta|=k} L_{\zeta,\emptyset;\emptyset,\zeta} \;.
\end{gather}
On the other hand, we have the finite filtration $\CC=I^{0,0}\subset I^{1,1}\subset... \subset I^{t+1,t+1}=I$ following from the definition of $I^{r,s}$ in (\ref{For Irs}). The submodule $I^{1,1}\subset I$ has the module (\ref{For socSkJ1001}) as the $k+1$-st layer of its defining filtration $I^{1,1}=\lim\limits_{k\to\infty} S^{k}Q^{1,1}$. Note that for $I^{1,1}$ the defining filtration coincides with its socle filtration, i.e.,
$$
\ul{\rm soc}^{k+1}(I^{1,1})\cong S^k(V_{1}^*/V_*\otimes \bar V_{1}/V) \cong \bigoplus\limits_{|\zeta|=k} L_{\zeta,\emptyset;\emptyset,\zeta} \;.
$$
It follows that ${\rm soc}(S^kQ/S^{k-1}Q)\subset \ul{\rm soc}^{k+1}I$ and, by induction on $j$, $\ul{\rm soc}^{j+1}(S^kQ/S^{k-1}Q)\subset \ul{\rm soc}^{j+k+1}I$. This completes the proof.
\end{proof}

\begin{prop}\label{Prop soc q Ilambdamu}
For $(\lambda_\bullet,\lambda;\mu,\mu_\bullet)\in \bm\Lambda$ the layers of the socle filtration of $I_{\lambda_\bullet,\lambda;\mu,\mu_\bullet}$ are
\begin{align*}
\ul{\rm soc}^{q+1}(J_{\lambda_\bullet,\lambda;\mu,\mu_\bullet}\otimes I) & \cong \bigoplus\limits_{j+k=q} \ul{\rm soc}^{j+1}J_{\lambda_\bullet,\lambda;\mu,\mu_\bullet} \otimes \ul{\rm soc}^{k+1}I \\
& \cong \bigoplus\limits_{i+j+k+|\zeta|=q} {\rm Hom}(V_{\xi;\eta},\ul{\rm soc}^{i+1}(V_{*\lambda}\otimes V_{\mu}))\otimes Z^{j+1}_{\lambda_{\bullet},\xi;\eta,\mu_{\bullet}} \otimes Z^{k+|\zeta|+1}_{\zeta,\emptyset;\emptyset,\zeta} \\
&\qquad{^{\xi,\eta,\zeta\in\Lambda}} \qquad\qquad\qquad\qquad\qquad \\
& \cong \bigoplus\limits_{|\lambda|-|\xi|+j+k=q} h^{\lambda;\mu}_{\xi;\eta} \cdot Z^{j+1}_{\lambda_{\bullet},\xi;\eta,\mu_{\bullet}} \otimes Z^{k+|\zeta|+1}_{\zeta,\emptyset;\emptyset,\zeta} \;,\\
& \quad {^{\xi,\eta,\zeta\in\Lambda}}
\end{align*}
where the numbers $h^{\lambda;\mu}_{\xi;\eta}$ are defined in (\ref{For h lambdamuxieta}).
\end{prop}

\begin{proof}
The socle filtrations of $J_{\lambda_\bullet,\lambda;\mu,\mu_\bullet}$ and $I$ are described in Propositions \ref{Prop soc q Jlambdamu} and \ref{Prop soc q I}, respectively. All simple subquotients of $I$ have the property that their tensor products with semisimple tensor modules are semisimple and their tensor products with essential extensions yield essential extensions. The implies the first line in the above formula. The rest follows from the expressions for $\ul{\rm soc}^{j+1}J_{\lambda_\bullet,\lambda;\mu,\mu_\bullet}$ and $\ul{\rm soc}^{k+1}I$ given in Propositions \ref{Prop soc q Jlambdamu} and \ref{Prop soc q I}.
\end{proof}

\begin{coro}\label{Coro soc filt Ilm}
For $\bm l=(l_\bullet,l;m,m_\bullet)\in \mc P$ the socle filtration of $I_{\bm l}$, and its layers, are
\begin{gather*}
{\rm soc}^{q+1}I_{\bm l} =\bigcap\limits_{f\in\Xi^{q+1}(I_{\bm l})} {\rm ker}f \;,\quad \ul{\rm soc}^{q+1}I_{\bm l} \cong \bigoplus\limits_{j+k=q} \ul{\rm soc}^{j+1}J_{\bm l} \otimes \ul{\rm soc}^{k+1}I \;.
\end{gather*}
Consequently, ${\rm Hom}(L_{\bm k},\ul{\rm soc}^{q+1}I_{\bm l})$ implies $\bm k\in\mc P^{q}(\bm l)$.
\end{coro}

\begin{proof}
The statement on the layers follows from \ref{Prop soc q Ilambdamu} and the decompositions of $I_{\bm l}$ and $J_{\bm l}$ given in (\ref{For LJITlm is sum LJITlambdamu}). The last statement of the corollary follows immediately. The expression for ${\rm soc}^{q+1}I_{\bm l} =\bigcap\limits_{f\in\Xi^{q+1}(I_{\bm l})} {\rm ker}f$ is then deduced by induction using (\ref{For Pql compute}) (applied for the poset $\mc P(\bm l)$), by an argument similar to the one applied for ${\rm soc}^{q+1}J_{\bm l}$ in Proposition \ref{Prop socqJlm}.
\end{proof}

\subsection{Order on the category $\TT_t$}\label{Sect Groth order}

\begin{theorem}\label{Theo TT Grothe Order}
The category $\TT_t$ is an ordered Grothendieck category with order-defining objects
\begin{gather*}
I_{\bm l} = I \otimes J_{\bm l} \;, \quad \bm l\in\mc P ,
\end{gather*}
parametrized by the poset $\mc P$ of Definition \ref{Def Poset P}, see (\ref{For LJITlm}). The socles of the order-defining objects are given by
\begin{gather*}
{\rm soc}I_{\bm l} = {\rm soc}J_{\bm l} = L_{\bm l} = \bigoplus\limits_{\bm\lambda\in\mc S_{\bm l}} \CC^{\bm\lambda}\otimes L_{\bm\lambda} \;,
\end{gather*}
where $\mc S_{\bm l}:=\{\bm\lambda\in\bm\Lambda:|\bm\lambda|=\bm l\}$.
\end{theorem}

\begin{proof}
We need to check the axioms (a)-(f) of Definition \ref{Def OrdGrothCat}. Let $\bm l=(l_\bullet,l;m,m_\bullet)\in\mc P$. The socle filtration of $I_{\bm l}$ is determined in Corollary \ref{Coro soc filt Ilm}. In particular, we obtain claimed expressions for ${\rm soc}I_{\bm l}$ (see also (\ref{For LJITlm is sum LJITlambdamu})). Therefore, axioms (a) and (e) are satisfied. Axiom (b) holds by the definition of $\TT$. Axiom (c) holds with the above set $\mc S_{\bm l}$, in view of Theorem \ref{Theo EndoSimpl}. Axiom (d) holds because of Corollary \ref{Coro soc filt Ilm}. The family of morphisms required in axiom (f), for $\bm k\prec\bm l$, consists of $f:I_{\bm l}\to I_{\bm k}$ such that $f\in\Xi^{q}(I_{\bm l})$, where $q$ is the unique integer such that $\bm k\in\mc P^q(\bm l)$.
\end{proof}

\begin{coro}\label{Coro TT simples and inj hulls}
The map $\bm\lambda\mapsto L_{\bm\lambda}$ induces a bijection of $\bm\Lambda$ with the set of isomorphism classes of simple objects in the category $\TT$. Furthermore, $I_{\bm\lambda}$ is an injective hull of $L_{\bm\lambda}$, and the modules $I_{\bm\lambda}$, $\bm\lambda\in\bm\Lambda$, exhaust (up to isomorphism) the indecomposable injectives of $\TT$.
\end{coro}

\begin{proof}
The statement follows immediately from Theorem \ref{Theo TT Grothe Order} and Proposition \ref{Prop InjHullInGrothCat}.
\end{proof}

Our next goal is to determine injective resolutions of the simple objects $L_{\bm\lambda}=L_{\lambda_\bullet,\lambda;\mu,\mu_\bullet}$ in the category $\TT$. As an intermediate step we shall solve the same problem for $L_{\lambda_\bullet,\lambda;\emptyset}$ in the category $\TT(V^*)$, or, analogously, $L_{\emptyset_\bullet;\mu,\mu_\bullet}$ in $\TT(\bar V)$. The general solution will be constructed in \S\ref{Sec InjResTT} with these ingredients, along with the resolutions of $L_{\emptyset_\bullet,\lambda;\mu,\emptyset_\bullet}=V_{\lambda,\mu}$ given in Theorem \ref{Theo smallTT Inj n Ext to Hom}.

\subsubsection{An involution on $\bm\Lambda$}\label{Sec InvolutLambdaBeth}

We introduce here some symmetries of the set $\bm\Lambda$ that will be useful in the descriptions of injective resolutions of simple objects. For any sequence $\lambda_\bullet$ of Young diagrams, we denote by $\lambda_\bullet^\perp$ the sequence whose terms are the conjugates $\lambda_\alpha^\perp$ of the terms of $\lambda_\bullet$. We denote by $\lambda_\bullet^{e\perp}$ and $\lambda_\bullet^{o\perp}$ the sequences, where only the diagrams with even, respectively odd, index $\alpha$ are replaced by their conjugates, while the odd, respectively even, terms remain unchanged. For an element $\bm\lambda=(\lambda_\bullet,\lambda;\mu,\mu_\bullet)\in\bm\Lambda$, we denote $\bm\lambda^{e\perp o}=(\lambda_\bullet^{e\perp},\lambda;\mu^\perp,\mu_\bullet^{o\perp})$. Clearly this defines an involution on $\bm\Lambda$.

\subsection{The category $\TT(V^*)$}

Recall that $\TT(V^*)$ is the smallest full tensor Grothendieck subcategory of $\TT_t$ containing $V^*$ and closed under taking subquotients. In this section we use the notation $l_\bullet=(l_{-1},...,l_{t})\in\NN^{t+2}$, as in Remark \ref{Rem index lm}.

\begin{definition}\label{Def Poset n}
We consider the poset $\mc P_{\rm left}= \NN^{t+2}$ with the following partial order:
$$
l_\bullet\preceq l'_\bullet \quad \tst \quad \begin{array}{|l} |l_\bullet|=|l'_\bullet| \\ \sum_{\alpha\geq\beta} l_\alpha \geq\sum_{\alpha\geq\beta} l'_\alpha \quad \forall \beta \end{array}\;.
$$
\end{definition}

\begin{rem}
The underlying set of $\mc P_{\rm left}$ is included as the set of elements $(l_\bullet;0_\bullet)$ with $l_\bullet\in\NN^{t+2}$ in the underlying set $\NN^{2(t+2)}$ of both posets $\mc P$ and ${\bf P}$. The partial order on $\mc P_{\rm left}$ coincides with the restrictions of both these partial orders. Analogously, we define a poset $\mc P_{\rm right}$ (isomorphic to $\mc P_{\rm left}$) as the set of elements of $\mc P$ of the form $(0_\bullet;m_\bullet)$ with $m_\bullet\in\NN^{t+2}$, with the restricted order from $\mc P$ or ${\bf P}$.
\end{rem}

\begin{theorem} (\cite[\S~4.2]{Chirvasitu-Penkov-RC})

The category $\TT(V^*)$ is an ordered Grothendieck category with order-defining objects
$$
J_{l_\bullet;\emptyset} = \bigotimes\limits_{\alpha=-1}^t (V^*/V^*_\alpha)^{\otimes l_\alpha}\;,
$$
parametrized by the poset $\mc P_{\rm left}$. Moreover,
$$
{\rm soc}J_{l_\bullet;\emptyset} = L_{l_\bullet;\emptyset} \cong \bigoplus\limits_{\lambda_\bullet\in\bm\Lambda_{\rm left}:|\lambda_\bullet|=l_\bullet} \CC^{\lambda_\bullet}\otimes L_{\lambda_\bullet;\emptyset} \;.
$$
The simple objects and the indecomposable injectives of $\TT(V^*)$ are, up to isomorphism, the modules $L_{\lambda_\bullet;\emptyset}$ and $J_{\lambda_\bullet;\emptyset}$ with $\lambda_\bullet\in\bm\Lambda_{\rm left}$, respectively.
\end{theorem}

\begin{rem}
The following properties hold in the category $\TT(V^*)$:
\begin{enumerate}
\item Any tensor product of semisimple modules is semisimple.
\item Any tensor product of injective modules is injective.
\item The pure simple modules are, up to isomorphism, exactly the modules of the form $(V^*_{\alpha+1}/V^*_{\alpha})_\lambda$ with $\lambda\in\Lambda$ and $\alpha\in\{-1,0,...,t\}$.
\end{enumerate}
\end{rem}

\subsubsection{Injective resolutions of simple objects in $\TT(V^*)$}

\begin{prop}\label{Prop InjResLlambda}
For any Young diagram $\lambda$ and $\alpha\in\{-1,0,...,t\}$, there is an injective resolution in $\TT(V^*)$ of the simple $\mk{gl}^M$-module $(V^*_{\alpha+1}/V^*_{\alpha})_{\lambda}$ of length $0$ if $\alpha=t$, and length $|\lambda|$ if $\alpha<t$. In the latter case, this resolution is
$$
0\to (V^*_{\alpha+1}/V^*_{\alpha})_{\lambda} \to \mc I^0((V^*_{\alpha+1}/V^*_{\alpha})_{\lambda}) \to \mc I^1((V^*_{\alpha+1}/V^*_{\alpha})_{\lambda}) \to ... \to \mc I^{|\lambda|}((V^*_{\alpha+1}/V^*_{\alpha})_{\lambda}) \to 0 \;,
$$
with
\begin{align*}
& \mc I^0((V^*_{\alpha+1}/V^*_{\alpha})_{\lambda}) = (V^*/V^*_{\alpha})_{\lambda} \\
& \mc I^1((V^*_{\alpha+1}/V^*_{\alpha})_{\lambda}) = (V^*/V^*_{\alpha+1})\otimes (\bigoplus\limits_{\sigma\in\lambda:N^{\lambda}_{\sigma(1)}=1}(V^*/V^*_{\alpha})_{\sigma}) \\
& \mc I^j((V^*_{\alpha+1}/V^*_{\alpha})_{\lambda}) = \bigoplus\limits_{\sigma,\tau\in\Lambda:|\tau|=j}N_{\sigma\tau^\perp}^{\lambda} (V^*/V^*_{\alpha+1})_{\tau}\otimes (V^*/V^*_{\alpha})_{\sigma} \\
& \mc I^{|\lambda|}((V^*_{\alpha+1}/V^*_{\alpha})_{\lambda}) = (V^*/V^*_{\alpha+1})_{\lambda^\perp} \;.
\end{align*}
\end{prop}

\begin{proof}
The result is proven in \cite{Chirvasitu-Penkov-OTC} under the assumption that $t=0$, but this assumption is inessential.
\end{proof}

\begin{theorem}\label{Theo InjResLlambdabullet}
Let $\lambda_\bullet\in\bm\Lambda_{\rm left}$. There exists an injective resolution of the simple module $L_{\lambda_\bullet;\emptyset}$ in $\TT(V^*)$ of length equal to $||\lambda_{\bullet_{<t}}||=\sum\limits_{-1\leq\alpha<t} |\lambda_\alpha|$. The decomposition of the $k$-th term of this resolution into indecomposable injective direct summands is
$$
\mc I^k(L_{\lambda_\bullet;\emptyset}) \cong \bigoplus\limits_{\kappa_\bullet\in\bm\Lambda_{\rm left}: k^{\lambda_\bullet}_{\kappa_\bullet}=k} 
p^{\lambda_\bullet}_{\kappa_\bullet} \cdot J_{\kappa_\bullet;\emptyset} \;,
$$
where
\begin{gather}\label{For klambdakappa plambdakappa}
k^{\lambda_\bullet}_{\kappa_\bullet}:=\sum\limits_{0\leq\alpha\leq t} (\alpha+1)(|\kappa_\alpha|-|\lambda_\alpha|)\;,\quad p^{\lambda_\bullet}_{\kappa_\bullet} := \sum\limits_{\sigma_\bullet,\tau_\bullet\in\bm\Lambda_{\rm left}:\sigma_t=\tau_{-1}=\emptyset} N^{\kappa_t}_{\lambda_t \tau_t} \prod\limits_{-1\leq\alpha<t} N^{\lambda_\alpha}_{\sigma_\alpha \tau_{\alpha+1}^\perp} N^{\kappa_\alpha}_{\sigma_\alpha \tau_\alpha}\;.
\end{gather}
The last nonzero term of the resolution is
$$
\mc I^{||\lambda_{\bullet_{<t}}||}(L_{\lambda_\bullet;\emptyset}) \cong (V^*/V^*_{t})_{\lambda_t} \otimes \left( \bigotimes\limits_{-1\leq\alpha<t} (V^*/V^*_{\alpha+1})_{\lambda_{\alpha}^{\perp}} \right) \;,
$$
and this term is an indecomposable $\mk{gl}^M$-module if and only if $\lambda_t=\emptyset$ or $\lambda_{t-1}=\emptyset$.
\end{theorem}

\begin{proof}
The category $\TT(V^*)$ has the property that the tensor product of injective modules is injective and the tensor product of semisimple modules is semisimple. Thus we can apply Lemma \ref{Lemma TensInjRes}, which yields the first line in the formula below; the second line follows from Proposition \ref{Prop InjResLlambda}, and the rest follows by standard rules for tensor products:
\begin{align*}
\mc I^{k}(L_{\lambda_\bullet;\emptyset}) &\cong \bigoplus\limits_{j_{-1}+j_0+...+j_{t}=k}\bigotimes\limits_{-1\leq\alpha\leq t} \mc I^{j_\alpha}(L_{\lambda_\alpha;\emptyset}) \\
& \cong (V^*/V^*_{t})_{\lambda_t}\otimes ( \bigoplus\limits_{j_{-1}+...+j_{t-1}=k}\bigotimes\limits_{-1\leq\alpha\leq t-1} \bigoplus\limits_{\sigma,\tau\in\Lambda:|\tau|=j_\alpha}N_{\sigma\tau^\perp}^{\lambda} (V^*/V^*_{\alpha+1})_{\tau}\otimes (V^*/V^*_{\alpha})_{\sigma} ) \\
& \cong (V^*/V^*_{t})_{\lambda_t}\otimes ( \bigoplus\limits_{\sigma_\bullet,\tau_\bullet\in\bm\Lambda_{\rm left}:\sigma_t=\tau_{-1}=\emptyset,||\tau_\bullet||=k} \bigotimes\limits_{-1\leq\alpha\leq t} N_{\sigma_\alpha\tau_{\alpha+1}^\perp}^{\lambda}  (V^*/V^*_{\alpha})_{\tau_{\alpha}}\otimes (V^*/V^*_{\alpha})_{\sigma_\alpha} ) \\
& \cong \bigoplus\limits_{\kappa_\bullet\in\bm\Lambda_{\rm left}} \left(\sum\limits_{\sigma_\bullet,\tau_\bullet\in\bm\Lambda_{\rm left}:\sigma_t=\tau_{-1}=\emptyset,||\tau_\bullet||=k} N^{\kappa_t}_{\lambda_t \tau_t} \prod\limits_{\alpha=-1}^{t-1} N^{\lambda_\alpha}_{\sigma_\alpha \tau_{\alpha+1}^\perp} N^{\kappa_\alpha}_{\sigma_\alpha \tau_\alpha} \right)\cdot J_{\kappa_\bullet;\emptyset} \;.
\end{align*}
To obtain the explicit form of the coefficients $p^{\lambda_\bullet}_{\kappa_\bullet}$ stated in (\ref{For klambdakappa plambdakappa}), it remains to show that the condition $||\tau_\bullet||=k$ appearing above can be substituted by the condition $k=k^{\lambda_\bullet}_{\kappa_\bullet}$. We claim that $p^{\lambda_\bullet}_{\kappa_\bullet}\ne0$ implies the following:
\begin{enumerate}
\item $|\lambda_\bullet|\succeq|\kappa_\bullet|$ in the poset $\mc P_{\rm left}$;
\item ${\rm supp}(\kappa_\bullet)\subset {\rm supp}(\lambda_\bullet)\cup (1+{\rm supp}(\lambda_\bullet))$;
\item every nonvanishing summand in the defining formula for $p^{\lambda_\bullet}_{\kappa_\bullet}$ arises for $\sigma_\bullet,\tau_\bullet$ satisfying $||\tau_\bullet||=k^{\lambda_\bullet}_{\kappa_\bullet}$.
\end{enumerate}
Indeed, the nonvanishing of a summand of $p^{\lambda_\bullet}_{\kappa_\bullet}$ implies $|\tau_t|=|\kappa_t|-|\lambda_t|$, since $N^{\kappa_t}_{\lambda_t \tau_t}\ne0$, and $|\tau_\alpha|=|\kappa_\alpha|-|\lambda_\alpha|+|\tau_{\alpha+1}|$ for $-1\leq\alpha<t$ since $N^{\lambda_\alpha}_{\sigma_\alpha \tau_{\alpha+1}^\perp} N^{\kappa_\alpha}_{\sigma_\alpha \tau_\alpha}\ne 0$. Now part 3 of the claim follows by induction on $t$. Parts 1 and 2 are trivial to verify. The statement on the injective dimension and the last nonzero term of the injective resolution follows immediately. This completes the proof.
\end{proof}

The above theorem allows us to compute the dimensions of the Ext-groups of pairs of simple objects in $\TT(V^*)$.

\begin{coro}\label{Coro dimExtsimples in Tleft}
Let $\kappa_\bullet,\lambda_\bullet\in\bm\Lambda_{\rm left}$. Then
$$
\dim {\rm Ext}_{\TT(V^*)}^k (L_{\kappa_\bullet;\emptyset},L_{\lambda_\bullet;\emptyset}) =\begin{cases} p^{\lambda_\bullet}_{\kappa_\bullet} \quad if \quad k=k^{\lambda_\bullet}_{\kappa_\bullet} \;, \\ 0 \qquad otherwise . \end{cases}
$$
\end{coro}

In the next corollary we encounter a new family of modules, whose socle filtrations relate to the injective resolutions of simple modules given in Theorem \ref{Theo InjResLlambdabullet}. For $(l_\bullet;m_\bullet)\in\mc P$ and $(\lambda_\bullet;\mu_\bullet)\in\bm\Lambda$, we denote
\begin{gather}\label{For Modules M}
\begin{array}{l}
M_{l_\bullet;m_\bullet}:= (V^*/V^*_{t})^{\otimes l_t} \otimes \left(\bigotimes\limits_{\alpha=-1}^{t-1} (V^*_{\alpha+2}/V^*_{\alpha})^{\otimes l_\alpha}\right) \otimes \left(\bigotimes\limits_{\alpha=-1}^{t-1} (\bar V_{\alpha+2}/\bar V_{\alpha})^{\otimes m_\alpha}\right) \otimes (\bar V/\bar V_t)^{\otimes m_t}\;,\\
M_{\lambda_\bullet;\mu_\bullet}:= (V^*/V^*_{t})_{\lambda_t} \otimes \left(\bigotimes\limits_{\alpha=-1}^{t-1} (V^*_{\alpha+2}/V^*_{\alpha})_{\lambda_\alpha}\right) \otimes \left(\bigotimes\limits_{\alpha=-1}^{t-1} (\bar V_{\alpha+2}/\bar V_{\alpha})_{\mu_\alpha}\right) \otimes (\bar V/\bar V_t)_{\mu_t}\;.
\end{array}
\end{gather}

\begin{coro}\label{Coro Ext to Hom in TVstar}
For $\kappa_\bullet,\lambda_\bullet\in\bm\Lambda_{\rm left}$ and $k\geq0$,
$$
\dim {\rm Ext}_{\TT(V^*)}^k (L_{\kappa_\bullet^{o\perp};\emptyset},L_{\lambda_\bullet^{o\perp};\emptyset}) = \dim {\rm Ext}_{\TT(V^*)}^k (L_{\kappa_\bullet^{e\perp};\emptyset},L_{\lambda_\bullet^{e\perp};\emptyset}) = \dim {\rm Hom}(L_{\kappa_\bullet},\ul{\rm soc}^{k+1}(M_{\lambda_\bullet;\emptyset}))\;,
$$
where $M_{\lambda_\bullet;\emptyset}$ is defined in (\ref{For Modules M}).
\end{coro}

\begin{proof}
First we compute the socle filtration of $M_{\lambda_\bullet;\emptyset}$. Since $M_{\lambda_\bullet;\emptyset}=\otimes_\alpha M_{\lambda_\alpha;\emptyset}$, we have
\begin{align*}
\ul{\rm soc}^{k+1}(M_{\lambda_\bullet;\emptyset})& \cong (V^*/V^*_{t})_{\lambda_t}\otimes \left( \bigoplus\limits_{j_{-1}+...+j_{t-1}=k} \bigotimes\limits_{-1\leq\alpha<t} \ul{\rm soc}^{j_\alpha+1}((V^*_{\alpha+2}/V^*_{\alpha})_{\lambda_\alpha})\right)\\
& \cong L_{\lambda_t;\emptyset}\otimes \left( \bigoplus\limits_{j_{-1}+...+j_{t-1}=k} \bigotimes\limits_{-1\leq\alpha<t} \left( \bigoplus\limits_{\begin{array}{c}\sigma,\tau\in\Lambda\\ |\tau|=j_\alpha\end{array}}  N^{\lambda_\alpha}_{\sigma \tau} \cdot (V^*_{\alpha+2}/V^*_{\alpha+1})_{\tau} \otimes (V^*_{\alpha+1}/V^*_{\alpha})_{\sigma} \right)\right)\\
& \cong \bigoplus\limits_{\kappa_\bullet\in\bm\Lambda_{\rm left}} \left(\sum\limits_{\sigma_\bullet,\tau_\bullet\in\bm\Lambda_{\rm left},\sigma_t=\tau_{-1}=\emptyset,||\tau_\bullet||=k} N^{\kappa_t}_{\lambda_t \tau_t} \prod\limits_{-1\leq\alpha<t} N^{\lambda_\alpha}_{\sigma_\alpha \tau_{\alpha+1}} N^{\kappa_\alpha}_{\sigma_\alpha \tau_\alpha} \right)\cdot L_{\kappa_\bullet;\emptyset}.
\end{align*}
On the other hand, we apply Corollary \ref{Coro dimExtsimples in Tleft} to compute $p^{\lambda_\bullet^{o\perp}}_{\kappa_\bullet^{o\perp}}=\dim{\rm Ext}_{\TT(V^*)}^k (L_{\kappa_\bullet^{o\perp};\emptyset},L_{\lambda_\bullet^{o\perp};\emptyset})$, the case of ${\rm Ext}_{\TT(V^*)}^k (L_{\kappa_\bullet^{e\perp};\emptyset},L_{\lambda_\bullet^{e\perp};\emptyset})$ being analogous. We also assume $t$ to be even, the odd case being similar. In the calculation below, we begin by replacing the product over $\{-1,0,...,t\}$ in the formula for $p^{\lambda_\bullet^{o\perp}}_{\kappa_\bullet^{o\perp}}$ by a product over the even indices of two-fold products of the respective consecutive terms. The subsequent manipulations follow by standard properties of the Littlewood-Richardson numbers. We obtain
\begin{align*}
\dim {\rm Ext}_{\TT(V^*)}^k (L_{\kappa_\bullet^{o\perp};\emptyset},L_{\lambda_\bullet^{o\perp};\emptyset}) & = \sum\limits_{\sigma_\bullet,\tau_\bullet\in\bm\Lambda_{\rm left},\sigma_t=\tau_{-1}=\emptyset,||\tau_\bullet||=k} \prod\limits_{\begin{matrix}-1\leq\alpha<t\\ \alpha\;{\rm odd}\end{matrix}} N^{\lambda_\alpha^\perp}_{\sigma_\alpha \tau_{\alpha+1}^\perp} N^{\kappa_\alpha^\perp}_{\sigma_\alpha \tau_\alpha} N^{\lambda_{\alpha+1}}_{\sigma_{\alpha+1} \tau_{\alpha+2}^\perp} N^{\kappa_{\alpha+1}}_{\sigma_{\alpha+1} \tau_{\alpha+1}} \\
& = \sum\limits_{\sigma_\bullet,\tau_\bullet\in\bm\Lambda_{\rm left},\sigma_t=\tau_{-1}=\emptyset,||\tau_\bullet||=k} \prod\limits_{\begin{matrix}-1\leq\alpha<t\\ \alpha\;{\rm odd}\end{matrix}} N^{\lambda_\alpha}_{\sigma_\alpha^\perp \tau_{\alpha+1}} N^{\kappa_\alpha}_{\sigma_\alpha^\perp \tau_\alpha^\perp} N^{\lambda_{\alpha+1}}_{\sigma_{\alpha+1} \tau_{\alpha+2}^\perp} N^{\kappa_{\alpha+1}}_{\sigma_{\alpha+1} \tau_{\alpha+1}} \\
& = \sum\limits_{\sigma_\bullet,\tau_\bullet\in\bm\Lambda_{\rm left},\sigma_t=\tau_{-1}=\emptyset,||\tau_\bullet||=k} \prod\limits_{\begin{matrix}-1\leq\alpha<t\\ \alpha\;{\rm odd}\end{matrix}} N^{\lambda_\alpha}_{\sigma_\alpha \tau_{\alpha+1}} N^{\kappa_\alpha}_{\sigma_\alpha \tau_\alpha} N^{\lambda_{\alpha+1}}_{\sigma_{\alpha+1} \tau_{\alpha+2}} N^{\kappa_{\alpha+1}}_{\sigma_{\alpha+1} \tau_{\alpha+1}} \\
& = \dim {\rm Hom}(L_{\kappa_\bullet},\ul{\rm soc}^{k+1}(M_{\lambda_\bullet;\emptyset})) \;.
\end{align*}
\end{proof}

\subsection{Injective resolutions of simple objects in $\TT_t$}\label{Sec InjResTT}

\begin{prop}\label{Prop InjRes C in T}
An injective resolution of the trivial module $\CC=L_{\emptyset;\emptyset}$ in the category $\TT$ is given by
$$
0\to \CC \to I \stackrel{\psi_0}{\to} I\otimes F \stackrel{\psi_1}{\to} I\otimes \Lambda^2F \stackrel{\psi_2}{\to} ... \stackrel{\psi_{j-1}}{\to} I\otimes \Lambda^jF \stackrel{\psi_j}{\to} ...,
$$
where $F:=V^*/V_*\otimes \bar V/V=J_{1,0;0,1}$, $\psi_0=\psi$, and the $j$-th map is defined as the direct limit $\psi_j:=\lim\limits_{k\to\infty} \psi_j^k$ of the morphisms
\begin{gather*}
\psi_j^k: S^kQ \otimes \Lambda^jF \stackrel{\Delta^k_1\otimes{\rm id}}{\lw} S^{k-1}Q\otimes Q\otimes \Lambda^jF \stackrel{{\rm id}\otimes\pi \otimes{\rm id}}{\lw} S^{k-1}Q \otimes F \otimes \Lambda^jF \stackrel{{\rm id}\otimes{\rm multiply}_{\Lambda^\bullet F}}{\lw} S^{k-1}Q \otimes \Lambda^{j+1}F \;.
\end{gather*}
The $j$-th term $\mc I^{j}_\TT(\CC):=I\otimes \Lambda^jF$ of this resolution decomposes into a direct sum of indecomposable injectives as
$$
\mc I^j_{\TT}(\CC) \cong \bigoplus\limits_{\zeta\in\Lambda:|\zeta|=j} I_{\zeta,\emptyset;\emptyset,\zeta^\perp}
$$
\end{prop}

\begin{proof}
The proposition is proven in \cite[\S~3.5]{Chirvasitu-Penkov-UTC} for $t=0$, but this assumption is not necessary.
\end{proof}

\begin{coro}
For $\bm\lambda\in\bm\Lambda$ and $j\in\NN$ we have
$$
{\rm Ext}^j_{\TT}(L_{\bm\lambda},\CC) = \begin{cases} \CC &\; if\;\bm\lambda=(\zeta,\emptyset;\emptyset,\zeta^\perp)\; with\; |\zeta|=j,\\
0 &\; otherwise. \end{cases}
$$
\end{coro}

\begin{theorem}\label{Theo InjResLlambdamu TT}
Let $\bm\lambda=(\lambda_\bullet,\lambda;\mu,\mu_\bullet)\in \bm\Lambda$. There is an injective resolution of the simple module $L_{\bm\lambda}$ in $\TT$, with $k$-th term
\begin{align*}
\mc I_{\TT}^{k}(L_{\lambda_\bullet,\lambda;\mu,\mu_\bullet}) & := \bigoplus\limits_{\begin{array}{c}i+j+l+m=k\\ \xi,\eta\in\Lambda\end{array}} \mc I_{\TT}^i(\CC)\otimes {\rm Ext}_{\TT(V_*,V)}^j(V_{\xi,\eta},V_{\lambda;\mu})\otimes \mc I_{\TT(V^*)}^{l}(L_{\lambda_{\bullet},\xi;\emptyset})\otimes \mc I_{\TT(\bar V)}^{m}(L_{\emptyset;\eta,\mu_\bullet}) \\
&\cong \bigoplus\limits_{(\kappa_\bullet,\kappa;\nu,\nu_\bullet)\in\bm\Lambda:k^{\lambda_\bullet,\lambda;\mu,\mu_\bullet}_{\kappa_\bullet,\kappa;\nu,\nu_\bullet}=k} \left(\sum\limits_{\xi,\eta,\zeta,\rho,\theta\in\Lambda} p^{\xi,\lambda_{\bullet}}_{\kappa,\rho,\kappa_{\bullet_{>0}}} N^{\kappa_0}_{\rho\zeta} m^{\lambda;\mu}_{\xi;\eta} N^{\nu_0}_{\theta\zeta^\perp} p^{\eta,\mu_\bullet}_{\nu,\theta,\nu_{\bullet_{>0}}} \right) \cdot I_{\kappa_\bullet,\kappa;\nu,\nu_\bullet} \;,
\end{align*}
where $k^{\lambda_\bullet,\lambda;\mu,\mu_\bullet}_{\kappa_\bullet,\kappa;\nu,\nu_\bullet}:=|\lambda|-|\kappa|+\sum\limits_{0\leq\alpha\leq t}(\alpha+\frac12)(|\kappa_\alpha|-|\lambda_\alpha|+|\nu_\alpha|-|\mu_\alpha|)$.
\end{theorem}

\begin{proof}
Let us first establish the relation between the two expressions for $\mc I_{\TT}^{k}(L_{\lambda_\bullet,\lambda;\mu,\mu_\bullet})$. The building blocks of the first expression are computed, respectively, in Theorem \ref{Theo InjResLlambdabullet} for the injective resolutions of the ``one-sided'' modules $L_{\lambda_\bullet,\lambda;\emptyset_\bullet}$ and $L_{\emptyset_\bullet;\mu,\mu_\bullet}$ in the respective categories $\TT(V^*)$ and $\TT(\bar V)$, Proposition \ref{Prop InjRes C in T} for the resolution of $\CC$ in $\TT$, and Theorem \ref{Theo smallTT Inj n Ext to Hom} for the resolution of $V_{\lambda;\mu}$ in $\TT(V_*,V)$. Compiling the coefficients from these building blocks we obtain
$$
\mc I_{\TT}^{k}(L_{\bm\lambda})\cong \bigoplus\limits_{(\kappa_\bullet,\kappa;\nu,\nu_\bullet)\in\bm\Lambda} \left(\sum\limits_{\stackrel{\xi,\eta,\zeta,\rho,\theta\in\Lambda}{|\lambda|-|\xi|+|\zeta|+k^{\xi,\lambda_{\bullet}}_{\kappa,\rho,\kappa_{\bullet_{>0}}}+k^{\eta,\mu_\bullet}_{\nu,\theta,\nu_{\bullet_{>0}}} = k}}  p^{\xi,\lambda_{\bullet}}_{\kappa,\rho,\kappa_{\bullet_{>0}}} N^{\kappa_0}_{\rho\zeta} m^{\lambda;\mu}_{\xi;\eta} N^{\nu_0}_{\theta\zeta^\perp} p^{\eta,\mu_\bullet}_{\nu,\theta,\nu_{\bullet_{>0}}} \right) \cdot I_{\kappa_\bullet,\kappa;\nu,\nu_\bullet} \;,
$$
and observe that the equality $|\lambda|-|\xi|+|\zeta|+k^{\xi,\lambda_{\bullet}}_{\kappa,\rho,\kappa_{\bullet_{>0}}}+k^{\eta,\mu_\bullet}_{\nu,\theta,\nu_{\bullet_{>0}}}=k^{\lambda_\bullet,\lambda;\mu,\mu_\bullet}_{\kappa_\bullet,\kappa;\nu,\nu_\bullet}$ holds whenever $p^{\xi,\lambda_{\bullet}}_{\kappa,\rho,\kappa_{\bullet_{>0}}} N^{\kappa_0}_{\rho\zeta} m^{\lambda;\mu}_{\xi;\eta} N^{\nu_0}_{\theta\zeta^\perp} p^{\eta,\mu_\bullet}_{\nu,\theta,\nu_{\bullet_{>0}}}\ne 0$. This establishes the equivalence of our two expressions.

The modules $\mc I_{\TT}^{k}(L_{\bm\lambda})$ are injective since, by Corollary \ref{Coro TT simples and inj hulls}, the modules $I_{\bm \kappa}$ for $\bm\kappa\in\bm\Lambda$ are indecomposable injectives in $\TT$. To show that we have the desired resolution, it remains to construct an exact sequence of morphisms $y_{k,\bm\lambda}:\mc I_{\bf T}^k(L_{\bm \lambda})\to \mc I_{\bf T}^{k+1}(L_{\bm \lambda})$, with ${\rm ker}y_{0,\bm\lambda}=L_{\bm\lambda}$.

We consider the triple K\"unneth product with $k$-th term
\begin{gather}\label{For inj res Llambdabullet otimes Lmubullet}
\mc I^k_{\TT}(L_{\lambda_\bullet,\lambda;\emptyset_\bullet} \otimes L_{\emptyset_\bullet;\mu,\mu_\bullet}) := \bigoplus\limits_{i+j_1+j_2=k} \mc I_{\TT}^i(\CC)\otimes \mc I_{\TT(V^*)}^{j_1}(L_{\lambda_\bullet,\lambda;\emptyset_\bullet})\otimes\mc I_{\TT(\bar V)}^{j_2}(L_{\emptyset_\bullet;\mu,\mu_\bullet}) \;,
\end{gather}
and we let $g_{k,\bm\lambda}: \mc I^k_{\TT}(L_{\lambda_\bullet,\lambda;\emptyset_\bullet} \otimes L_{\emptyset_\bullet;\mu,\mu_\bullet}) \to \mc I^{k+1}_{\TT}(L_{\lambda_\bullet,\lambda;\emptyset_\bullet} \otimes L_{\emptyset_\bullet;\mu,\mu_\bullet})$ be its $k$-th map. We have ${\rm ker}g_{0,\bm\lambda}=L_{\lambda_\bullet,\lambda;\emptyset_\bullet} \otimes L_{\emptyset_\bullet;\mu,\mu_\bullet}$, and hence (\ref{For inj res Llambdabullet otimes Lmubullet}) is an injective resolution of $L_{\lambda_\bullet,\lambda;\emptyset_\bullet} \otimes L_{\emptyset_\bullet;\mu,\mu_\bullet}$ in $\TT$. We shall modify this resolution into a resolution of the simple module $L_{\lambda_\bullet,\lambda;\mu,\mu_\bullet}$ using the fact that, by Proposition \ref{Prop TT tensor prod simples},
$$
L_{\lambda_\bullet,\lambda;\mu,\mu_\bullet}={\rm soc}(L_{\lambda_\bullet,\lambda;\emptyset_\bullet}\otimes L_{\emptyset_\bullet;\mu,\mu_\bullet})\cong L_{\lambda_\bullet,\emptyset;\emptyset,\mu_\bullet}\otimes{\rm soc}(L_{\lambda;\emptyset}\otimes L_{\emptyset;\mu}) \;.
$$
We note that for every $\xi,\eta\in\Lambda$ such that $m^{\lambda;\mu}_{\xi;\eta}\ne 0$ and $k^{\lambda;\mu}_{\xi;\eta}=1$ we have $m^{\lambda;\mu}_{\xi;\eta}=1$. Thus
$$
\bigoplus\limits_{\xi,\eta\in\Lambda:k^{\lambda;\mu}_{\xi;\eta}=1} m^{\lambda;\mu}_{\xi;\eta}\cdot \mc I^0_{\TT}(L_{\lambda_\bullet,\xi;\emptyset_\bullet} \otimes L_{\emptyset_\bullet;\eta,\mu_\bullet}) \cong \bigoplus\limits_{\xi,\eta\in\Lambda:m^{\lambda;\mu}_{\xi;\eta}=k^{\lambda;\mu}_{\xi;\eta}=1} I_{\lambda_\bullet,\xi;\eta,\mu_\bullet} \;.
$$
Let
$$
w_{\bm\lambda}:= (\bigoplus\limits_{f\in\Xi^1(I_{\lambda_\bullet,\lambda;\mu,\mu_\bullet})\; \textrm{of type (iii)}} f): I_{\bm\lambda} \to \bigoplus\limits_{\xi,\eta\in\Lambda:m^{\lambda;\mu}_{\xi;\eta}=k^{\lambda;\mu}_{\xi;\eta}=1} I_{\lambda_\bullet,\xi;\eta,\mu_\bullet}\;.
$$
We obtain a morphism
\begin{gather}\label{For differential 0 InjLbmlambda}
y_{0,\bm\lambda}=g_{0,\bm\lambda}\oplus w_{\bm\lambda} : \mc I^0_{\TT}(L_{\lambda_\bullet,\lambda;\mu,\mu_\bullet}) \to \mc I^1_{\TT}(L_{\lambda_\bullet,\lambda;\mu,\mu_\bullet})
\end{gather}
with the properties ${\rm ker}y_{0,\bm\lambda}\cong L_{\lambda_\bullet,\lambda;\mu,\mu_\bullet}$ and ${\rm image}y_{0,\bm\lambda} \cap \mc I^0_{\TT}(L_{\lambda_\bullet,\xi;\emptyset_\bullet} \otimes L_{\emptyset_\bullet;\eta,\mu_\bullet}) \cong L_{\lambda_\bullet,\xi;\eta,\mu_\bullet}$. We proceed to define
$$
y_{1,\bm\lambda}:=g_{1,\bm\lambda} \oplus (\bigoplus\limits_{\xi,\eta\in\Lambda:m^{\lambda;\mu}_{\xi;\eta}=k^{\lambda;\mu}_{\xi;\eta}=1} y_{0,(\lambda_\bullet,\xi;\eta,\mu_\bullet)})
$$
and, more generally,
$$
y_{k,\bm\lambda}:= (\bigoplus\limits_{\xi,\eta\in\Lambda:0\leq k^{\lambda;\mu}_{\xi;\eta}\leq k} (g_{k-k^{\lambda;\mu}_{\xi;\eta},(\lambda_\bullet,\xi;\eta,\mu_\bullet)})^{\oplus m^{\lambda;\mu}_{\xi;\eta}}) \oplus (\bigoplus\limits_{\xi,\eta\in\Lambda:k^{\lambda;\mu}_{\xi;\eta}=k} (w_{(\lambda_\bullet,\xi;\eta,\mu_\bullet)})^{\oplus m^{\lambda;\mu}_{\xi;\eta}}) \;.
$$
It follows by induction on $|\lambda\cap\mu^\perp|$ (which is the injective length of $V_{\lambda;\mu}$ in $\TT(V_*,V)$), using the Koszulity of the category $\TT(V_*,V)$, that the morphisms $y_{k,\bm\lambda}$ form an exact sequence.
\end{proof}

\begin{coro}\label{Cor dimExt in TT}
Let $(\lambda_\bullet,\lambda;\mu,\mu_\bullet),(\kappa_\bullet,\kappa;\nu,\nu_\bullet)\in\bm\Lambda$. Then, for $k\geq 0$,
\begin{gather*}
\begin{array}{l}
\dim{\rm Ext}_{\TT}^k(L_{\kappa_\bullet,\kappa;\nu,\nu_\bullet},L_{\lambda_\bullet,\lambda;\mu,\mu_\bullet}) =\\ \qquad = \sum\limits_{\stackrel{\xi,\eta,\zeta,\rho,\theta\in\Lambda}{|\lambda|-|\xi|+|\zeta|+k^{\xi,\lambda_{\bullet}}_{\kappa,\rho,\kappa_{\bullet_{>0}}}+k^{\eta,\mu_\bullet}_{\nu,\theta,\nu_{\bullet_{>0}}} = k}}  p^{\xi,\lambda_{\bullet}}_{\kappa,\rho,\kappa_{\bullet_{>0}}} N^{\kappa_0}_{\rho\zeta} m^{\lambda;\mu}_{\xi;\eta} N^{\nu_0}_{\theta\zeta^\perp} p^{\eta,\mu_\bullet}_{\nu,\theta,\nu_{\bullet_{>0}}}\end{array}\;.
\end{gather*}
If ${\rm Ext}_{\TT}^k(L_{\kappa_\bullet,\kappa;\nu,\nu_\bullet},L_{\lambda_\bullet,\lambda;\mu,\mu_\bullet})\ne0$, then $k=k^{\lambda_\bullet,\lambda;\mu,\mu_\bullet}_{\kappa_\bullet,\kappa;\nu,\nu_\bullet}$.
\end{coro}

\section{The category ${\bf T}_t$}

Recall Proposition \ref{Prop I is algebra} stating that the $\mk{gl}^M$-module $I$ is endowed with a structure of a commutative algebra via the isomorphism $I\cong S^\bullet Q/(1-\iota(1))$. Generalizing a concept introduced in \cite{Chirvasitu-Penkov-UTC}, we define a category ${\bf T}_t$ as the category of $(I,\mk{gl}^M)$-bimodules which are objects of $\TT$ and are free as $I$-modules. The category ${\bf T}_t$ is a tensor category with respect to $\otimes_I$. Since $t$ is fixed, we put ${\bf T}={\bf T}_t$. We note that the functor $I\otimes\bullet:\TT\to{\bf T}$ is left adjoint to the forgetful functor ${\bf T}\to\TT$. The simple objects in ${\bf T}$ are related to those in $\TT$ as follows.

\begin{theorem}\label{Theo L to IL} (\cite[Theorem~3.24]{Chirvasitu-Penkov-UTC})
The simple objects in ${\bf T}$ are exactly the modules of the form $I\otimes L$ with $L$ - a simple object in $\TT$. Furthermore, each simple object in ${\bf T}$ has endomorphism algebra isomorphic to $\CC$.

Consequently, the isomorphism classes of simple objects in ${\bf T}$ are parametrized by the set $\bm\Lambda=\Lambda^{2(t+2)}$ of $2(t+2)$-tuples of Young diagrams, with representatives (see (\ref{For LJITlambda}))
$$
K_{\bm\lambda}=I\otimes L_{\bm\lambda} \;,\; \bm\lambda\in\bm\Lambda \;.
$$
\end{theorem}

The proof given in \cite{Chirvasitu-Penkov-UTC} is independent of the assumption $t=0$ made in that article.

\begin{prop}
There is a surjective morphism of $\mk{gl}^M$-modules
$$
{_I}{\bf p}: (I\otimes V^*) \otimes_I (I\otimes \bar V) \to I \;.
$$
\end{prop}

\begin{proof}
The claimed morphism is the following composition
$$
(I\otimes V^*) \otimes_I (I\otimes \bar V) \cong I \otimes (V^*\otimes \bar V) \stackrel{{\rm id}\otimes\tilde\pi}{\lw} I\otimes Q \stackrel{multiply}{\lw} I \;. 
$$
It is surjective, since $\CC\cong \mk q\subset Q$.
\end{proof}

\begin{prop}\label{Prop ITsoc Ilambdamu}
Let $\bm\lambda=(\lambda_\bullet,\lambda;\mu,\mu_\bullet)\in \bm\Lambda$. The injective object $I_{\bm\lambda}$ has finite length in ${\bf T}$. The socle filtration of $I_{\bf\lambda}$ in ${\bf T}$ has length $1+q^{(|\bm\lambda|)}$ and its layers are
\begin{align*}
\ul{\rm soc}_{{\bf T}}^{q+1}I_{\lambda_\bullet,\lambda;\mu,\mu_\bullet} & = I\otimes \ul{\rm soc}_{\TT}^{q+1}J_{\lambda_\bullet,\lambda;\mu,\mu_\bullet}\\
& \cong \bigoplus\limits_{j+k=q} \bigoplus\limits_{\xi,\eta\in\Lambda} {\rm Hom}(V_{\xi,\eta},\ul{\rm soc}_{\TT}^{j+1}(V_{\lambda;\emptyset}\otimes V_{\emptyset;\mu}))\otimes I\otimes Z^{k+1}_{\lambda_\bullet,\xi;\eta,\mu_\bullet}\\
& \cong \bigoplus\limits_{j+k=q} \bigoplus\limits_{\xi,\eta\in\Lambda:|\lambda|-|\xi|=j} h^{\lambda;\mu}_{\xi;\eta} \cdot I\otimes Z^{k+1}_{\lambda_\bullet,\xi;\eta,\mu_\bullet} \;,
\end{align*}
where $Z^{k+1}_{\lambda_\bullet,\xi;\eta,\mu_\bullet}$ are the $\mk{gl}^M$-modules defined in (\ref{For module Sk}) and $h^{\lambda;\mu}_{\xi;\eta}$ are the numbers defined in (\ref{For h lambdamuxieta}).
\end{prop}

\begin{proof}
Note that the finiteness of the length of $I_{\bm\lambda}$ follows from the proposed description of the socle filtration, because the multiplicities of simple objects in the (finitely many) socle layers are finite. Next, recall that the socle filtration of $I_{\bm\lambda}$ as a $\mk{gl}^M$-module is known from Proposition \ref{Prop soc q Ilambdamu}. Theorem \ref{Theo L to IL} allows us to determine the simple subquotients of $I_{\bm\lambda}$ in ${\bf T}$ and observe that they correspond to the simple subquotients of $J_{\bm\lambda}$ in $\TT$. To prove the first line of the formula claimed in the theorem, it remains to show that the number of the layer in which a given simple subquotient of $J_{\bm\lambda}$ appears remains the same for the respective simple subquotient of $I_{\bm\lambda}$ in ${\bf T}$. This holds, since ${\rm soc}_\TT I_{\bm\kappa} = L_{\bm\kappa}$ for every $\bm\kappa\in\bm\Lambda$, and the $\TT$-socle filtration of $I_{\bm\lambda}$ is subordinate to the ${\bf T}$-socle filtration. This implies the first line, and the rest follows from Proposition \ref{Prop soc q Jlambdamu} describing $\ul{\rm soc}_{\TT}^{q+1}J_{\bm\lambda}$.
\end{proof}

We are now ready to prove the following generalization of \cite[Proposition 3.25]{Chirvasitu-Penkov-UTC} where the result is obtained for $t=0$.

\begin{theorem}\label{Theo ITT is Grothendieck Cat}
The category ${\bf T}$ is an ordered Grothendieck category with order-defining objects $I_{\bm l}$, $\bm l\in{\bf P}$, parametrized by the poset ${\bf P}$ of Definition \ref{Def Poset IP}. The isomorphism classes of simple objects in ${\bf T}$ are parametrized by the set $\bm\Lambda$, with representatives $K_{\bm\lambda}$, $\bm\lambda\in\bm\Lambda$. The indecomposable injectives are, up to isomorphism, $I_{\bm\lambda}$, $\bm\lambda\in\bm\Lambda$. The socles of the order-defining objects are
\begin{gather*}
{\rm soc}_{{\bf T}}I_{\bm l} = K_{\bm l} = \bigoplus\limits_{\bm\lambda\in \mc S_{\bm l}} \CC^{\bm\lambda}\otimes K_{\bm\lambda} \;,
\end{gather*}
with $\mc S_{\bm l}=\{\bm\lambda\in\bm\Lambda:|\bm\lambda|=\bm l\}$ as in Theorem \ref{Theo TT Grothe Order}.
\end{theorem}

\begin{proof}
From Proposition \ref{Prop ITsoc Ilambdamu} we deduce that
\begin{gather}\label{For Relation socIT to socT}
\ul{\rm soc}_{{\bf T}}^{k+1}I_{\bm l} = I\otimes \ul{\rm soc}_{\TT}^{k+1}J_{\bm l} \;.
\end{gather}
Now, the theorem follows by arguments analogous to these in the proof of Theorem \ref{Theo ITT is Grothendieck Cat}, using the socle filtration of $J_{\bm l}$ determined in Proposition \ref{Prop socqJlm} where the layers correspond to indices $\bm k\stackrel{\bf P}{\preceq}\bm l$.
\end{proof}

\subsection{Tensor products of simple objects and the subcategory ${\bf T}(K_{1;0},K_{0;1})$}

We make here some technical observations which will be used further on for the construction of injective resolutions of simple objects in ${\bf T}$.

\begin{prop}\label{Prop IT tensor prod simples}
Let $\bf\lambda=(\lambda_\bullet,\lambda;\mu,\mu_\bullet),\bf\lambda'=(\lambda'_\bullet,\lambda';\mu',\mu'_\bullet)\in\bm\Lambda$. Then, for $q\geq 0$, we have ${\rm soc}_{\bf T}^{q+1}(K_{\bm\lambda}\otimes_I K_{\bm\lambda'}) \cong I\otimes {\rm soc}_{\TT}^{q+1}(L_{\bm\lambda}\otimes L_{\bm\lambda'})$ and
\begin{align*}
\ul{\rm soc}_{\bf T}^{q+1}(K_{\bm\lambda}\otimes_I K_{\bm\lambda'}) & \cong I\otimes \ul{\rm soc}_{\TT}^{q+1}(L_{\bm\lambda}\otimes L_{\bm\lambda'})\\ & \cong K_{\lambda_\bullet,\emptyset;\emptyset,\mu_\bullet}\otimes_I K_{\lambda'_\bullet,\emptyset;\emptyset,\mu'_\bullet} \otimes_I \ul{\rm soc}_{\bf T}^{q+1}(K_{\lambda;\mu}\otimes_I K_{\lambda';\mu'}) \\
& \cong K_{\lambda_\bullet,\emptyset;\emptyset,\mu_\bullet}\otimes_I K_{\lambda'_\bullet,\emptyset;\emptyset,\mu'_\bullet} \otimes_I \left(\bigoplus\limits_{\kappa,\nu\in\Lambda:|\lambda|+|\lambda'|-|\kappa|=q}(\sum\limits_{\xi,\eta\in\Lambda} N^{\xi}_{\lambda\lambda'}N^{\eta}_{\mu\mu'} h^{\xi;\eta}_{\kappa;\nu}) \cdot K_{\kappa;\nu} \right) \\
& \cong \bigoplus\limits_{(\kappa_\bullet,\kappa,\nu,\nu_\bullet)\in\bm\Lambda:|\lambda|+|\lambda'|-|\kappa|=q} \left({\bf N}^{\kappa_\bullet}_{\lambda_\bullet\lambda'_\bullet} {\bf N}^{\nu_\bullet}_{\mu_\bullet\mu'_\bullet}\sum\limits_{\xi,\eta\in\Lambda} N^{\xi}_{\lambda\lambda'}N^{\eta}_{\mu\mu'} h^{\xi;\eta}_{\kappa;\nu} \right) \cdot K_{\kappa_\bullet,\kappa;\nu,\nu_\bullet}  \;,
\end{align*}
where the numbers $h^{\lambda;\mu}_{\kappa;\nu}$ and ${\bf N}^{\kappa_\bullet}_{\lambda_\bullet\lambda'_\bullet}$ are given respectively in (\ref{For h lambdamuxieta}) and Proposition \ref{Prop TT tensor prod simples}.

In particular, the necessary and sufficient condition for semisimply of $L_{\bm\lambda} \otimes L_{\bm\lambda'}$ in $\TT$, given in part (c) of Proposition \ref{Prop TT tensor prod simples}, holds as well for $K_{\bm\lambda} \otimes_I K_{\bm\lambda'}$ in ${\bf T}$.
\end{prop}

\begin{proof}
The socle filtrations of the $\mk{gl}^M$ modules $V_{*\lambda}\otimes V_\mu$ remain unaltered after restriction to the ideal $\mk{sl}(V,V_*)\subset\mk{gl}(V)$, by Theorem \ref{Theo PS soc filt Vlambda TENS Vmu}. On the other hand, $\mk{sl}(V,V_*)$ acts trivially on $I$. We can apply Proposition \ref{Prop DenseProperties} that the claim holds for the socle filtration of a tensor product of the form $K_{\lambda;\emptyset}\otimes_I K_{\emptyset;\mu'}$. Now the general statement follows from Proposition \ref{Prop TT tensor prod simples} in a straightforward manner.
\end{proof}

Let $\ul{\bf T}={\bf T}(K_{1;0},K_{0;1})$ be the smallest full tensor Grothendieck subcategory of ${\bf T}$ containing the objects $K_{1;0}=I\otimes V_*$ and $K_{0;1}=I\otimes V$ and closed under taking subquotients.

\begin{theorem}\label{Theo smallIT smallTT}
The categories $\TT(V_*,V)$ and $\ul{\bf T}$ are equivalent under the functor $I\otimes\bullet$. This functor is also determined by the universal property of $\TT(V_*,V)$ and the assignment $V_*\mapsto K_{1;0}$, $V\mapsto K_{0;1}$, ${\bf p} \mapsto (K_{1;0}\otimes_I K_{0;1}\cong I\otimes V_*\otimes V\stackrel{{\rm id}\otimes{\bf p}}{\to} I)$. In particular, $\ul{\bf T}$ has the structure of an ordered Grothendieck category, with order-defining objects $I_{l;m}$, $(l;m)\in\NN\times\NN$, parametrized by the poset $\mc P_{0,0}$ from Definition \ref{Def Poset 0 0}. Representatives of the isomorphism classes of simple objects and indecomposable injective objects of $\ul{\bf T}$ are given respectively by $K_{\lambda;\mu}$ and $I_{\lambda;\emptyset}\otimes_I I_{\emptyset;\mu}$ for $(\lambda;\mu)\in\Lambda\times\Lambda$.
\end{theorem}

\begin{proof}
The existence of the claimed functor is due to the universality property of $\TT(V_*,V)$, cf. \cite{Penkov-Serganova-Mackey},\cite{Chirvasitu-Penkov-OTC}. The fact that $I\otimes\bullet$ fits exactly with the required assignment for the universal functor is obvious. The verification that this functor defines an equivalence in our case is straightforward in view of Proposition \ref{Prop IT tensor prod simples}.
\end{proof}

Theorem \ref{Theo smallIT smallTT} allows us to translate the results from Section \ref{Sec Vstar l po V m} into results about the category $\ul{\bf T}$. In particular, Theorem \ref{Theo smallTT Inj n Ext to Hom} yields the following.

\begin{coro}\label{Coro smallIT Ext}
For $\lambda,\mu,\xi,\eta\in\Lambda$ and $k\geq 0$, we have
$$
\dim{\rm Ext}^k_{\ul{\bf T}}(K_{\xi;\eta},K_{\lambda,\mu}) = \dim{\rm Ext}^k_{\TT(V_*,V)}(V_{\xi;\eta},V_{\lambda,\mu}) = m^{\lambda;\mu}_{\xi,\eta} \;.
$$
If this dimension is nonzero then $k=k^{\lambda;\mu}_{\xi;\eta}=|\lambda|-|\xi|=|\mu|-|\eta|$.
\end{coro}

\subsection{Injective resolutions of simple objects in ${\bf T}_t$}

\begin{theorem}\label{Theo InjResILlambdamu IT}
Let $\bm\lambda=(\lambda_\bullet,\lambda;\mu,\mu_\bullet)\in \bm\Lambda$. There is an injective resolution of the simple object $K_{\bm\lambda}$ in ${\bf T}$, of length $k^{\bm\lambda}:=||\bm\lambda||-(|\lambda_t|+|\mu_t|)$ and with $k$-th term
\begin{align*}
\mc I_{\bf T}^{k}(K_{\lambda_\bullet,\lambda;\mu,\mu_\bullet}) & \cong \bigoplus\limits_{i+j_1+j_2=k} \bigoplus\limits_{\xi,\eta\in\Lambda: k^{\lambda;\mu}_{\xi;\eta}=i} m^{\lambda;\mu}_{\xi;\eta}\cdot \left( \mc I_{\bf T}^{j_1}(K_{\lambda_\bullet,\xi;\emptyset_\bullet})\otimes_I\mc I_{\bf T}^{j_2}(K_{\emptyset_\bullet;\eta,\mu_\bullet}) \right) \\
& \cong \bigoplus\limits_{i+j_1+j_2=k} \bigoplus\limits_{\xi,\eta\in\Lambda: k^{\lambda;\mu}_{\xi;\eta}=i}  m^{\lambda;\mu}_{\xi;\eta} \cdot I\otimes  \mc I_{\TT(V^*)}^{j_1}(L_{\lambda_\bullet,\xi;\emptyset_\bullet})\otimes\mc I_{\TT(\bar V)}^{j_2}(L_{\emptyset_\bullet;\eta,\mu_\bullet}) \\
&\cong \bigoplus\limits_{(\kappa_\bullet,\kappa;\nu,\nu_\bullet)\in\bm\Lambda:k^{\lambda_\bullet,\lambda;\mu,\mu_\bullet}_{\kappa_\bullet,\kappa;\nu,\nu_\bullet}=k} \left(\sum\limits_{\xi,\eta\in\Lambda} p^{\xi,\lambda_\bullet}_{\kappa,\kappa_\bullet} m^{\lambda;\mu}_{\xi;\eta} p^{\eta,\mu_\bullet}_{\nu,\nu_\bullet} \right) \cdot I_{\kappa_\bullet,\kappa;\nu,\nu_\bullet} \;,
\end{align*}
where $k^{\lambda_\bullet,\lambda;\mu,\mu_\bullet}_{\kappa_\bullet,\kappa;\nu,\nu_\bullet}$ is as in Theorem \ref{Theo InjResLlambdamu TT}.
\end{theorem}

\begin{proof}
The strategy relays on the fact that the tensor product of injective objects in ${\bf T}$ is injective, which allows us to apply Lemma \ref{Lemma TensInjRes}.

We begin with the one-sided case, and the observation that the injective resolution of $L_{\lambda_\bullet,\lambda;\emptyset_\bullet}$ in $\TT(V^*)$ (see Theorem \ref{Theo InjResLlambdabullet}) is transformed under the functor $I\otimes\bullet$ into an exact sequence with $j$-th term $\mc I_{\bf T}^{j}(K_{\lambda_\bullet,\lambda;\emptyset_\bullet}):=I\otimes \mc I_{\TT(V^*)}^{j}(L_{\lambda_\bullet,\lambda;\emptyset_\bullet})$. The indecomposable injectives of $\TT(V^*)$ are of the form $J_{\kappa_\bullet,\kappa;\emptyset_\bullet}$, and hence $I\otimes\bullet$ transforms them into indecomposable injectives of ${\bf T}$, by Theorem \ref{Theo ITT is Grothendieck Cat}. In particular, $\mc I_{\bf T}^{j}(K_{\lambda_\bullet,\lambda;\emptyset_\bullet})$ is injective in ${\bf T}$ for all $j$, and our exact sequence is an injective resolution of $K_{\lambda_\bullet,\lambda;\emptyset_\bullet}$ in ${\bf T}$. The case of $\mc I_{\bf T}^{j}(K_{\emptyset_\bullet;\mu,\mu_\bullet}):=I\otimes \mc I_{\TT(\bar V)}^{j}(L_{\emptyset_\bullet;\mu,\mu_\bullet})$ is analogous.

By Lemma \ref{Lemma TensInjRes}, the K\"unneth product of the resolutions of $K_{\lambda_\bullet,\lambda;\emptyset_\bullet}$ and $K_{\emptyset_\bullet;\mu,\mu_\bullet}$ is a resolution of $K_{\lambda_\bullet,\lambda;\emptyset_\bullet}\otimes_I K_{\emptyset_\bullet;\mu,\mu_\bullet}$, with $k$-th term
$$
\bigoplus\limits_{j_1+j_2=k} \mc I_{\bf T}^{j_1}(K_{\lambda_\bullet,\lambda;\emptyset_\bullet})\otimes_I\mc I_{\bf T}^{j_2}(K_{\emptyset_\bullet;\mu,\mu_\bullet}) \;.
$$
We have an analogous resolution of $K_{\lambda_\bullet,\xi;\emptyset_\bullet}\otimes_I K_{\emptyset_\bullet;\eta,\mu_\bullet}$ for $\xi,\eta\in\Lambda$ such that $m^{\lambda;\mu}_{\xi;\eta}$, and we can combine these resolutions, using Corollary \ref{Coro smallIT Ext}, in a manner similar to the one in the proof of Theorem \ref{Theo InjResLlambdamu TT}, to obtain the claimed resolution of $K_{\lambda_\bullet,\lambda;\mu,\mu_\bullet}$.
\end{proof}

\begin{coro}\label{Coro dimExt Tkk Tll}
For $(\kappa_\bullet,\kappa;\nu,\nu_\bullet),(\lambda_\bullet,\lambda;\mu,\mu_\bullet)\in\bm\Lambda$ and $q\geq0$, we have
$$
\dim {\rm Ext}_{{\bf T}}^q (K_{\kappa_\bullet,\kappa;\nu,\nu_\bullet},K_{\lambda_\bullet,\lambda;\mu,\mu_\bullet}) = \sum\limits_{\xi,\eta\in\Lambda:q=k^{\xi,\lambda_\bullet}_{\kappa,\kappa_\bullet} +k^{\lambda;\mu}_{\xi;\eta} +k^{\eta,\mu_\bullet}_{\nu,\nu_\bullet}} p^{\xi,\lambda_\bullet}_{\kappa,\kappa_\bullet} m^{\lambda;\mu}_{\xi;\eta} p^{\eta,\mu_\bullet}_{\nu,\nu_\bullet} \;.
$$
If ${\rm Ext}_{{\bf T}}^q (K_{\kappa_\bullet,\kappa;\nu,\nu_\bullet},K_{\lambda_\bullet,\lambda;\mu,\mu_\bullet})\ne 0$, then $q=k^{\lambda_\bullet,\lambda;\mu,\mu_\bullet}_{\kappa_\bullet,\kappa;\nu,\nu_\bullet}$.
\end{coro}

\begin{coro}\label{Coro Ext to Hom in IT}
For $\bm\kappa,\bm\lambda\in\bm\Lambda$ and $q\geq0$,
$$
\dim {\rm Ext}_{{\bf T}}^k (K_{\bm\kappa^{e\perp o}},K_{\bm\lambda^{e\perp o}})= \dim {\rm Ext}_{{\bf T}}^k (K_{\bm\kappa^{o\perp e}},K_{\bm\lambda^{o\perp e}}) = \dim {\rm Hom}_{\bf T}(K_{\bm\kappa},\ul{\rm soc}_{\bf T}^{k+1}(I\otimes M_{\bm\lambda}))\;,
$$
where $M_{\lambda_\bullet,\lambda;\mu,\mu_\bullet}$ is the module defined in (\ref{For Modules M}) and ${^{e\perp o}}$ is the involution defined in \S\ref{Sec InvolutLambdaBeth}.
\end{coro}

\begin{proof}
The corollary follows from Corollary \ref{Coro dimExt Tkk Tll}, Corollary \ref{Coro Ext to Hom in TVstar}, Theorem \ref{Theo smallTT Inj n Ext to Hom}.
\end{proof}

As a special case, we obtain the following.

\begin{coro}\label{Coro Ext to Hom IT aleph0}
Assume $t=0$, meaning that $V$ is of countable dimension. Then
$$
\dim {\rm Ext}_{{\bf T}_0}^k (K_{\kappa_0^\perp,\kappa;\nu^\perp,\nu_0},K_{\lambda_0^\perp,\lambda;\mu^\perp,\mu_0}) = \dim {\rm Hom}_{{\bf T}_0}(K_{\kappa_0,\kappa;\nu,\nu_0},\ul{\rm soc}_{{\bf T}_0}^{k+1}I_{\lambda_0,\lambda;\mu,\mu_0})
$$
holds for any $(\kappa_0,\kappa;\nu,\nu_0), (\lambda_0,\lambda;\mu,\mu_0)\in\bm\Lambda$ and $k\geq0$.
\end{coro}

\begin{proof}
The cardinal $\aleph_0$ is the only infinite cardinal for which $M_{\bm\lambda}=J_{\bm\lambda}$ holds for all $\bm\lambda\in\bm\Lambda$. Now the statement is given by Corollary \ref{Coro Ext to Hom in IT}.
\end{proof}

\section{Symmetries}\label{Sec Sym}

If $\lambda_\bullet=(\lambda_{-1},\lambda_0,...,\lambda_t)$ is a finite sequence of Young diagrams, we denote by ${\rm rev}\lambda_\bullet=(\lambda_t,...\lambda_0,\lambda_{-1})$ the reversed sequence.

\begin{prop}\label{Prop REV SYM}
Let $(\lambda_\bullet;\mu_\bullet),(\kappa_\bullet;\nu_\bullet)\in\bm\Lambda$ and $q\geq 0$. Then:
\begin{enumerate}
\item $p^{\lambda_\bullet}_{\kappa_\bullet} = p_{{\rm rev}\lambda_\bullet}^{{\rm rev}\kappa_\bullet}$ and if this number is nonzero then $k^{\lambda_\bullet}_{\kappa_\bullet} = k_{{\rm rev}\lambda_\bullet}^{{\rm rev}\kappa_\bullet}$; analogously $p^{\mu_\bullet}_{\nu_\bullet} = p_{{\rm rev}\nu_\bullet}^{{\rm rev}\mu_\bullet}$ and if this number is nonzero then $k^{\lambda_\bullet}_{\kappa_\bullet} = k_{{\rm rev}\lambda_\bullet}^{{\rm rev}\kappa_\bullet}$;
\item $\dim{\rm Hom}_\TT(L_{\kappa_\bullet;\emptyset_\bullet},\ul{\rm soc}_\TT^{q+1}J_{\lambda_\bullet;\emptyset_\bullet}) = \dim{\rm Hom}_\TT(L_{{\rm rev}\lambda_\bullet;\emptyset_\bullet},\ul{\rm soc}_\TT^{q+1}J_{{\rm rev}\kappa_\bullet;\emptyset_\bullet})$;
\item $\dim{\rm Ext}_\TT^q(L_{\kappa_\bullet;\emptyset_\bullet},L_{\lambda_\bullet;\emptyset_\bullet}) = \dim{\rm Ext}_\TT^q(L_{{\rm rev}\lambda_\bullet;\emptyset_\bullet},L_{{\rm rev}\kappa_\bullet;\emptyset_\bullet})$;
\item $\dim{\rm Hom}_{\bf T}(K_{\kappa_\bullet;\emptyset_\bullet},\ul{\rm soc}_{\bf T}^{q+1}I_{\lambda_\bullet;\emptyset_\bullet}) = \dim{\rm Hom}_{\bf T}(K_{{\rm rev}\lambda_\bullet;\emptyset_\bullet},\ul{\rm soc}_{\bf T}^{q+1}I_{{\rm rev}\kappa_\bullet;\emptyset_\bullet})$;
\item $\dim{\rm Ext}_{\bf T}^q(K_{\kappa_\bullet;\emptyset_\bullet},K_{\lambda_\bullet;\emptyset_\bullet}) = \dim{\rm Ext}_{\bf T}^q(K_{{\rm rev}\lambda_\bullet;\emptyset_\bullet},K_{{\rm rev}\kappa_\bullet;\emptyset_\bullet})$;
\item we have
$$
\dim{\rm Ext}^q_\TT(L_{\kappa_0,\kappa;\nu,\nu_0},L_{\lambda_0,\lambda;\mu,\mu_0}) = \dim{\rm Ext}^q_\TT(L_{\lambda,\lambda_0;\mu_0,\mu},L_{\kappa,\kappa_0;\nu_0,\nu})
$$
$$
= \sum\limits_{\delta,\tau,\theta,\varphi,\psi,\xi,\eta,\zeta\in\Lambda} N^{\kappa_0}_{\zeta\varphi} N^{\varphi}_{\lambda_0\tau} N^{\xi}_{\tau^\perp\kappa} N^{\lambda}_{\xi\delta} N^{\mu}_{\delta^\perp\eta} N^{\eta}_{\nu\theta^\perp} N^{\psi}_{\theta\mu_0} N^{\nu_0}_{\psi\zeta^\perp} \;,
$$
and if this number is nonzero then $q$ is unique and equals
$$
q=|\kappa_0|-|\lambda_0| + |\mu|-|\nu| = |\lambda|-|\kappa| + |\nu_0|-|\mu_0|\;;
$$
\item if $|\lambda_0|+|\lambda|=|\kappa_0|+|\kappa|$ we have
$$
\dim{\rm Ext}^q_{\bf T}(K_{\kappa_0,\kappa;\nu,\nu_0},K_{\lambda_0,\lambda;\mu,\mu_0}) = \dim{\rm Ext}^q_{\bf T}(K_{\lambda,\lambda_0;\mu_0,\mu},K_{\kappa,\kappa_0;\nu_0,\nu})
$$
$$
= \sum\limits_{\tau,\theta\in\Lambda} N^{\kappa_0}_{\lambda_0\tau} N^{\lambda}_{\tau^\perp\kappa} N^{\mu}_{\nu\theta^\perp} N^{\nu_0}_{\theta\mu_0} \;,
$$
and if this number is nonzero then $q$ is unique and equals
$$
q=|\lambda|-|\kappa|+|\mu|-|\nu|=|\kappa_0|-|\lambda_0| + |\mu|-|\nu| = |\lambda|-|\kappa| + |\nu_0|-|\mu_0|\;.
$$
\end{enumerate}
\end{prop}

\begin{proof}
The first statement follows by standard properties of Littlewood-Richardson coefficients from the defining formulas of $p^{\lambda_\bullet}_{\mu_\bullet}$ and $k^{\lambda_\bullet}_{\mu_\bullet}$ in Theorem \ref{Theo InjResLlambdabullet}. The rest of the statements follow from the first and the explicit formulas for the dimensions of the involved Ext- and Hom-groups, obtained in Proposition \ref{Prop soc q Jlambdamu}, Corollary \ref{Cor dimExt in TT} and Corollary \ref{Coro dimExt Tkk Tll}.
\end{proof}

\section{Universality}\label{Sec Univ}

Before addressing the topic of universality we should point out that a seed of the following discussion can be traced to the work \cite{Sam-Snowden}. Here we follow \cite{Chirvasitu-Penkov-UTC}.

Let ${\bf T}_{\rm fin}$ denote the full tensor subcategory of ${\bf T}$ containing $I_{\bm l}$ for $\bm l\in{\bf P}$ and closed under taking subquotients. The goal of this section is to prove the following theorem.

\begin{theorem}\label{Theo Univ IT}
Let $t\in\NN$. Let $(\mc D, \otimes, \bm 1)$ be a ($\CC$-linear abelian) tensor category with a given pair of objects $X,Y$, a morphism
\begin{gather}\label{For D pairing}
{\bf q}:X\otimes Y \to \bm 1 \;,
\end{gather}
and filtrations $0=X_{-1}\subset X_{0}\subset X_1\subset ... \subset X_{t+1}=X$ and $0=Y_{-1}\subset Y_{0}\subset Y_1\subset ... \subset Y_{t+1}=Y$. Then the following hold.
\begin{enumerate}
\item[{\rm (i)}] There is a unique, up to a monoidal isomorphism, left exact symmetric monoidal functor $\Phi: {\bf T}_{\rm fin}\to \mc D$ sending the pairing ${_I}{\bf p}:(I \otimes V^*)\otimes_I (I\otimes \bar V)\to I$ to the pairing ${\bf q}$, and for $-1\leq\alpha<\beta\leq t$ the morphisms $I\otimes (V^*/V^*_{\alpha})\to I\otimes (V^*/V^*_\beta)$ and $I\otimes(\bar V/\bar V_{\alpha})\to I\otimes(\bar V/\bar V_\beta)$ respectively to the morphisms $X/X_{\alpha}\to X/X_\beta$ and $Y/Y_{\alpha}\to Y/Y_\beta$.
\item[{\rm (ii)}] If $\mc D$ is additionally a Grothendieck category then $\Phi$ extends to a functor ${\bf T}\to \mc D$. 
\end{enumerate}
\end{theorem}

The proof will be given after some preparation.

\begin{prop}\label{Prop End Ilm in IT}
For $\bm l\in{\bf P}$, the endomorphism algebra ${\rm End}_{{\bf T}}I_{\bm l}$ is isomorphic to the group algebra $\CC[\mk S_{\bm l}]$ via the $\mk S_{\bm l}$-action on $I_{\bm l}$, where the $\mk S_{l_\alpha}$-factor of $\mk S_{\bm l}$ permutes the tensorands of the tensorand $(V^*/V^*_\alpha)^{\otimes l_\alpha}$ of $I_{\bm l}$ and the $\mk S_{m_\alpha}$-factor permutes the tensorands of the tensorand $(\bar V/\bar V_\alpha)^{\otimes m_\alpha}$ of $I_{\bm l}$.
\end{prop}

\begin{proof}
We follow the idea of \cite[Lemma 3.34]{Chirvasitu-Penkov-UTC} and only outline the main steps as the details are analogous. By Theorems \ref{Theo L to IL} and \ref{Theo ITT is Grothendieck Cat}, the endomorphism algebra of every indecomposable injective is trivial: ${\rm End}_{{\bf T}}I_{\bm\lambda}\cong \CC$ for all $\bm\lambda\in\bm\Lambda$. The $\mk S_{\bm l}$-action defined in the proposition extends to an injective homomorphism $\CC[\mk S_{\bm l}]\hookrightarrow{\rm End}_{{\bf T}}I_{\bm l}$. The surjectivity follows from a dimension argument. 
\end{proof}

Let $R$ be the tensor algebra in ${\bf T}$ of the object $R_1:=\bigoplus\limits_{\bm l\in{\bf P}:|\bm l|=1} I_{\bm l}$ and let $R_d:=R_1^{\otimes_I d}$ be the degree $d$ component of $R$. Let
\begin{gather}\label{For Algebra A}
\mc A := \bigoplus\limits_{k,l\in\NN} {\rm Hom}_{{\bf T}}(R_l,R_k) \cong \bigoplus\limits_{\bm k,\bm l\in{\bf P}} {\rm Hom}_{{\bf T}}(I_{\bm l},I_{\bm k}) \;;
\end{gather}
this is an $\NN$-graded algebra with degree components
$$
\mc A_d:=\bigoplus\limits_{\bm k,\bm l\in{\bf P}:\bm k\in{\bf P}^d(\bm l)} {\rm Hom}_{{\bf T}}(I_{\bm l},I_{\bm k}) \;.
$$

\begin{theorem}\label{Theo IT Koszul A quadr}
The category ${\bf T}$ is Koszul, in the sense of \cite{Chirvasitu-Penkov-RC}, namely, for every pair of simple objects $K,L$ and every $q\geq 2$, the canonical Yoneda map
$$
\bigoplus\limits_{M_{1},...,M_{q-1}\; simple} {\rm Ext}^1(K,M_1)\otimes {\rm Ext}^1(M_1,M_2)\otimes...\otimes {\rm Ext}^1(M_{q-1},L) \to {\rm Ext}^q(K,L)
$$
is surjective. Consequently, the algebra $\mc A$ is Koszul and, in particular, quadratic.
\end{theorem}

\begin{proof}
The surjectivity of the Yoneda maps in ${\bf T}$ follows from Corollary \ref{Coro dimExt Tkk Tll}. It is shown in \cite{Chirvasitu-Penkov-UTC} that the Koszulity of the category ${\bf T}$ implies that $\mc A$ is a Koszul algebra, and is hence quadratic.  
\end{proof}

\begin{prop}\label{Prop A QuadRel}
The space of quadratic relations between degree 1 elements of $\mc A$ decomposes as a sum of monogenerated $(\mk S_{\bm l},\mk S_{\bm k})$-bimodules, along the pairs $\bm l,\bm k$ at distance 2 in the poset ${\bf P}$, as follows:
$$
{\rm ker}(\mc A_1\otimes\mc A_1\to\mc A_2) = \bigoplus\limits_{\bm k,\bm l\in{\bf P}:\bm k\in{\bf P}^2(\bm l)} {\rm ker}g_{\bm l,\bm k} \;,\quad {\rm ker}g_{\bm l,\bm k} \cong (\CC[\mk S_{\bm k}]\otimes \CC[\mk S_{\bm l}])\cdot f_{\bm l,\bm k}.
$$
Here
\begin{gather*}
g_{\bm l,\bm k}:\bigoplus\limits_{\bm k'\in{\bf P}:\bm l\stackrel{\bf P}{\succ}\bm k'\stackrel{\bf P}{\succ}\bm k} {\rm Hom}(I_{\bm k'},I_{\bm k})\otimes_{{\rm End}I_{\bm k'}}{\rm Hom}(I_{\bm l},I_{\bm k'}) \to {\rm Hom}(I_{\bm l},I_{\bm k})
\end{gather*}
is the morphism induced by composition, it is surjective, and a generator $f_{\bm l,\bm k}\in {\rm ker}g_{\bm l,\bm k}$ is specified below, case by case, along with the $(\mk S_{\bm l},\mk S_{\bm k})$-bimodule structure on the domain of $g_{\bm l,\bm k}$, in dependence of the occurring intermediate elements $\bm l\stackrel{\bf P}{\succ}\bm k'\stackrel{\bf P}{\succ}\bm k$. For $\bm l=(l_\bullet,l;m,n_\bullet)\in{\bf P}$, we let ${\bf p}_{\bm l}:I_{\bm l}\to I_{\bm l-(1;1)}$ stand for the morphism defined by the identity on all tensorands in $I_{\bm l}=I_{l_\bullet,0;0,n_\bullet}\otimes_I I_{1;0}^{\otimes_I l} \otimes I_{0;1}^{\otimes_I m}$, except on the last tensorand of $I_{1;0}^{\otimes_I l}$ and the last tensorand of $I_{0;1}^{\otimes_I m}$ on which ${_I}{\bf p}:I_{1;0}\otimes I_{0;1}\to I_{0;0}$ is applied. Similarly, for $0\leq\alpha\leq t$, $f^{(\alpha)}_{\bm l}:I_{\bm l}\to I_{\bm l+(1_\alpha,-1_{\alpha-1};0)}$ is the projection $V^*/V^*_{\alpha-1}\to V^*/V^*_{\alpha}$ applied to the last tensorand in $I_{1_{\alpha-1};0}^{\otimes_I l_{\alpha}}$, extended by identity on all other tensorands in $I_{\bm l}$, and $\bar f^{(\alpha)}_{\bm l}:I_{\bm l}\to I_{\bm l +(0;-1_{\alpha-1},1_{\alpha})}$ is the analogous morphism obtained from $\bar V/\bar V_{\alpha-1}\to \bar V/\bar V_{\alpha}$.

\begin{enumerate}
\item For $\bm l=(l_t,...,l_0,l;m,m_0,...,m_t)$, $\bm k=(l_t,...,l_0,l-2;m-2,m_0,...,m_t)=\bm l-(2;2)$ we have a single intermediate element $\bm k'=(l_t,...,l_0,l-1;m-1,m_0,...,m_t)=\bm l-(1;1)$; the domain of $g_{\bm l,\bm k}$ is
$$
{\rm Hom}(I_{\bm k'},I_{\bm k})\otimes_{{\rm End}I_{\bm k'}}{\rm Hom}(I_{\bm l},I_{\bm k'}) \cong \CC[\mk S_{\bm l}]
$$
as an $(\mk S_{\bm l},\mk S_{\bm k})$-bimodule, and the kernel of $g_{\bm l,\bm k}$ is generated by
$$
f_{\bm l,\bm k}={\bf p}_{\bm k'}\otimes{\bf p}_{\bm l} - {\bf p}_{\bm k'}\otimes{\bf p}_{\bm l}\circ s \;,
$$
where $s$ is the product of the two simple transpositions in $\mk S_{l}\times \mk S_{m}$ exchanging respectively the last two tensorands in $(V^*)^{\otimes l}$ and the last two tensorands in $\bar V^{\otimes m}$.
\item For $\bm l=(l_t,...,l_0,l;m,m_0,...,m_t)$, $\bm k=(l_t,...,l_0+1,l-2;m-1,m_0,...,m_t)$ we have two intermediate elements $\bm k'=(l_t,...,l_0,l-1;m-1,m_0,...,m_t)$, $\bm k''=(l_t,...,l_0+1,l-1;m,m_0,...,m_t)$; the domain of $g_{\bm l,\bm k}$ is
\begin{align*}
& {\rm Hom}(I_{\bm k'},I_{\bm k})\otimes_{{\rm End}I_{\bm k'}}{\rm Hom}(I_{\bm l},I_{\bm k'})\oplus {\rm Hom}(I_{\bm k''},I_{\bm k})\otimes_{{\rm End}I_{\bm k''}}{\rm Hom}(I_{\bm l},I_{\bm k''})\\
& \qquad\qquad\qquad \cong \CC[\mk S_{l_t,...,l_0+1,l;m,m_0,...,m_t}]^{\oplus 2}\cong ({\rm ind}^{\mk S_{l_0+1}}_{\mk S_{l_0}}\CC[\mk S_{\bm l}])^{\oplus 2}
\end{align*}
as an $(\mk S_{\bm l},\mk S_{\bm k})$-bimodule (the two summands are isomorphic), and the kernel of $g_{\bm l,\bm k}$ is generated by
$$
f_{\bm l,\bm k}= f^{(0)}_{\bm k'}\otimes{\bf p}_{\bm l} - {\bf p}_{\bm k''}\otimes f^{(0)}_{\bm l} \circ s ,
$$
where $s$ is the simple transposition in $\mk S_{l}\subset\mk S_{\bm l}$ exchanging respectively the last two tensorands in $(V^*)^{\otimes l}$.
\item For $\bm l=(l_t,...,l_0,l;m,m_0,...,m_t)$, $\bm k=\bm l+(1_{\alpha},-1_{\alpha-1};-1_{\beta-1},1_{\beta})$, with $0\leq\alpha,\beta\leq t$, there are two intermediate elements $\bm k'=\bm l+(1_{\alpha},-1_{\alpha-1};0_\bullet)$, $\bm k''=\bm l+(0_\bullet;-1_{\beta-1},1_{\beta})$; the domain of $g_{\bm l,\bm k}$ is
$$
{\rm Hom}(I_{\bm k'},I_{\bm k})\otimes_{{\rm End}I_{\bm k'}}{\rm Hom}(I_{\bm l},I_{\bm k'})\oplus {\rm Hom}(I_{\bm k''},I_{\bm k})\otimes_{{\rm End}I_{\bm k''}}{\rm Hom}(I_{\bm l},I_{\bm k''}) \cong \CC[\mk S_{\bm l+(1_\alpha;1_{\beta})}]^{\oplus 2}
$$
as an $(\mk S_{\bm l},\mk S_{\bm k})$-bimodule (the two summands are isomorphic), and the kernel of $g_{\bm l,\bm k}$ is generated by
$$
f_{\bm l,\bm k}= f^{(\alpha)}_{\bm k''}\otimes \bar f^{(\beta)}_{\bm l} - \bar f^{(\beta)}_{\bm k'}\otimes f^{(\alpha)}_{\bm l} \;.
$$
\item For $\bm l=(l_t,...,l_0,l;m,m_0,...,m_t)$, $\bm k=\bm l+(2_{\alpha},-2_{\alpha-1};0_\bullet)$, with $0\leq\alpha\leq t$, there is one intermediate element $\bm k'=\bm l+(1_{\alpha},-1_{\alpha-1};0_\bullet)$; the domain of $g_{\bm l,\bm k}$ is
$$
{\rm Hom}(I_{\bm k'},I_{\bm k})\otimes_{{\rm End}I_{\bm k'}}{\rm Hom}(I_{\bm l},I_{\bm k'}) \cong \CC[\mk S_{\bm l+(2_{\alpha};0_\bullet)}]
$$
as an $(\mk S_{\bm l},\mk S_{\bm k})$-bimodule, and the kernel of $g_{\bm l,\bm k}$ is generated by
$$
f_{\bm l,\bm k}= s' \circ (f^{(\alpha)}_{\bm k'}\otimes f^{(\alpha)}_{\bm l}) - (f^{(\alpha)}_{\bm k'}\otimes f^{(\alpha)}_{l_{\alpha}}) \circ s \;,
$$
where $s\in \mk S_{l_{\alpha-1}}\subset \mk S_{\bm l}$ is the transposition of the last two tensorands in $(V^*_{\alpha-1})^{\otimes l_{\alpha-1}}$ as a tensorand of $I_{\bm l}$ and $s'\in \mk S_{l_{\alpha}+2}\subset\mk S_{\bm k}$ is the transposition of the last two tensorands of $(V^*_{\alpha})^{\otimes l_{\alpha}+2}$ as a tensorand of $I_{\bm k}$.
\item For $\bm l=(l_t,...,l_0,l;m,m_0,...,m_t)$, $\bm k=\bm l+(1_{\alpha},-1_{\alpha-1};0_\bullet)+(1_{\beta},-1_{\beta-1};0_\bullet)$, with $0\leq\alpha<\beta\leq t$, there are two intermediate elements $\bm k'=\bm l+(1_{\alpha},-1_{\alpha-1};0_\bullet)$, $\bm k''=\bm l+(1_{\beta},-1_{\beta-1};0_\bullet)$; the domain of $g_{\bm l,\bm k}$ is
$$
{\rm Hom}(I_{\bm k'},I_{\bm k})\otimes_{{\rm End}I_{\bm k'}}{\rm Hom}(I_{\bm l},I_{\bm k'})\oplus {\rm Hom}(I_{\bm k''},I_{\bm k})\otimes_{{\rm End}I_{\bm k''}}{\rm Hom}(I_{\bm l},I_{\bm k''}) \cong \CC[\mk S_{\bm l+(1_{\beta},1_\alpha;0_\bullet)}]^{\oplus 2}
$$
as an $(\mk S_{\bm l},\mk S_{\bm k})$-bimodule (the two summands are isomorphic), and the kernel of $g_{\bm l,\bm k}$ is generated by an element $f_{\bm l,\bm k}$ determined depending on $\beta-\alpha$ as follows:
\begin{enumerate}
\item if $\beta=\alpha+1$ then 
$$
f_{\bm l,\bm k}= f^{(\alpha)}_{\bm k''}\otimes f^{(\alpha+1)}_{\bm l} - s\circ (f^{(\alpha+1)}_{\bm k'}\otimes f^{(\alpha)}_{\bm l}) \;,
$$
where $s\in \mk S_{l_{\alpha}+1}\subset \mk S_{\bm l+(1_{\alpha+1},1_\alpha;0_\bullet)}$ is the transposition of the last two tensorands in $(V^*_{\alpha})^{\otimes l_{\alpha}}$ as a tensorand of $I_{\bm l+(1_{\alpha+1},1_\alpha;0_\bullet)}$.
\item if $\beta>\alpha+1$ then 
$$
f_{\bm l,\bm k}= f^{(\alpha)}_{\bm k''}\otimes f^{(\beta)}_{\bm l} - f^{(\beta)}_{\bm k'}\otimes f^{(\alpha)}_{\bm l} \;,
$$
where $s\in \mk S_{l_{\alpha}+1}\subset \mk S_{\bm l+(1_{\alpha+1},1_\alpha;0_\bullet)}$ is the transposition of the last two tensorands in $(V^*_{\alpha})^{\otimes l_{\alpha}}$ as a tensorand of $I_{\bm l+(1_{\alpha+1},1_\alpha;0_\bullet)}$.
\end{enumerate}
\end{enumerate}
\end{prop}

\begin{proof}
The proof is a compilation of the proofs of \cite[Lemma 5.16]{Chirvasitu-Penkov-RC} and \cite[Theorem 3.33]{Chirvasitu-Penkov-UTC}.
\end{proof}

\begin{proof}[Proof of Theorem \ref{Theo Univ IT}]
We follow the strategy of \cite[Theorem 3.33]{Chirvasitu-Penkov-UTC}, \cite[Theorem 5.3]{Chirvasitu-Penkov-RC}. The general properties of tensor categories imply that the relations given Proposition \ref{Prop A QuadRel} are satisfied in $\mc D$ for the respective objects and morphisms derived from $X$ and $Y$ instead of $V^*$ and $\bar V$. Now Theorem \ref{Theo IT Koszul A quadr}, together with Proposition \ref{Prop A QuadRel}, implies that the assignment $\Phi(V^*_\alpha)=X_\alpha$, $\Phi(\bar V_\alpha)=Y_\alpha$, for $\alpha=-1,0,...,t$, $\Phi({\bf p})={\bf q}$, $\Phi(V^*/V^*_\alpha\to V^*/V^*_{\alpha+1})=X/X_{\alpha}\to X/X_{\alpha+1}$, $\Phi(\bar V/\bar V_\alpha\to\bar V/\bar V_{\alpha+1})=Y/Y_{\alpha}\to Y/Y_{\alpha+1}$, for $\alpha=-1,...,t-1$, provides a consistent definition of a functor $\Phi:{\bf T}_{\rm fin}\to\mc D$. The uniqueness of this functor, up to tensor natural isomorphism, its left-exactness, and its extension to the Grothendieck category ${\bf T}$ if $\mc D$ is a Grothendieck category, follow from standard arguments as in \cite[\S 8]{Chirvasitu-Penkov-UTC}.
\end{proof}

\begin{coro}\label{Coro ITs is ITVsWs}
Let $0\leq s\leq t$ and let $\TT(V^*_s,I^{s,s},\bar V_s)\subset \TT_t$ be the smallest full tensor Grothendieck subcategory of $\TT_t$ containing $V^*_s,\bar V_s$ and the module $I^{s,s}$, which is also a commutative subalgebra of $I$, defined at the end of section \ref{Sec Module I}. Let ${\bf T}(V^*_s,I^{s,s},\bar V_s)$ be the category, whose objects are $\mk{gl}^M$-modules in $\TT(V^*_s,I^{s,s},\bar V_s)$ which are also free as $I^{s,s}$-modules, and whose morphisms are morphisms of $\mk{gl}^M$-modules as well as of $I^{s,s}$-modules. Then ${\bf T}(V^*_s,I^{s,s},\bar V_s)$ is equivalent to the category ${\bf T}_s$ constructed from an arbitrary diagonalizable pairing between two $\aleph_s$-dimensional complex vector spaces.
\end{coro}

\bibliographystyle{plain}

\begin{thebibliography}{20}

\bibitem[BH1998]{Bergman-Hrushovski} G. Bergman, E. Hrushovski, {\it Linear ultrafilters}, Communications in Algebra {\bf 26} (1998), pp. 4079-4113.

\bibitem[BHO2023]{Polish} O. Bezushchak, W. Ho{\l}ubowski, B. Oliynyk, {\it Ideals of general linear Lie algebras of infinite-dimensional vector spaces}, Proc. Amer. Math. Soc. {\bf 151} (2023), pp. 467-473.

\bibitem[Ch2014]{Chirvasitu-3Mackey} A. Chirvasitu, {\it Three Results on Representations of Mackey Lie Algebras}. In: Mason, G., Penkov, I., Wolf, J. (eds) Developments and Retrospectives in Lie Theory. Developments in Mathematics, vol. 38, Springer 2014.

\bibitem[ChP2017]{Chirvasitu-Penkov-RC} A. Chirvasitu, I. Penkov, {\it Representation categories of Mackey Lie algebras as universal monoidal categories}, Pure and Applied Mathematics Quarterly {\bf 13} (2017), pp. 77-121.

\bibitem[ChP2019]{Chirvasitu-Penkov-OTC} A. Chirvasitu, I. Penkov, {\it Ordered tensor categories and representations of the Mackey Lie algebra of infinite matrices}, Algebras and Representation Theory {\bf 22} (2019), pp. 249-279.

\bibitem[ChP2021]{Chirvasitu-Penkov-UTC} A. Chirvasitu, I. Penkov, {\it Universal tensor categories generated by dual pairs}, Applied Categorical Structures {\bf 29} (2021), pp. 915-950.

\bibitem[PS2014]{Penkov-Serganova-Mackey} I. Penkov, V. Serganova, {\it Tensor representations of Mackey Lie algebras and their dense subalgebras}, in: Developments and Retrospectives in Lie Theory: Algebraic Methods, Developments in Mathematics, vol. 38, Springer Verlag 2014, pp. 291-330.

\bibitem[PSt2011]{Penkov-Styrkas} I. Penkov, K. Styrkas, {\it Tensor representations of classical locally finite Lie algebras}, in Developments and Trends in Infinite-Dimensional Lie Theory, Progress in Mathematics {\bf 288}, Springer Verlag 2011, pp. 127-150.

\bibitem[Ph2022]{Pham} T. Pham, {\it Ext-groups between simple objects of the category $\TT_{\aleph_1}$ of tensor representations of a Mackey Lie algebra of $\aleph_1$-dimensional vector spaces}, Bachelor's thesis, Jacobs University, Bremen, 2022.

\bibitem[SS2015]{Sam-Snowden} S. Sam, A. Snowden, {\it Stability patterns in representation theory}, Forum Math. Sigma {\bf 3} (2015), paper no. e11, 108 pages.

\bibitem[S1972]{Stewart} I. Stewart, {\it The Lie algebra of endomorphisms of an infinite-dimensional vector space}, Compositio Mathematica {\bf 25} (1972), pp. 79-86.

\end{thebibliography}

\small{

}

\noindent Ivan Penkov\\
Constructor University, 28759 Bremen, Germany\\
E-mail address: ipenkov@constructor.university\\

\noindent Valdemar Tsanov\\
Institute of Mathematics and Informatics, Bulgarian Academy of Sciences,\\
Bulgaria, Sofia 1113, Acad. G. Bonchev Str., Bl. 8\\
and Constructor University, 28759 Bremen, Germany.\\
E-mail address: valdemar.tsanov@math.bas.bg\\

\end{document}